\documentclass[10pt,a4paper]{article}

\usepackage{amsmath, amssymb, tikz-cd}
\usepackage{amscd}
\usepackage{amssymb}
\usepackage{amsthm}
\usepackage{xspace}
\usepackage[all,tips]{xy}
\usepackage{syntonly}
\usepackage{hyperref}
\usepackage{arydshln}
\usepackage{tikz-cd}
\usepackage{todo}
\usepackage{enumerate}
\usepackage{a4wide}
\usepackage{comment}
\newtheorem{thm}{Theorem}[section]
\newtheorem{cor}[thm]{Corollary}
\newtheorem{prop}[thm]{Proposition}
\newtheorem{lemma}[thm]{Lemma}
\newtheorem{ques}{Question}

\newtheorem*{cor*}{Corollary}

\newtheorem{thmx}{Theorem}

\theoremstyle{definition}

\newtheorem{defn}[thm]{Definition}

\newtheorem{ex}[thm]{Example}

\DeclareMathOperator{\Aut}{Aut} \DeclareMathOperator{\Out}{Out}
 
\DeclareMathOperator{\GL}{GL}

\DeclareMathOperator{\Inn}{Inn}
\DeclareMathOperator{\Fitt}{Fitt}\DeclareMathOperator{\Hom}{Hom}

\DeclareMathOperator{\Fr}{Fr}
\newcommand{\Z}{\mathbb{Z}}
\newcommand{\Q}{\mathbb{Q}}
\newcommand{\C}{\mathbb{C}}

\newcommand\restr[2]{{
		\left.\kern-\nulldelimiterspace 
		#1 
		\vphantom{\big|} 
		\right|_{#2} 
}}

\begin{document} 
\title{Automorphism groups of solvable groups\\ of finite abelian ranks}
\author{Jonas Der\'e\thanks{KU Leuven Kulak, Kortrijk, Belgium. Email: \tt{jonas.dere@kuleuven.be}. The author is supported by a start-up grant of KU Leuven (grant number 3E210985).} \ and Mark Pengitore\thanks{Shanghaitech University, Shanghai, China. E-mail: \tt{mpengito@gmail.com}. The author was supported by the National Science Center Grant Maestro-13 UMO- 2021/42/A/ST1/00306}} 
\maketitle

\abstract{This paper gives a new explicit construction of the $\Q$-algebraic hull for virtually solvable groups $\Gamma$ of finite abelian ranks, taking into account the spectrum $S$ of the group $\Gamma$. This class of groups, which are exactly the virtually solvable groups that are linear over $\Q$, contains the virtually polycyclic groups as well as the solvable Baumslag-Solitar groups. As an application, we make a detailed study of the structure of $\Aut(\Gamma)$ in the finitely generated case and show that a number of natural subgroups are $S$-arithmetic under the condition that $\Fitt(\Gamma)$ is $S$-arithmetic. We then proceed by demonstrating that $\Out(\Gamma)$ has a $S$-arithmetic image in the group of algebraic outer automorphisms of the $\Q$-algebraic hull. We finish by discussing further applications of the $\Q$-algebraic hull towards an open conjecture by Nekrashevych and Pete and topological fixed point theory.}

\tableofcontents

\section{Introduction}
\subsection{Background: $\Q$-algebraic hulls} Given a virtually polycyclic group $\Gamma$, a natural way to study the automorphism group $\Aut(\Gamma)$ of $\Gamma$ is through the $\mathbb{Q}$-algebraic hull of $\Gamma$, which we denote as $\textbf{H}_\Gamma$. To frame our discussion, we give the precise definition of the $\mathbb{Q}$-algebraic hull which most authors attribute to Mostow  \cite{mostow}. In the context of torsion-free, finitely generated nilpotent groups, its construction is due to Mal'tsev in \cite{malcev} and is known as the rational Mal'tsev completion. See \cite[4.39 Definition]{raghunathan} or  \cite[Appendix A]{baues} for a modern formulation for virtually polycyclic groups.
\begin{defn}
	\label{defpolycyclic}
Let $\Gamma$ be a virtually polycyclic group. A $\mathbb{Q}$-algebraic hull for $\Gamma$ is a pair $(\textbf{H}_\Gamma, \iota )$ where $\textbf{H}_\Gamma$ is a $\mathbb{Q}$-defined linear algebraic group and $\iota  \colon \Gamma \to \textbf{H}_\Gamma(\mathbb{Q})$ is an injective morphism such that
\begin{enumerate}[(i)]
\item $\iota (\Gamma)$ is Zariski dense in $\textbf{H}_\Gamma$;
\item the centralizer of the unipotent radical $\textbf{U}(\textbf{H}_\Gamma)$ is contained in $\textbf{U}(\textbf{H}_\Gamma)$;
\item $\dim_{\mathbb{Q}} (\textbf{U}(\textbf{H}_\Gamma)) = h(\Gamma)$.
\end{enumerate}
Moreover, under these assumptions, the intersection $\iota (\Gamma) \cap \textbf{H}_\Gamma(\mathbb{Z})$ is of finite index in $i(\Gamma)$.
\end{defn}
\noindent The constant $h(\Gamma)$ is known as the \emph{Hirsch length} of $\Gamma$, namely the number of infinite cyclic factors in a composition series of $\Gamma$. It is straightforward to see that this value is independent of the choice of composition series of $\Gamma$.

By the paragraph after \cite[Corollary A.9]{baues}, we see that a virtually polycyclic group admits a $\Q$-algebraic hull if and only if the only finite normal subgroup is the trivial one. Using Definition \ref{defpolycyclic}, \cite[Corollary A.3]{baues} implies that it is unique up to algebraic isomorphism. Moreover, using \cite[Proposition A.3]{baues}, it follows that every monomorphism on $\Gamma$ extends to an algebraic automorphism of $\textbf{H}_\Gamma$. As an immediate consequence, we have that $\Aut(\Gamma)$ is linear over $\mathbb{Z}$ for groups having a $\Q$-algebraic hull. More generally, the $\mathbb{Q}$-algebraic hull allows us to bring powerful techniques from linear algebraic groups to study polycyclic groups, which has been useful to prove many deep and interesting properties about polycyclic groups and their (outer) automorphism groups. 

For instance, Baues and Gr\"{u}newald \cite[Theorem 1.1]{baues_grunewald} demonstrated that $\Out(\Gamma)$ is an arithmetic group, namely by constructing a $\mathbb{Q}$-defined linear algebraic group $\textbf{G} \leq \GL(n,\mathbb{C})$ such that $\Out(\Gamma)$ is commensurable with the group of integer points $\textbf{G}(\mathbb{Z})$. This is a remarkable result because previously this was only known in the context of finitely generated nilpotent groups due to Segal \cite{segal_outer}. Indeed, this demonstrates that the outer automorphism group of a polycyclic group satisfies strong arithmetic and algebraic properties and has many nice finiteness properties. A simple model of this result is in the case $\Gamma = \mathbb{Z}^n$ for $n \in \mathbb{N}$, where we have
$$
\Aut(\mathbb{Z}^n) = \Out(\mathbb{Z}^n) = \GL(n,\mathbb{Z}).
$$

More generally, $\Aut(\Gamma)$ was shown to be arithmetic for every finitely generated virtually nilpotent group $\Gamma$, see \cite{auslander_baumslag, baumslag_nilpotent, baumslag_lecture_nilpotent,  segal}. However, this result does not generalize to a general polycyclic group. Indeed, Baues and Gr\"{u}newald \cite{baues_grunewald} construct a polycyclic group $\Gamma$ of Hirsch length $4$ such that $\Aut(\Gamma)$ is not arithmetic. On the other hand, they demonstrate using the $\Q$-algebraic hull construction that every virtually polycyclic group $\Gamma$ contains a finite index subgroup $\Gamma^\prime$ such that $\Aut(\Gamma^\prime)$ is arithmetic.

A natural question is how far can we generalize the construction of the $\mathbb{Q}$-algebraic hull of a virtually polycyclic group within the collection of finitely generated solvable groups in order to recover and generalize the results of the above discussion. It is well known that not all finitely generated solvable groups are linear or even have a solvable word problem, so an extension of these results to all finitely generated solvable groups is not possible. Even within the collection of linear groups, there exist solvable groups such as $\mathbb{Z} \wr \mathbb{Z}$ which are linear over $\mathbb{R}$ but not linear over $\mathbb{Q}$. In this paper, we construct a $\mathbb{Q}$-linear algebraic hull for finitely generated solvable groups that are $\mathbb{Q}$-linear but not necessarily linear over $\mathbb{Z}$. Such groups are part of the collection of finitely generated solvable groups of finite abelian ranks, which include solvable Baumslag-Solitar groups and also Leary-Minasyan groups, see \cite{lm21-1}. In fact, \cite[10.2.5]{lennox_robinson} and \cite[Theorem 2]{kopytov} together imply that every torsion-free finitely generated solvable group of finite cohomological dimension is included in this collection. 

An abelian group $A$ has finite abelian ranks if $\dim_{\mathbb{Q}}(\mathbb{Q} \otimes_{\mathbb{Z}} A) < \infty$ and for all primes $p$, we have $\dim_{\mathbb{F}_p}(\Hom(\mathbb{F}_p,A)) < \infty$. A solvable group $\Gamma$ is of \emph{finite abelian ranks} if it admits a decreasing normal series $\Gamma = \Gamma_1 \supset \cdots \supset \Gamma_m = \{1\}$ such that $\Gamma_{i} / \Gamma_{i+1}$ is an abelian group of finite abelian ranks.  The \emph{Hirsch length} of a solvable group of finite abelian ranks is then given by
$$
h(\Gamma) = \sum_{i=1}^{m-1} \dim_\Q(\mathbb{Q} \otimes_{\mathbb{Z}} (\Gamma_i / \Gamma_{i+1})).
$$
We see that $h(\Gamma)$ is well defined independent of the series $\{\Gamma_i\}$, and moreover, we observe that this notion of Hirsch length agrees with the usual notion when $\Gamma$ is a polycyclic group. For this collection of solvable groups, the existence of a $\mathbb{Q}$-algebraic hull was first given by  Arapura and Nori \cite{arapura_nori} in order to study K\"ahler manifolds. Their method is a categorical argument where they use the pro-algebraic hull associated to $\Gamma$ equipped with the discrete topology and a topological field. In our case, our field is $\mathbb{Q}$ which is equipped with the Euclidean metric and the discrete topology on $\Gamma$ comes from the word metric.

\subsection{Main results}

For our first result, we reprove the existence of the $\mathbb{Q}$-algebraic hull for a finitely generated virtually solvable group $\Gamma$ of finite abelian ranks. Our proof gives a direct construction of  the $\mathbb{Q}$-algebraic hull $\textbf{H}_{\Gamma}$ and a morphism $\iota \colon \Gamma \to \textbf{H}_{\Gamma}(\mathbb{Q})$, allowing us to realize the image in the smaller group $\textbf{H}_{\Gamma}(\Z[\frac{1}{S}])$, where $S$ is the spectrum of $\Gamma$. As a reminder to the reader, the spectrum is the set of primes $p$ for which there exists a pair of subgroups $N \trianglelefteq M$ of $\Gamma$ such that $M/N$ is a $p$-quasicyclic group; see Definition \ref{def:quasi} for more details. 
\begin{defn}\label{def:algebraichull}
Let $\Gamma$ be a virtually solvable group of finite abelian ranks with spectrum $S$. A $\mathbb{Q}$-algebraic hull for $\Gamma$ is a pair $(\textbf{H}_\Gamma, \iota )$ where $\textbf{H}_\Gamma$ is a $\mathbb{Q}$-defined linear algebraic group and  $\iota  \colon \Gamma \to \textbf{H}_\Gamma$ is an injective morphism such that
\begin{enumerate}[(i)]
\item $\iota (\Gamma) \leq \textbf{H}_\Gamma(\Q)$ is Zariski dense in $\textbf{H}_\Gamma$;
\item the centralizer of the unipotent radical $\textbf{U}(\textbf{H}_\Gamma)$ is contained in $\textbf{U}(\textbf{H}_\Gamma)$;
\item $\dim_\Q (\textbf{U}(\textbf{H}_\Gamma)) = h(\Gamma)$.
\end{enumerate}
If $\Gamma$ is moreover minimax, we additionally assume that  $\iota (\Gamma) \leq \textbf{H}_\Gamma(\Z[\frac{1}{S}])$. 
\end{defn}
\noindent When $S$ is empty, or equivalently when $\Gamma$ is a polycyclic group, the morphism is of the form $\iota \colon \Gamma \to \textbf{H}_\Gamma(\mathbb{Z})$. In general, when $\Gamma$ has a $\Q$-algebraic hull, it will admit a faithful representation into $\textbf{H}_\Gamma(\Q)$ and thus must be linear. We also study the properties of the $\Q$-algebraic hull.
\begin{thmx}\label{algebraic_hull_thm}
	A virtually torsion-free solvable group of finite abelian ranks $\Gamma$ has a $\mathbb{Q}$-algebraic hull if and only if $\Gamma$ has a trivial maximal normal torsion subgroup. Moreover, the $\mathbb{Q}$-algebraic hull is unique up to algebraic isomorphism, and every monomorphism $\psi$ of $\Gamma$ uniquely extends to an algebraic automorphism $\Psi$ of $\textbf{H}_\Gamma$.
\end{thmx}

Our main application of Theorem \ref{algebraic_hull_thm} extends \cite[Theorem 1.5]{baues_grunewald} to when $\Gamma$ is a finitely generated virtually torsion-free solvable group of finite abelian ranks satisfying an arithmetic condition on the Fitting subgroup $\Fitt(\Gamma)$, which is the maximal nilpotent normal subgroup of $\Gamma$. 
As a reminder, a subgroup $\Gamma$ of a $\mathbb{Q}$-defined linear algebraic group $\textbf{G}$ is $S$-arithmetic if $\Gamma$ is commensurable with $\textbf{G}(\mathbb{Z}[\frac{1}{S}])$. A group $\Gamma$ is $S$-arithmetic if it is abstractly isomorphic to a $S$-arithmetic subgroup of a $\mathbb{Q}$-defined linear algebraic group. If this linear algebraic group is moreover unipotent, we will call $\Gamma$ unipotently $S$-arithmetic.

Let $\Gamma$ be a finitely generated virtually torsion-free  solvable group of finite abelian ranks which has a trivial maximal normal torsion subgroup, and let $\Fitt(\Gamma)$ be the Fitting subgroup of $\Gamma$. Given that $\Fitt(\Gamma)$ is characteristic in $\Gamma$, we may define
$$
A_{\Gamma | \Fitt(\Gamma)} = \{ \varphi \in \Aut(\Gamma) \: | \: \varphi|_{\Gamma/\Fitt(\Gamma)} = \text{id}_{\Gamma/\Fitt(\Gamma)} \}.
$$
For the next theorem, we denote $\Aut(\textbf{H}_\Gamma)$ as the group of algebraic automorphisms of $\textbf{H}_{\Gamma}$. We now give a structural result for the automorphism group of a finitely generated virtually solvable group of finite abelian ranks when $\Fitt(\Gamma)$ is unipotently $S$-arithmetic.

\begin{thmx}\label{theorem1.4_baues_grunewald}
Let $\Gamma$ be a finitely generated solvable group of finite abelian ranks with spectrum $S$ and trivial maximal normal torsion subgroup such that $\Fitt(\Gamma)$ is unipotently $S$-arithmetic. The group $A_{\Gamma  |  \Fitt(\Gamma)}$ is a normal subgroup of $\Aut(\Gamma)$ that is $S$-arithmetic in its Zariski closure in $\Aut(\textbf{H}_\Gamma)$. There exists a nilpotent group $B \leq \Gamma$ such that $A_{\Gamma |  \Fitt(\Gamma)} \cdot \Inn_B^\Gamma$ has finite index in $\Aut(\Gamma)$.
\end{thmx}

As a consequence, all of these groups have a finite index subgroup for which the automorphism group is $S$-arithmetic:

\begin{cor*}
Under the assumption of the previous theorem, there exists a finite index subgroup $\Gamma_0 \leq \Gamma$ such that $\Aut(\Gamma_0)$ is $S$-arithmetic.
\end{cor*}

For the last main theorem, we note that $\Inn_{\textbf{H}_\Gamma}$ is a $\mathbb{Q}$-closed subgroup of the group of the algebraic automorphisms $\Aut(\textbf{H}_\Gamma)$ of  $\textbf{H}_\Gamma$. We then call the quotient $\Out_a(\textbf{H}_\Gamma) = \Aut(\textbf{H}_\Gamma) / \Inn_{\textbf{H}_\Gamma}$ the algebraic outer automorphism group of $\textbf{H}_\Gamma$, which is again a $\mathbb{Q}$-defined linear algebraic group. We obtain a group morphism $\pi_\Gamma \colon \Out(\Gamma) \to \Out_a(\textbf{H}_\Gamma)$ by restricting the quotient morphism $\Aut(\textbf{H}_\Gamma) \to \Out_a(\textbf{H}_\Gamma)$ to the subgroup $\Aut(\Gamma) \leq \Aut(\textbf{H}_\Gamma)$. 
\begin{thmx}\label{out(G)_arithmetic}
	Let $\Gamma$ be a finitely generated virtually solvable group of finite abelian ranks without torsion normal subgroups with spectrum $S$, and suppose that $\Fitt(\Gamma)$ is unipotently $S$-arithmetic. Then
	\begin{enumerate}[(1)]
		\item the image $\pi_\Gamma(\Out(\Gamma))$ is a $S$-arithmetic group in its Zariski closure;
		\item the kernel $\ker(\pi_\Gamma)$ has finite Hirsch length, is virtually abelian, and is centralized by a finite index subgroup of $\Out(\Gamma)$;
		\item if $\Gamma$ is moreover virtually nilpotent, then the kernel of $\pi_\Gamma$ is finite.
	\end{enumerate}
\end{thmx}
\noindent Is is important to note that the techniques used in \cite[Section 10]{baues_grunewald}, which were important in demonstrating that $\Out(\Gamma)$ is an arithmetic lattice when $\Gamma$ is polycyclic, do not generalize to the $S$-arithmetic lattice setting. The authors are unaware of an construction analogous to \cite[Proposition 10.8]{baues_grunewald} which could be used to go from Theorem \ref{finite_index_aut(gamma)} to a $S$-arithmetic version of \cite[Theorem 1.1]{baues_grunewald}. Subsequently, Theorem \ref{out(G)_arithmetic} is the best one can hope for in the setting of finitely generated virtually solvable groups of finite abelian ranks with current theory.

In the final section several other applications of the $\Q$-algebraic hull are given, including a positive answer to a conjecture of Nekrashevych and Pete about the existence of strongly scale-invariant monomorphisms for the class of virtually solvable groups of finite abelian ranks. In a similar vein, we demonstrate that in this class of groups, the only endomorphisms with a well-defined Reidemeister-zeta function are the ones that are eventually virtually nilpotent. In the recent work \cite{dv26-1} the construction of a $\Q$-algebraic hull was also used to study the automorphisms of direct products, including their Reidemeister spectrum and whether they are cohopfian.

\subsection{Open questions}
In our main results, we always assume that $\Fitt(\Gamma)$ is $S$-arithmetic in a unipotent linear algebraic group. It seems likely that this is in fact a necessary condition, also motivated by studying direct products of Baumslag Solitar groups in Theorem \ref{thm_nonarith}.
\begin{ques}
Do there exist examples of finitely generated virtually FAR-groups with $A_{\Gamma |  \Fitt(\Gamma)} $ $S$-arithmetic but $\Fitt(\Gamma)$ not unipotently $S$-arithmetic?
\end{ques}

Given a finitely generated virtually solvable group $\Gamma$ of finite abelian ranks with spectrum $S$ such that $\Fitt(\Gamma)$ is unipotently $S$-arithmetic, it remains open when $\Out(\Gamma)$ is $S$-arithmetic.
\begin{ques}\label{question_1}
Under what conditions on $\Gamma$ is the group $\Out(\Gamma)$ an $S$-arithmetic group?
\end{ques}
\noindent Theorem \ref{calculation_baumslag_solitar} shows that even for Baumslag Solitar groups $\Gamma$, $\Out(\Gamma)$ could be commensurable with the group of $\mathbb{Z}[\frac{1}{S}]$-points of some $\mathbb{Q}$-defined linear algebraic group that is not naturally associated with the group $\Gamma$ or its algebraic hull. As mentioned before, we are unaware of a general construction of such a linear algebraic group for general finitely generated solvable groups of finite abelian rank that are not virtually polycyclic. 

A natural example of a $\mathbb{R}$-linear group that is not $\mathbb{Q}$-linear is given by $\mathbb{Z} \wr \mathbb{Z}$. This group is isomorphic to the matrix group given by
$$
\mathbb{Z} \wr \mathbb{Z} \cong \left\{ \begin{bmatrix}
 \pi^m & f \\
0 & 1 
\end{bmatrix}
:
m \in \mathbb{Z}, f \in \mathbb{Z}[\pi, \pi^{-1}]
\right\}.
$$
It is clear that $\mathbb{Z} \wr \mathbb{Z}$ is isomorphic to an index $2$-subgroup of the group of $\mathbb{Z}[X,X^{-1}]$-points of the $\mathbb{R}$-defined linear algebraic group $\mathbb{C} \rtimes \mathbb{C}^*$ where $\mathbb{C}$ is equipped with its additive abelian structure and $\mathbb{C}^*$ with its multiplicative structure.  Hence, it is natural to ask whether there exists a $\mathbb{R}$-algebraic hull for $\mathbb{R}$-linear groups that are not $\mathbb{Q}$-linear. Therefore, we finish with the following question.
\begin{ques} Suppose that $\Gamma$ is a finitely generated solvable group which admits a faithful finite-dimensional representation over a characteristic $0$ field but no faithful finite-dimensional representation over $\mathbb{Q}$. Can one construct an analogue of the $\mathbb{Q}$-linear algebraic hull for $\Gamma$?
\end{ques}
If such a construction exists, one might recover and generalize the results of this article to the more general class of $\mathbb{R}$-linear groups that are not $\mathbb{Q}$-linear.

\subsection{Plan of the paper}
Section \ref{background_section} gives background material needed for this article with subsections devoted to solvable group of finite abelian ranks and the basics of linear algebraic groups. Section \ref{algebraic_hull_section} gives an explicit construction of the $\Q$-algebraic hull for a solvable group of finite abelian ranks and details its properties proving Theorem \ref{algebraic_hull_thm}. Section \ref{sec:thick} introduces the notion of a thickening of a finitely generated virtually torsion-free solvable group of finite abelian ranks which allows one to find virtually nilpotent supplements to the Fitting subgroup of a solvable group of finite abelian ranks. Section \ref{section_s_arithmetic_subgroups} discusses how to construct $S$-arithmetic subgroups of $\Aut(\Gamma)$ using the $\Q$-algebraic hull of $\Gamma$. Section \ref{sec_theorem1.4_proof} gives the proof of Theorem \ref{theorem1.4_baues_grunewald}. Finally, Section \ref{section_futher_applications} gives more applications of the $\Q$-algebraic hull to questions about endomorphisms of solvable groups of finite abelian ranks, and shows that the $S$-arithmeticity condition on $\Fitt(\Gamma)$ cannot be dropped by working out examples using $BS(1,n)$.

\subsection{AI use disclosure}
AI was only used for proofreading and literature search. The authors take full responsibility for all of the content of this paper.

\section{Background}\label{background_section}

\subsection{Solvable groups of finite rank}
We give the necessary results about solvable groups needed for this paper; a more detailed discussion of the following topics can be found in \cite{lennox_robinson,segal}. In order to fix some notation, given a group $\Gamma$ with a subset $X \subset \Gamma$, we denote the subgroup generated by $X$ as $\left<X\right> \subset \Gamma$. Given two subgroups $G$ and $H$ in $\Gamma$, we denote the normalizer of $G$ in $H$ as $N_H(G)$ and the centralizer of $G$ in $H$ as $C_H(G)$.

Let $\Gamma$ be a group. The $0$-th term of the \emph{derived series} of $\Gamma$ is given as $\Gamma^{(0)}= \Gamma$, and for $i>0$, the $i$-th term of the derived series is given by $\Gamma^{(i)} = [\Gamma^{(i-1)}, \Gamma^{(i-1)}]$. The first term of the \emph{lower central series} is defined as $\gamma_1(\Gamma) = \Gamma$, and for $i>1$, the $i$-th term of the lower central series is given by $\gamma_i(\Gamma) = [\Gamma, \gamma_{i-1}(\Gamma)]$.

\begin{defn}
Let $\Gamma$ be a group. We say that $\Gamma$ is a \emph{solvable group of derived length $s$} if $s$ is the smallest natural number such that $\Gamma^{(s+1)} = \{1\}$, and when the derived length is unspecified, we simply refer to $\Gamma$ as a solvable group. 

We say that $\Gamma$ is a \emph{nilpotent group of step length $c$} if $c$ is the smallest natural number such that $\gamma_{c+1}(\Gamma) = \{1\}$. As before, if the step length is unspecified, we then say that $\Gamma$ is a nilpotent group.
\end{defn}


For a solvable group $\Gamma$, the \emph{Fitting subgroup} of $\Gamma$ is the subgroup generated by all normal nilpotent subgroups of $\Gamma$. We denote the Fitting subgroup of $\Gamma$ as
$$
\Fitt(\Gamma) = \left< N \: | \: N \trianglelefteq \Gamma, \: N \text{ nilpotent} \right>.
$$
The subgroup $\Fitt(\Gamma)$ is always non-trivial since it contains the last nontrivial step of the derived series.

Because the join of normal torsion subgroups of a finitely generated virtually solvable group $\Gamma$ is a torsion group, there exists a unique largest normal torsion subgroup which we denote as $\tau(\Gamma)$. We refer to any  finitely  generated virtually solvable group where $\tau(\Gamma) = 1$ as a \emph{WTN}-group. When $\Gamma$ is a nilpotent group, the subgroup $\tau(\Gamma)$ contains all of the elements of finite order by \cite[Corollary 10, page 13]{segal}, and thus, $\Gamma/\tau(\Gamma)$ is torsion-free. When $\Gamma$ is a finitely generated nilpotent group, then $\tau(\Gamma)$ is finite.

Let $A$ be an abelian group. The \emph{torsion-free rank} of $A$ is defined as
$$
r_0(A) = \dim_{\mathbb{Q}}(A \otimes_{\mathbb{Z}} \mathbb{Q}).
$$
For any prime $p$, the $p$-\emph{rank} of the abelian group $A$ is given as
$$
r_p(A) = \dim_{\mathbb{F}_p}(\Hom(\mathbb{F}_p, A)).
$$
Letting $P = \{a \in A \: | \: pa = 0\}$, we have that $\Hom(\mathbb{F}_p, A) = \Hom(\mathbb{F}_p, P)$. Hence,
$$
r_p(A) = \dim_{\mathbb{F}_p}(P).
$$
A solvable group $\Gamma$ has \emph{finite abelian ranks}, or is a \emph{solvable FAR-group}, if there is a series $\{\Gamma_i\}_{i=0}^k$ such that each factor $\Gamma_{i+1} / \Gamma_i$ is abelian and $r_p(\Gamma_{i+1}/\Gamma_i)$ is finite for all $p=0$ and $p$ prime. In this context, we denote the Hirsch length of $\Gamma$ as
$$
h(\Gamma) = \sum_{i=1}^k r_0(\Gamma_{i} / \Gamma_{i-1}).
$$
which by a standard argument does not depend on the choice of series $\Gamma_i$. Note that if $h(\Gamma) = 0$ if and only if $\Gamma$ is a torsion group. For every normal subgroup $N \triangleleft \Gamma$, we have that $h(\Gamma) = h(N) + h\left(\Gamma / N \right)$.

A group $\Gamma$ satisfies the \emph{maximal condition}, also known as \textbf{Max}, \emph{Noetherian}, or \emph{slender}, if every strictly ascending chain of subgroups
$$
\Gamma_1 \lneq \Gamma_2 \lneq \Gamma_3 \lneq \cdots
$$
ends after finitely many terms. It is straightforward that a group satisfies \textbf{Max} if and only if every subgroup is finitely generated. Any finitely generated nilpotent group $N$ is well known to satisfy $\textbf{Max}$. A finitely generated solvable group satisfies \textbf{Max} if and only if it is polycyclic. We say that $\Gamma$ satisfies the \emph{minimal condition}, also known as \textbf{Min}, if every strictly descending chain of subgroups
$$
\Gamma_1 \gneq \Gamma_2 \gneq \Gamma_3 \gneq \cdots
$$
ends after finitely many terms. Groups that satisfy \textbf{Min} are also known as \emph{Artinian}.

\begin{defn}
We say a solvable group $\Gamma$ is \emph{minimax} if it has a normal series $\{\Gamma_i\}_{i=1}^k$ of finite length such that each quotient $\Gamma_i / \Gamma_{i-1}$ satisfies \textbf{Max} or \textbf{Min}.
\end{defn}
\noindent When given a finitely generated solvable minimax group $\Gamma$, we have by \cite[5.2.2]{lennox_robinson} that $\Fitt(\Gamma)$ is nilpotent and $\Gamma/ \Fitt(\Gamma)$ is virtually abelian. Therefore, we have that $\Gamma$ is virtually nilpotent-by-abelian.

In the context of finitely generated solvable groups, we have by  \cite[Corollary 10.5.3]{lennox_robinson} that any finitely generated solvable group with finite abelian ranks is a minimax group. Therefore, we have the following lemma.
\begin{lemma}\label{rank_definitions_finitely_generated}
Let $\Gamma$ be a finitely generated solvable group. $\Gamma$ has finite abelian ranks if and only if $\Gamma$ is minimax. 
\end{lemma}

Whenever we are given a finitely generated solvable FAR-group, it must be minimax by the previous lemma. If $\Gamma$ is a solvable minimax group, then any subgroup or quotient of $\Gamma$ is also a solvable minimax group. In particular, $\Gamma$ does not contain the rationals as a subgroup.

\begin{defn}
	\label{def:quasi}
Let $p$ be a prime. The $p$-\emph{quasicyclic group} or Pr\"{u}fer-$p$ group, denoted $\mathbb{Z}(p^\infty)$, is given by the direct limit $\mathbb{Z}(p^\infty) \cong \underset{\longrightarrow}{\lim} \: \mathbb{Z} / p^n\mathbb{Z}$ where maps are given by multiplication by $p$. We will say a group $\Gamma$ is quasicyclic if it is a $p$-quasicyclic group for some prime $p$.
\end{defn}

Using the notion of $p$-quasicyclic groups, we introduce the spectrum of a finitely generated solvable group $\Gamma$ which is a subset of integral primes that encode linearity information of $\Gamma$ when $\Gamma$ is $\mathbb{Q}$-linear.
\begin{defn}
Let $\Gamma$ be a finitely generated solvable FAR-group. The \emph{spectrum} of $\Gamma$ is the set of primes $p$ such that there exist subgroups $N \triangleleft M$ in $\Gamma$ such that $M/N$ is a $p$-quasicyclic group.
\end{defn}

A group $\Gamma$ is \emph{residually finite} if for all non-trivial elements $g \in \Gamma$, there exists a surjective morphism $\varphi \colon \Gamma \to Q$ to a finite group $Q$ such that $\varphi(g) \neq 1$. The \emph{finite residual of $\Gamma$} is defined as 
$$
\Fr(\Gamma) = \bigcap_{N \triangleleft \Gamma, [\Gamma:N] < \infty} N.
$$
The following lemma gives a characterization of when $\Fr(\Gamma) = 1$ for finitely generated solvable FAR-groups $\Gamma$, i.e., when $\Gamma$ is residually finite.

\begin{lemma}
	\label{lem:equivalencesRF}
Let $\Gamma$ be a virtually solvable minimax group with spectrum $S$. The following are equivalent:
\begin{enumerate}[(i)]
\item $\Gamma$ is residually finite;
\item $\tau(\Gamma)$ is finite.
\end{enumerate}
Moreover, if $\Gamma$ satisfies one of these properties, it is virtually torsion-free and linear over $\Z[\frac{1}{S}]$.
\end{lemma}
\begin{proof}

From \cite[Theorem 5.3.1.]{lennox_robinson}, we see that $\Fr(\Gamma)$ is the unique maximal radicable subgroup of $\Gamma$ and that $\Fr(\Gamma)$ is abelian. As the group of rationals $\Q$ is not minimax, we conclude that $\Fr(\Gamma)$ is a direct sum of quasicyclic groups. Given that $\tau(\Gamma)$ contains all normal torsion subgroups, we have $\Fr(\Gamma) \leq \tau(\Gamma)$. Moreover, as $\tau(\Gamma)$ is a Chernikov group by \cite[5.2.1]{lennox_robinson} and all radicable subgroups are contained in $\Fr(\Gamma)$, it holds that
$\tau(\Gamma)$ is a finite extension of $\Fr(\Gamma)$. 
Since $\Gamma$ is residually finite if and only if $\Fr(\Gamma)$ is trivial, it indeed holds that $\Gamma$ is residually finite if and only if $\tau(\Gamma)$ is finite. In particular, this discussion implies that $(i)$ and $(ii)$ are equivalent. 

Note that the fact that $\Gamma$ is virtually torsion-free also follows from the discussion above. The final statement follows directly from \cite[5.1.8]{lennox_robinson}.
\end{proof}

Note that a classical result by Mal'tsev shows that finitely generated linear groups are always residually finite, giving a converse to the last statement in this special case. As any group of finite abelian ranks is always solvable, we will often suppress the term solvable when talking about groups of finite abelian ranks and just write FAR-group.

\subsection{Linear algebraic groups}\label{background_linear_algebraic_groups}
For more details for the discussion below, we refer to \cite{borel, springer}. 

Let $R$ be a subring of $\mathbb{Q}$. For a set of primes $S$ in $R$, we denote $R[\frac{1}{S}]$ as the ring of fractions $\frac{a}{b}$ where $a \in R$ and $b \in \langle S \rangle - \{0\}$ where $\left<S \right>$ is the multiplicative set generated by $S$. We denote the subgroup of invertible $(n \times n)$-matrices with coefficients in a ring $R$ as $\GL(n,R)$.  

A $\mathbb{Q}$-defined \emph{linear algebraic group} $\textbf{G} \leq \GL(n,\mathbb{C})$ is an algebraic variety 
$$
\textbf{G} = V(F_1, \ldots, F_m) = \left\{ \left. X \in \mathbb{C}^{n^2} \: \right| \: F_i(X) = 0, i = 1, \ldots, m\right\},
$$ 
with $F_i$ rational polynomials that is also a group with matrix multiplication. Note that this set $\textbf{G}(\Q)$ of $\Q$-rational points isa Zariski dense subset in $\textbf{G}$. In particular, the maps defining the group structure $\mu \colon \textbf{G} \times \textbf{G} \to \textbf{G}$ with $\mu(x,y) = xy$ and $\iota \colon x \to x^{-1}$ are algebraic. 
When given a $\mathbb{Q}$-defined linear algebraic group $\textbf{G}$, we denote the connected component of the identity as $\textbf{G}^\circ$. Given a subring of $\mathbb{C}$ and a $\mathbb{Q}$-defined linear algebraic group $\textbf{G} \leq \GL(n, \mathbb{C})$, we denote the group of $R$-points of $\textbf{G}$ as $\textbf{G}(R) = \textbf{G} \cap \GL(n,R)$.

A $\mathbb{Q}$-\emph{closed subgroup} $\textbf{H}$ of the $\mathbb{Q}$-defined linear algebraic group $\textbf{G}$ is a subgroup that is closed in the Zariski topology. There is a unique structure of an algebraic group on $\textbf{H}$ such that the inclusion map $\textbf{H} \to \textbf{G}$ is defined over $\mathbb{Q}$. Specifically, we have polynomials 
$$
F_1, F_2, \hdots F_{r_{\textbf{G}}}, F_{r_{\textbf{G} + 1}}, \hdots, F_{r_{\textbf{H}}}
$$
satisfying
$$
\textbf{H} = V(F_1, \ldots, F_{r_{\textbf{H}}}) \quad \text{ and } \quad \textbf{G} = V(F_1, \ldots, F_{r_{\textbf{G}}}).
$$

Let $\textbf{G}$ and $\textbf{H}$ be $\mathbb{Q}$-defined linear algebraic groups. A \emph{algebraic morphism of linear algebraic groups} $\Psi: \textbf{G} \to \textbf{H}$ is a group morphism which is also a morphism of $\mathbb{Q}$-varieties. The notions of isomorphism and automorphisms are then clear.
 As usual, we write $\mathbb{G}_a = \mathbb{C}$ with addition as its group operation and $\mathbb{G}_m = \mathbb{C}^*$ with its multiplicative structure. We note that quotients are defined in the category of $\mathbb{Q}$-defined linear algebraic groups. In particular, if $\textbf{N}$ is a $\mathbb{Q}$-closed normal subgroup of $\textbf{G}$, then $\textbf{G} / \textbf{N}$ is a $\mathbb{Q}$-defined linear algebraic group where the natural projection $\textbf{G} \to \textbf{G} / \textbf{N}$ is an algebraic morphism.
We say that $\textbf{G}$ is the \emph{almost direct product} of two Zariski closed subgroups $\textbf{H}$ and $\textbf{K}$ if $\textbf{H}$ and $\textbf{K}$ centralize each other, their intersection $\textbf{H} \cap \textbf{K}$ is finite, and if $\textbf{G} = \textbf{H} \cdot \textbf{K}$.

Given a subgroup $H \leq \textbf{G}(\Q)$, we denote its Zariski closure as $\overline{H}$ which is the smallest linear algebraic group defined over $\mathbb{Q}$ which contains $H$. Given subgroups $X,Y \leq \textbf{G}$ which are not necessarily Zariski closed, we have that if $X$ normalizes $Y$, then $\overline{X}$ normalizes $\overline{Y}$, and in particular, we have $\overline{[X,Y]} = [\overline{X}, \overline{Y}]$. Subsequently, we have  $\gamma_i(\overline{H}) = \overline{\gamma_i(H)}$ and $\overline{H^{(i)}} = (\overline{H})^{(i)}$. Thus, if $H$ is nilpotent of step length $c$, then $\overline{H}$ is nilpotent of step length $c$, and if $H$ is solvable of derived length $\ell$, then $\overline{H}$ is solvable of derived length $\ell$.

Let $\textbf{G}$ be a $\mathbb{Q}$-defined linear algebraic group. The \emph{unipotent radical} $\textbf{U}(\textbf{G})$ is the set of unipotent elements of the maximal closed connected normal solvable subgroup of $\textbf{G}$. It is always a Zariski closed subgroup which is defined over $\mathbb{Q}$. Most linear algebraic groups we consider will be virtually solvable, and in this context, we have that $\textbf{U}(\textbf{G})$ is the set of unipotent elements of $\textbf{G}$. Any surjective algebraic morphism $\varphi \colon \textbf{G} \to \textbf{H}$ of $\mathbb{Q}$-defined groups maps unipotent elements to unipotent elements, and subsequently, $\varphi(\textbf{U}(\textbf{G})) = \textbf{U}(\textbf{H})$. A Zariski closed subgroup of $\textbf{G}$ is a \emph{d-group} if it consists entirely of semisimple elements.  A \emph{torus} is a $\mathbb{Q}$-defined connected abelian linear algebraic group consisting of semisimple elements, so in particular if $\textbf{C}$ is a d-subgroup, then $\textbf{C}^\circ$ is a torus. If the unipotent radical of $\textbf{G}$ is trivial, then the group $\textbf{G}$ is called \emph{reductive}. When $\textbf{G}$ is virtually solvable and reductive, then the linear algebraic group is virtually a torus. A \emph{Levi subgroup} is a reductive linear algebraic subgroup $\textbf{L} \leq \textbf{G}$ such that $\textbf{G}$ is the semidirect product of $\textbf{L}$ and $\textbf{U}(\textbf{G})$. 
We say a linear algebraic group $\textbf{G}$ has a \emph{strong unipotent radical} if the centralizer $C_{\textbf{G}}(\textbf{U}(\textbf{G}))$ of its unipotent radical $\textbf{U}(\textbf{G})$ is  contained  in $\textbf{U}(\textbf{G})$.

Every element $X \in \textbf{G}$ can be written as a product of a unipotent element $X_u$ and a semisimple element $X_s$ that commute. Moreover, if $\textbf{G}$ is a nilpotent $\mathbb{Q}$-defined linear algebraic group, then the maps $\pi_s, \pi_u: \textbf{G} \to \textbf{G}$ given by $\pi_s(X) = X_s$ and $\pi_u(X) = X_u$ are algebraic morphisms. In this case, the maximal torus of $\textbf{G}$ is unique and equal to the set of semisimple elements.

We now list some basic facts that can be found in \cite{borel,platonov_rapinchuk,springer}. Let $\textbf{G}$ be a $\mathbb{Q}$-defined linear algebraic group. 
\begin{enumerate}
\item[\textbf{AG1}] There exists a Levi subgroup $\textbf{L} \leq \textbf{G}$ such that $\textbf{G} = \textbf{U}(\textbf{G}) \cdot \textbf{L}$ where $\textbf{L} \cong \textbf{G} / \textbf{U}(\textbf{G})$. Indeed, we have $\textbf{G} \cong \textbf{U}(\textbf{G}) \rtimes \textbf{L}$. 

\item[\textbf{AG2}] Levi subgroups are conjugate by elements of $\textbf{U}(\textbf{G})$.

\item[\textbf{AG3}] Let $\textbf{T} \leq \textbf{G}$ be a torus. Then the centralizer $C_{\textbf{G}}(\textbf{T})$ of $\textbf{T}$ in $\textbf{G}$ has finite index in the normalizer $N_{\textbf{G}}(\textbf{T})$.
\end{enumerate}

The following facts hold for a $\mathbb{Q}$-defined, virtually solvable linear algebraic group $\textbf{H}$.
\begin{enumerate}
\item[\textbf{SG1}] There exist maximal tori $\textbf{T} \leq \textbf{H}$ that are $\mathbb{Q}$-closed.

\item[\textbf{SG2}] Let $\textbf{T} \leq \textbf{H}$ be a maximal torus. Then $\textbf{C} = N_{\textbf{H}}(\textbf{T})$ is $\mathbb{Q}$-closed. Moreover, it contains a $d$-subgroup $\textbf{S}$ which is $\mathbb{Q}$-closed and satisfies $\textbf{S}^\circ = \textbf{T}$. Such a subgroup $\textbf{C}$ is called a Cartan subgroup.

\item[\textbf{SG3}] Let $\textbf{S}$ be a maximal $d$-subgroup of $\textbf{H}$. Then $\textbf{H} = \textbf{U}(\textbf{H}) \cdot \textbf{S} = \textbf{U}(\textbf{H}) \rtimes \textbf{S}.$

\item[\textbf{SG4}] All $\mathbb{Q}$-closed maximal tori and all $\mathbb{Q}$-closed maximal $d$-subgroups are conjugate by elements of $[\textbf{H}^\circ, \textbf{H}^\circ]$.

\item[\textbf{SG5}] Let $\textbf{C}$ be a Cartan subgroup of $\textbf{H}$ and $\textbf{F} \leq \textbf{U}(\textbf{H})$ a normal subgroup of $\textbf{H}$ which contains $[\textbf{H}^\circ, \textbf{H}^\circ]$. Then we have $\textbf{H} = \textbf{F} \cdot \textbf{C}$.
\end{enumerate}
Justification for all of these facts can be found in \cite{borel,platonov_rapinchuk,springer}. Note that \textbf{AG3} is sometimes called the rigidity of tori.

We now introduce the notion of a $S$-arithmetic subgroup of a $\mathbb{Q}$-defined linear algebraic group where $S$ is a finite set of integral primes.
\begin{defn}
Let $\textbf{G}$ be a $\mathbb{Q}$-defined linear algebraic group, and let $S$ be a finite set of integral primes. We say that a subgroup $G \leq \textbf{G}$ is \emph{$S$-arithmetic} if $G \cap \textbf{G}(\mathbb{Z}[\frac{1}{S}])$ has finite index in both $G$ and $\textbf{G}(\mathbb{Z}[\frac{1}{S}])$. We say an abstract group $G$ is \emph{$S$-arithmetic} if it has a finite index subgroup which is isomorphic to a $S$-arithmetic subgroup of a $\mathbb{Q}$-defined linear algebraic group. When $S$ is empty, we will say $G$ is \emph{arithmetic}. 
\end{defn}

Already having introduced the concept of $S$-arithmetic group, we now describe the behavior of $S$-arithmetic subgroups under algebraic morphisms. For this discussion, let $\rho \colon \textbf{A}_1 \to \textbf{A}_2$ be a surjective algebraic morphism between $\mathbb{Q}$-defined linear algebraic groups, and let $A \leq \textbf{A}_1$ be a subgroup. The following hold:

\begin{itemize}
\item[\textbf{AR1}]  If $A$ is an $S$-arithmetic subgroup of $\textbf{A}_1$, then $\rho(A)$ is a $S$-arithmetic subgroup of $\textbf{A}_2$.

\item[\textbf{AR2}] If $\rho$ is injective, then $A$ is a $S$-arithmetic in $\textbf{A}_1$ if and only if $\rho(A)$ is commensurable with $\textbf{A}_2(\mathbb{Z}[\frac{1}{S}])$.

\item[\textbf{AR3}] Every solvable $S$-arithmetic group with $S$ finite has finite Hirsch length.
\item[\textbf{AR4}] Let $\textbf{A}$ be a $\mathbb{Q}$-defined linear algebraic group, and $\textbf{A}_1, \textbf{A}_2$ be $\mathbb{Q}$-defined linear algebraic groups such that $\textbf{A} = \textbf{A}_1 \rtimes \textbf{A}_2$. Let $A$ be a group satisfying the following. There exists subgroups $A_1, A_2 \leq A(\mathbb{Q})$ such that $A_1 \leq \textbf{A}_1$ and $A_2 \leq \textbf{A}_2(\mathbb{Q})$ are  $S$-arithmetic subgroups and where $A = A_1 \rtimes A_2$. Then $A$ is $S$-arithmetic in $\textbf{A}$.
\end{itemize}

We have that \textbf{AR1} is a consequence of \cite[Theorem 5.9]{platonov_rapinchuk} and that \textbf{AR2} is a natural consequence of \textbf{AR1}. Property \textbf{AR4} follows from \cite[Lemma 5.9]{platonov_rapinchuk}, whereas \textbf{AR3} is a consequence of the fact that the Hirsch length of an abelian group can be bounded in terms of the dimension of its Zariski closure.

\section{Algebraic hulls for FAR-groups}\label{algebraic_hull_section}

Definition \ref{def:algebraichull} introduced the $\Q$-algebraic hull of a virtually FAR-group. This section is devoted to the proof of Theorem \ref{algebraic_hull_thm}. The proof is split in two parts: the first discussing the construction of the $\Q$-algebraic hull, and the second dealing with its uniqueness and other properties.

\subsection{Construction of the $\Q$-algebraic hull}

In order to construct the $\Q$-algebraic hull, we first construct algebraic groups satisfying condition $(ii)$ of Definition \ref{def:algebraichull}. Let $\textbf{G}$ be any virtually solvable $\mathbb{Q}$-defined linear algebraic group. Recall that $\textbf{G}$ has a strong unipotent radical if $C_{\textbf{G}}(\textbf{U}(\textbf{G}))$ lies in $\textbf{U}(\textbf{G})$. This is equivalent to saying that the only semisimple element of $\textbf{G}$ that centralizes $\textbf{U}(\textbf{G)}$ is the trivial one. The following lemma shows that every $\mathbb{Q}$-defined virtually solvable group $\textbf{G}$ has a canonical largest quotient satisfying this property.

\begin{lemma}
	\label{quotientstrongunipotent}
	Let $\textbf{G}$ be a virtually solvable $\mathbb{Q}$-defined linear algebraic group. Take $\textbf{S}$ equal to the set of semisimple elements of $\textbf{G}$ that centralize the unipotent radical $\textbf{U}(\textbf{G})$. Then $\textbf{S}$ is a $\mathbb{Q}$-closed normal subgroup, and the quotient  $\textbf{G} / \textbf{S}$ is a $\mathbb{Q}$-defined linear algebraic group with a strong unipotent radical.
\end{lemma}
\begin{proof}
	The centralizer $\textbf{C} = C_{\textbf{G}}(\textbf{U}(\textbf{G}))$ is a $\mathbb{Q}$-closed subgroup with unipotent radical equal to $\textbf{V} = \textbf{U}(\textbf{G}) \cap \textbf{C}$. As a Levi subgroup in $\textbf{C}$ is unique up to conjugation by unipotent elements and these commute with all elements in $\textbf{C}$, there exists a unique reductive subgroup $\textbf{L} \subset \textbf{C}$ such that $\textbf{C} = \textbf{L} \cdot \textbf{V}$, where $\textbf{L}$ and $\textbf{V}$ commute. This implies that $\textbf{L}$ consists exactly of the semisimple elements in $C_{\textbf{G}}(\textbf{U}(\textbf{G}))$, or thus that $\textbf{L} = \textbf{S}$ is a subgroup.
	
	Since $\textbf{C}$ is normal in $\textbf{G}$, it follows that $\textbf{S}$ as the set of semisimple elements in $\textbf{C}$ is again normal. Hence, we can take the quotient $\textbf{G} / \textbf{S}$ which is a $\mathbb{Q}$-defined linear algebraic group. Moreover, the kernel of the quotient map $\pi: \textbf{G} \to \textbf{G}/\textbf{S}$ consists of semisimple elements, and thus, $\textbf{U}(\textbf{G}/\textbf{S}) \cong \textbf{U}(\textbf{G})$ via the restriction of $\pi$. Any semisimple element $s \in \textbf{G}$ with $s \notin \textbf{S}$ acts non-trivally via conjugation on $\textbf{U}(\textbf{G})$, and thus,  $\pi(s) \neq 1$ also acts non-trivially on $\pi(\textbf{U}(\textbf{G})) = \textbf{U}(\textbf{G}/\textbf{S})$. Equivalently, the group $\textbf{G}/\textbf{S}$ has a strong unipotent radical.  
\end{proof}

\begin{prop}
	\label{prop:dim}
	Suppose that $\textbf{G}$ is a $\Q$-defined linear algebraic group and that $\Gamma \subset \textbf{G}(\Q)$ is a virtually solvable Zariski dense subgroup. Then $\dim_\Q(\!\textbf{U}(\textbf{G})) \leq h(\Gamma)$.  Moreover, if $\textbf{G}$ is unipotent, then equality holds.
\end{prop}
\noindent Note that by our assumptions, the group $\Gamma$ is a subset of the $\Q$-rational points of $\textbf{G}$, otherwise the proposition is false, for example by embedding $\Q^2$ as a Zariski-dense subgroup of $\mathbb{R}$.
\begin{proof}
	Note that both $\dim_\Q(\textbf{U}(\textbf{G}))$ and $h(\Gamma)$ do not change under taking finite index subgroups. Hence, we may assume that $\textbf{G}$ is connected by replacing $\textbf{G}$ by $\textbf{G}^\circ$ and $\Gamma$ by $\Gamma \cap \textbf{G}^\circ$, which are both of finite index in the original groups. Under these assumptions, the group $\textbf{G}$ and hence also $\Gamma$ is solvable. 
	
	We first prove the second statement, so we assume that $\textbf{G}$ is unipotent and in particular nilpotent, and show that $\dim_\Q(\textbf{G}) = h(\Gamma)$. We prove this via induction on $\dim_\Q(\textbf{G})$, where the case $\dim_{\mathbb{Q}}(\textbf{G}) = 0$ or equivalently $\textbf{G} = \{1\}$ is immediate. Now take any normal subgroup $N \triangleleft \Gamma$ with $\Gamma/N$ abelian and $h(N) = h(\Gamma) - 1$, which of course always exists as $\Gamma$ is nilpotent. Take $\overline{N}$ as the algebraic closure of $N$ in $\textbf{G}$, then via induction we know that $\dim_\Q(\overline{N}) = h(\Gamma) - 1$. As the abelian group $\Gamma / N$ lies Zariski-dense in the unipotent group $\textbf{G}/\overline{N}$, the latter is also abelian of dimension $1$ and thus isomorphic to $\mathbb{C}$. We conclude that the statement follows.
	
For the first statement, write $N = \Gamma \cap \textbf{U}(\textbf{G})$, which is a nilpotent normal subgroup of $\Gamma$. Moreover, as the quotient $\textbf{G} /  \textbf{U}(\textbf{G})$ is abelian, so is $\Gamma / N$ as the image of $\Gamma$ under the projection $\textbf{G} \to \textbf{G}/ \textbf{U}(\textbf{G})$. Note that $h(\Gamma) = h(\Gamma/N) + h(N)$ by definition of the Hirsch length. Now take $\overline{N}$ as the algebraic closure of $N$ in $\textbf{U}(\textbf{G})$, then by the unipotent case, we find that $\dim_\Q(\overline{N}) = h(N)$. The group $\textbf{G} / \overline{N}$ is again abelian as it contains the group $\Gamma / N$ as a Zariski-dense subgroup. Consider the natural projection $\pi_u: \textbf{G}/\overline{N} \to \textbf{U}(\textbf{G}/\overline{N}) = \textbf{U}(\textbf{G}) / \overline{N}$, then $\pi_u(\Gamma/N)$ is still Zariski-dense in $ \textbf{U}(\textbf{G})/\overline{N}$. In particular, we get that $$\dim_\Q( \textbf{U}(\textbf{G})/\overline{N}) = h(\pi_u(\Gamma/ N)) \leq h(\Gamma/N)$$ and thus also $\dim_\Q( \textbf{U}(\textbf{G})) \leq h(\Gamma)$. 
\end{proof}

This motivates the following definition.

\begin{defn}
	Let $\Gamma$ be a  virtually FAR-group. We call a representation $\rho: \Gamma \to \textbf{H}$ to a $\mathbb{Q}$-defined linear algebraic group $\textbf{H}$ \emph{full} if $\Gamma$ if $\dim_{\mathbb{Q}}( \textbf{U}(\textbf{G})) = h(\Gamma)$, where $\textbf{G}$ is the Zariski-closure of $\rho(\Gamma)$ in $\textbf{H}$.
\end{defn}

For abelian groups, the existence of such a representation is automatic from the definition. 
\begin{lemma}
	\label{ex:ab} 
	Every torsion-free abelian FAR-group $\Gamma$ has an injective full representation. Moreover, if $\Gamma$ is minimax, we can assume that $\rho(\Gamma)$ lies in $\textbf{H}(\Z[\frac{1}{S}])$, where $S$ is the spectrum of $\Gamma$.
\end{lemma}

\begin{proof}
	This is immediate, as such a group $A$ is a subgroup of $A \otimes \mathbb{Q} \approx \mathbb{Q}^k$ for $k = h(A)$ and the latter has a faithful unipotent representation. The latter statement follows from the definition of the spectrum for minimax groups.
\end{proof}

The existence of $\mathbb{Q}$-algebraic hulls for virtually FAR-groups with no torsion normal subgroups is a consequence of the previous results.

\begin{thm}\label{algebraic_hull}
	If $\Gamma$ is a  virtually FAR \emph{WTN}-group with spectrum $S$, then $\Gamma$ has a $\mathbb{Q}$-algebraic hull.
\end{thm}

\begin{proof}
	We only give the proof for minimax groups $\Gamma$, the general case follows by replacing $\mathbb{Z}\left[\frac{1}{S}\right]$ by $\Q$. We start by showing that there exists an injective full representation of $\Gamma$ with image in $\textbf{H}\left(\mathbb{Z}\left[\frac{1}{S}\right]\right)$. Note that it suffices to show this for a finite index subgroup of $\Gamma$. Indeed, if $\rho: \Gamma_0 \to \GL(n, \mathbb{Z}\left[\frac{1}{S}\right])$ is a full representation, then the induced representation $\tilde{\rho} \colon \Gamma \to \overline{\tilde{\rho}(\Gamma)} \leq \GL(nd, \mathbb{Z}\left[\frac{1}{S}\right])$ will be full as well where $d = |\Gamma : \Gamma_0|$. So from now on, we can assume that $\Gamma / \Fitt(\Gamma)$ is abelian by replacing $\Gamma$ by a finite index subgroup if necessary. 
	
	We know that the group $\Gamma$ is linear over $\Q$ by Lemma \ref{lem:equivalencesRF}, meaning that there exists an injective map $$\rho_1: \Gamma \to \GL\left(n,\mathbb{Z}\left[\tfrac{1}{S}\right]\right)$$ for some $n > 0$.  Now take the Zariski closure $\textbf{H}_1 \subset \GL(n,\mathbb{C})$ of $\rho_1(\Gamma)$, so $\textbf{H}_1$ is a $\mathbb{Q}$-defined linear algebraic group. Similarly as above, we can assume that $\textbf{H}_1$ is connected. As $\Gamma$ is solvable, so is $\textbf{H}_1$, and moreover, $\dim_\Q(\textbf{U}(\textbf{H}_1)) \leq h(\Gamma)$. 
	
	Take the nilpotent normal subgroup $N = \Gamma \cap \textbf{U}(\textbf{H})$ of $\Gamma$, which thus lies in the Fitting subgroup $\Fitt(\Gamma)$. Since the Zariski-closure $\overline{N}$ of $N$ lies in $\textbf{U}(\textbf{H})$, Proposition \ref{prop:dim} implies $\dim_\Q(\overline{N}) = h(N)$. Now consider the group $\Gamma / N$, which is a group of rank $h(\Gamma)- h(N)$. By considering this group as a subgroup of $\textbf{H} / \textbf{U}(\textbf{H})$, we see that $\Gamma / N$ is abelian. Therefore, we can use Lemma \ref{ex:ab} on the torsion-free quotient of $\Gamma / N$ to find a map $$\rho_2: \Gamma / N \to \GL\left(m,\mathbb{Z}\left[\tfrac{1}{S}\right]\right)$$ that is full. Letting $\pi \colon \Gamma \to \Gamma/ N$ to be the natural projection, the injective morphism $$\rho \colon \Gamma \to \GL\left(n, \mathbb{Z}\left[\tfrac{1}{S}\right]\right) \oplus \GL\left(m, \mathbb{Z}\left[\tfrac{1}{S} \right]\right)$$ given by $\rho(x) = (\rho_1(x), \rho_2(\pi(x)))$ is the desired representation that is full, as a direct argument shows.

	Now we use the full representation to prove the theorem. Note that we can apply Lemma \ref{quotientstrongunipotent} to $\overline{\rho(\Gamma)}$ to find the $\mathbb{Q}$-algebraic hull $\textbf{H}_\Gamma = \overline{\rho(\Gamma)}/\textbf{S}$ and the natural map $\iota: \Gamma \to \textbf{H}_\Gamma$ where $\textbf{S}$ is the set of semisimple elements of $\overline{\rho(\Gamma)}$ that centralize $\textbf{U}(\overline{\rho(\Gamma)})$. As the other properties are immediate, we are left to show that $\iota$ remains injective. However, note that $h(\Gamma) = \dim_{\mathbb{Q}}(\textbf{U}(\textbf{H}_\Gamma))$ and hence that the kernel of this map satisfies $h(\ker(\iota)) = 0$ by Proposition \ref{prop:dim}. In particular, $\ker(\iota)$ is a torsion normal subgroup of $\Gamma$, and hence trivial by the assumption. 
\end{proof}

Finally, we show that the conditions of the previous theorem are necessary for the existence of a $\Q$-algebraic hull.
\begin{thm}
	A  virtually FAR-group $\Gamma$ has a $\mathbb{Q}$-algebraic hull if and only if $\tau(\Gamma) = 1$, so if and only if $\Gamma$ is a \emph{WTN}-group. 
\end{thm}
\begin{proof}
	Note that we have shown in the theorem above that $\tau(\Gamma) = 1$ is a sufficient condition. Conversely, assume that $\Gamma$ is a group and $\iota: \Gamma \to \textbf{H}_\Gamma$ is its $\mathbb{Q}$-algebraic hull. Write $T = \iota(\tau(\Gamma))$, which is the maximal torsion normal subgroup of the image, and $T$ is locally finite under our assumptions. By the Jordan-Zassenhaus theorem, the group $T$ is virtually abelian, and since $T$ is torsion, it consists of semisimple elements. Hence, its algebraic closure $\overline{T}$ is virtually a torus. The group $\textbf{U}(\textbf{H}_\Gamma)$ acts on $\overline{T}$ via conjugation, and by the rigidity of tori, this action must be trivial. In particular, $T \leq C_{\textbf{H}_\Gamma}(\textbf{U}(\textbf{H}_\Gamma))$ and thus must be trivial.
\end{proof}

\begin{ex}
	\label{ex:BS}
	For the Baumslag-Solitar group $\Gamma = BS(1,n) = \langle x, t \mid t^{-1} x t = x^n \rangle$ with $\vert n \vert > 1$, the construction of Theorem \ref{algebraic_hull} gives us the embedding $$\iota: \Gamma \to \GL\left(2,\mathbb{Z}\left[\tfrac{1}{n}\right]\right) \times \GL\left(2,\mathbb{Z}\left[\tfrac{1}{n}\right]\right) \subset \GL\left(4,\mathbb{Z}\left[\tfrac{1}{n}\right]\right)$$ defined by \begin{align*}\iota(x) &= \left(\begin{pmatrix} 1 & 1 \\ 0 & 1 \end{pmatrix}, \begin{pmatrix}1 & 0 \\ 0 & 1 \end{pmatrix}\right)\\ \iota(t) &= \left(\begin{pmatrix} \frac{1}{n} & 0 \\ 0 & 1 \end{pmatrix}, \begin{pmatrix} 1 & 1 \\ 0 & 1 \end{pmatrix}\right),\end{align*} leading to the algebraic hull $\textbf{H}_\Gamma$ being isomorphic to $\C^2 \rtimes \C^\ast = U(\textbf{H}_\Gamma) T$ where the action is by multiplication in the first component and trivial in the second.
\end{ex}

\subsection{Properties of the $\Q$-algebraic hull}\label{properties_Q_hull}

The following result is key to showing  the $\Q$-algebraic hull of a virtually FAR-group $\Gamma$ is a universal construction in the category of $\mathbb{Q}$-defined linear algebraic groups. 

\begin{prop}\label{isomorphism_hull}
	Let $(\textbf{H}_1, \iota_1)$ and $(\textbf{H}_2, \iota_2)$ be $\Q$-algebraic hulls for  virtually FAR-groups $\Gamma_1$ and $\Gamma_2$, respectively, and $\varphi: \Gamma_1 \to \Gamma_2$ a surjective morphism. There exists a unique algebraic morphism $\Psi \colon \textbf{H}_1 \to \textbf{H}_2$ such that $\iota_2 \circ \varphi = \Psi \circ \iota_1$.
\end{prop}
\begin{proof}
	Consider the subgroup
	$$
	\Delta = \left\{ \left(\iota_1(x), \iota_2(\varphi(x))\right) \: | \: x \in \Gamma_1 \right\} \subset \textbf{H}_1 \times \textbf{H}_2,
	$$
which is isomorphic to $\Gamma_1$ as $\iota_1$ is injective. In particular, we have $h(\Delta) = h(\Gamma_1)$. Denote by $\overline{\Delta}$ the Zariski closure of $\Delta$ in the product of algebraic groups $\textbf{H}_1 \times \textbf{H}_2$. Let $\pi_i \colon \textbf{H}_1 \times \textbf{H}_2 \to \textbf{H}_i$ be the natural projections where $i \in \{1,2\}$, and let $\alpha_i$ to be the restriction of $\pi_i$ to the subgroup $\overline{\Delta}$. Note that $\alpha_1(\Delta) = \iota_1(\Gamma_1)$ and $\alpha_2(\Delta) = \iota_2(\varphi(\Gamma_1)) = \iota_2(\Gamma_2)$ are both Zariski dense in their image, and thus, the maps $\alpha_i$ are surjective.
	
	We first show that the map $\alpha_1$ is an isomorphism of algebraic groups. Consider $\textbf{U}(\overline{\Delta})$ the unipotent radical of $\overline{\Delta}$. As $\alpha_1$ is surjective, we get $\alpha_1(\textbf{U}(\overline{\Delta})) = \textbf{U}(\textbf{H}_1)$ and the same holds for $\alpha_2$. By comparing the dimensions, we find that
	$$
h(\Gamma_1) = \dim_\Q(\textbf{U}(\textbf{H}_1)) \leq \dim_\Q(\textbf{U}(\overline{\Delta})) \leq h(\Delta) = h(\Gamma_1),
	$$ 
	and therefore, 
	$$
	\dim_\Q(\textbf{U}(\textbf{H}_1)) = \dim_\Q(\textbf{U}(\overline{\Delta})).
	$$
	Hence, $\alpha_1$ is an isomorphism when restricted to $\textbf{U}(\overline{\Delta})$. Thus, the subgroup $\ker( \alpha_1)$ consists entirely of semisimple elements. 
	
	To show that the kernel is trivial, take any $x \in \ker(\alpha_1)$. For every $y \in \textbf{U}(\overline{\Delta})$, we have 
	$$
	\alpha_1(xyx^{-1}) = \alpha_1(x) \alpha_1(y) \alpha(x)^{-1} = \alpha_1(y).
	$$
	Since $\alpha_1$ maps $\textbf{U}(\overline{\Delta})$ isomorphically to $\textbf{U}(\textbf{H}_1)$, we have $xyx^{-1} = y$. Hence, if $x \in \ker(\alpha_1)$, then $x$ centralizes $\textbf{U}(\overline{\Delta})$. Thus, $\alpha_2(x)$ also centralizes $\textbf{U}(\textbf{H}_2) = \alpha_2(\textbf{U}(\overline{\Delta}))$, and since $\textbf{H}_2$ has a strong unipotent radical, we have $\alpha_2(x) \in \textbf{U}(\textbf{H}_1)$. Since $x$ is semisimple and thus also $\alpha_2(x)$, we find that $\alpha_2(x) = 1$. In combination with $\alpha_1(x) = 1$, this implies that $x = 1$. Since $\ker(\alpha_1)$ is trivial and $\alpha_1$ is a surjective algebraic morphism, we have $\alpha_1$ is an isomorphism defined over $\Q$ by \cite[Section 6.1.]{borel}. 
	
	Now, consider the map $\Psi = \alpha_2 \circ \alpha_1^{-1}: \textbf{H}_1 \to \textbf{H}_2$, which is a morphism defined over $\mathbb{Q}$. By construction, it satisfies $\Psi \circ \iota_1 = \iota_2 \circ \varphi$ and thus is the map we needed.
\end{proof}

\begin{cor}\label{unipotent_on_fitting_subgroup}
	The $\Q$-algebraic hull $\textbf{H}$ of $\Gamma$ is unique up to an isomorphism over $\mathbb{Q}$ of algebraic groups. In particular, every automorphism of $\Gamma$ extends uniquely to an algebraic automorphism of $\textbf{H}$.
\end{cor}

For finite index subgroups of  virtually  FAR-groups, it is easy to determine the $\mathbb{Q}$-algebraic hull.
\begin{lemma}
	\label{fihull}
	Let $\Gamma^\prime$ be a finite index subgroup of $\Gamma$ with $(\textbf{H}_\Gamma, \iota)$ being the $\mathbb{Q}$-algebraic hull of $\Gamma$. The $\Q$-algebraic hull of $\Gamma^\prime$ is equal to the $\iota^\prime: \Gamma^\prime \to \overline{\iota(\Gamma^\prime)}$. In particular, if $\textbf{H}_\Gamma$ is connected, then the $\mathbb{Q}$-algebraic hull of $\Gamma'$ is given by $(\textbf{H}_\Gamma, \restr{\iota}{\Gamma'})$. 
\end{lemma}

\begin{proof}
Conditions $(i)$ and $(ii)$ follow immediately and condition $(iii)$ follows from the fact that $h(\Gamma) = h(\Gamma^\prime)$. For the second part, note that $\overline{\iota(\Gamma^\prime)}$ is a finite index subgroup of $\textbf{H}_\Gamma$ and hence equal to $\textbf{H}_\Gamma$. 
\end{proof}

The previous lemma will be convenient to study injective morphisms. Indeed, in the following proposition, we show how injectivity implies that the image always has finite index. 

\begin{prop}\label{injfi}
	Let $\Gamma$ be a virtually FAR-group. If $\varphi: \Gamma \to \Gamma$ is an injective morphism, then $\varphi(\Gamma)$ has finite index in $\Gamma$.
\end{prop}
 Recall that a group $G$ is co-Hopfian if every injective endomorphism $\varphi \colon G \to G$ is in fact an automorphism. We start by proving the following lemma.

\begin{lemma}
	\label{lem:cohopf}
If a group satisfies \textbf{Min}, then it is co-Hopfian.
\end{lemma}

\begin{proof}
We prove this claim via contraposition, and to proceed, we assume that there exists a injective morphism $\varphi: G \to G$ such that $\varphi(G)$ is a proper subgroup of $G$. Letting $G_i = \varphi^i(G)$ for $i>0$, induction implies that $G_{i+1}$ is a proper subgroup of $G_i$ for all $i$. Hence, $\{G_i\}_{i=1}^\infty$ is a strictly descending chain of subgroups of $G$ that does not stabilize after finitely many steps. Thus, $G$ does not satisfy $\textbf{Min}$.
\end{proof}

\begin{proof}[Proof of Proposition \ref{injfi}]
First assume that $\Gamma$ is \textit{solvable}. We prove this via induction on the derived length $d$.

For the base case, assume that $\Gamma$ is abelian. If $T$ is the subgroup of torsion elements, we will then show that $\varphi(T) = T$. Let $A_p$ be the $p$-component of $T$, i.e.~the set of elements in $T$ that have order a power of $p$. As $\Gamma$ has finite abelian ranks, we know that $A_p$ satisfies \textbf{Min} by \cite[5.1.2.]{lennox_robinson}, and thus, the group is co-Hopfian by Lemma \ref{lem:cohopf}. In particular, $\varphi(A_p) = A_p$, and as $T$ is equal to the direct sum of its $p$-components, we get that $\varphi(T) = T$.

To continue, note that the injective morphism $\varphi \colon \Gamma \to \Gamma$ induces an injective morphism $\bar{\varphi} \colon \bar{\Gamma} \to \bar{\Gamma}$ on the torsion-free quotient $\bar{\Gamma} = \Gamma/T$. By \cite[6.1.3]{lennox_robinson}, we have that $\bar{\varphi}\left(\bar{\Gamma}\right)$ has finite index in $\bar{\Gamma}$. Therefore we know that
\begin{eqnarray*}
|\Gamma : \varphi(\Gamma)| =  |\bar{\Gamma} : \bar{\varphi}(\bar{\Gamma})|  < \infty,
\end{eqnarray*} and thus, the abelian case holds.

Now suppose that $\Gamma$ has derived length $d>1$, and let $A$ be the last nontrivial term of the derived series, which is invariant under $\varphi$. Define the normal subgroup 
$$
B = \left\{ x \in \Gamma \mid  \exists \: n \in \mathbb{N}_0:  \varphi^n(x) \in A \right\}.
$$
It then immediately follows that $B$ is also $\varphi$-invariant. Moreover, the induced map $\bar{\varphi} \colon \Gamma / B \to \Gamma / B$ is injective. Note that $B$ is abelian, as if $x, y \in B$, there exists $m, n \in \mathbb{N}$ such that $\varphi^m(x), \varphi^n(y) \in A$. By replacing $m$ and $n$ by their maximum, we may assume that $m = n$. Now as $A$ is abelian, we have $\varphi^n(x y x^{-1} y^{-1}) = 1$ and thus, as $\varphi$ is injective, that also implies $x y x^{-1} y^{-1} = 1$. In particular, as this holds for all $x, y \in B$, the group is abelian. By induction, $|\Gamma / B : \bar{\varphi}(\Gamma / B) | < \infty$, and the base case implies $|B : \varphi(B)| < \infty$. By combining the previous inequalities, we find that $$[\Gamma:\varphi(\Gamma)] = |\Gamma / B : \bar{\varphi}(\Gamma / B) | \cdot |B : \varphi(B)| < \infty.$$

Now we proceed with the general case of when $\Gamma$ is a virtually FAR-group, and let $\Gamma_0 \leq \Gamma$ be a normal solvable subgroup of finite abelian ranks of index $m$. Letting $\Gamma^m = \left<x^m \: | \: x \in \Gamma \right>$, we see that $\Gamma/ \Gamma^m$ is a virtually solvable group of finite abelian rank with exponent at most $m$, and hence finite. Since $\Gamma^m \leq \Gamma_0,$ it is solvable of finite abelian ranks, and we see by the above arguments that $\varphi(\Gamma^m)$ has finite index in $\Gamma^m$. Thus, we may write
\begin{align*}
|\Gamma : \varphi(\Gamma)| \leq | \Gamma : \varphi(\Gamma^m) | = \vert \Gamma : \Gamma^m \vert \cdot \vert \Gamma^m : \varphi(\Gamma^m) \vert < \infty. \quad \quad \quad \qedhere
\end{align*}\end{proof}

\begin{cor}\label{injectiveextension}
	Let $\varphi: \Gamma \to \Gamma$ be an injective morphism of a virtually FAR-group with $\mathbb{Q}$-algebraic hull $\iota: \Gamma \to \textbf{H}$, then $\varphi$ has a unique extension $\Psi: \textbf{H} \to \textbf{H}$.
\end{cor}

\begin{proof}
This is a combination of Proposition \ref{injfi} and Lemma \ref{fihull}.
\end{proof}

\begin{prop}\label{unipotent_hull_fitt_alg_hull}
	Let $\Gamma$ be a  virtually FAR \emph{WTN}-group with $\Q$-algebraic hull $\textbf{H}_\Gamma$. If $\textbf{U}(\textbf{H}_\Gamma)$ is the unipotent radical of $\textbf{H}_\Gamma$, then $\textbf{U}(\textbf{H}_\Gamma) \cap \Gamma= \Fitt(\Gamma)$.
\end{prop}
\begin{proof}
		As $\textbf{U}(\textbf{H}_\Gamma) \cap \Gamma$ is a nilpotent normal subgroup of $\Gamma$, it is clear that  $\textbf{U}(\textbf{H}_\Gamma) \cap \Gamma \leq \Fitt(\Gamma)$. Hence, we focus on the other inclusion.
		
	Take any maximal nilpotent normal $\textbf{N}$ subgroup of $\textbf{H}_\Gamma$, then we show that $\textbf{N}$ is unipotent. By maximality, we get that $\textbf{N}$ is Zariski-closed, and moreover $\textbf{U}(\textbf{N}) = \textbf{U}(\textbf{H}_\Gamma)$. As $\textbf{N}$ is nilpotent, it is isomorphic to the direct sum of $\textbf{U}(\textbf{N})$ and the subgroup $\textbf{N}_s$ of semisimple elements. In particular, any element in $\textbf{N}_s$ would centralize $\textbf{U}(\textbf{N}) = \textbf{U}(\textbf{H}_\Gamma)$ and hence is trivial. 
	
	In particular, the unique maximal normal nilpotent subgroup of $\textbf{H}_\Gamma$ is given by $\textbf{U}(\textbf{H}_\Gamma)$. Hence, $\overline{\Fitt(\Gamma)} \leq \textbf{U}(\textbf{H}_\Gamma)$. Therefore,
	$$
	\Fitt(\Gamma) \leq \overline{\Fitt(\Gamma)} \cap \Gamma \leq \textbf{U}(\textbf{H}_\Gamma) \cap \Gamma
	$$ which ends this proposition.
\end{proof}

Via Corollary \ref{unipotent_on_fitting_subgroup}, once an embedding $\Gamma \leq \textbf{H}_\Gamma$ is fixed, we may identify $\Aut(\Gamma)$ with a subgroup of $\Aut(\textbf{H}_\Gamma)$. It is important to note that $\Aut(\Gamma)$ is not Zariski dense in $\Aut_{\mathbb{Q}}(\textbf{H}_\Gamma)$ since elements of $\Aut(\Gamma)$ preserve $\Fitt(\Gamma)$. That means they also preserve the Zariski closure $\textbf{F}$ of $\Fitt(\Gamma)$, which is not always the case for a general automorphism.

\begin{ex}
Consider the group $\Gamma = \Z^3 \rtimes_A \Z$, where $A \in \GL(3,\Z)$ is given by $$A = \begin{pmatrix}
	2 & 1 & 0\\1 & 1 & 0 \\ 0 & 0 & 1 
\end{pmatrix}.$$
Note that the $\Q$-algebraic hull has unipotent radical $\mathbb{C}^4$ where $\textbf{F} = \C^3 \times \{0\}$ and $Z(\textbf{H}) = \{0\} \times \{0\}\times \C^2$. In particular, $\textbf{F}$ is not preserved by every automorphism of $\textbf{H}_\Gamma$, for example the automorphism that permutes the two components of the center and is the identity on the other components.
	\end{ex}

\section{Thickenings and virtually nilpotent supplements}
\label{sec:thick}

In this section, we demonstrate for a finitely generated virtually FAR \emph{WTN}-group $\Gamma$ with $\Q$-algebraic hull $\textbf{H}_\Gamma$ there exists a finite extension $\tilde{\Gamma} \leq \textbf{H}_\Gamma$, which is called a \emph{thickening} of $\Gamma$, that satisfies the following properties:
\begin{enumerate}[(i)]
\item the group $\Aut(\tilde{\Gamma})$ is a finite extension of $\Aut(\Gamma)$;
\item there exists a virtually nilpotent supplement $C$ for $\Fitt(\tilde{\Gamma})$ in $\tilde{\Gamma}$ that is maximal with respect to inclusion;
\item there exists a Zariski-dense subgroup $\theta \leq U(H_\Gamma)$, called the \emph{strong unipotent shadow}, such that $ \theta \cap \overline{\Fitt(\tilde{\Gamma})} = \Fitt(\tilde{\Gamma})$, $\tilde{\Gamma}$ normalizes $\theta$, and if $\varphi \in \Aut(\tilde{\Gamma})$ satisfies $\varphi(C) = C$, then the extension $\Psi$ of $\varphi$ to $\textbf{H}_\Gamma$ satisfies $\Psi(\theta) = \theta$.
\end{enumerate}
Let $\Gamma$ be a finitely generated virtually FAR-group. Since $\Fitt(\Gamma)$ is a characteristic nilpotent subgroup, we may define
$$
A_{\Gamma | \Fitt(\Gamma)} = \left\{ \varphi \in \Aut(\Gamma) \: | \: \varphi|_{\Gamma / \Fitt(\Gamma)} = \text{id}_{\Gamma / \Fitt(\Gamma)} \right\}.
$$ Section \ref{section_s_arithmetic_subgroups} then uses the strong unipotent shadow to show that $A|_{\tilde{\Gamma} \: | \: \Fitt(\tilde{\Gamma})}$ is a $S$-arithmetic subgroup of $\textbf{A}_{\textbf{H}_\Gamma \: | \: \overline{\Fitt(\tilde{\Gamma})}}$. We refer to Section \ref{section_s_arithmetic_subgroups} for the exact definitions.

To simplify the following discussion, we define $\textbf{F} = \overline{\Fitt(\Gamma)} \leq \textbf{H}_\Gamma$ as the Zariski closure of the Fitting subgroup of $\Gamma$. In this case, we have that $\textbf{F}$ is a connected unipotent normal subgroup of $\textbf{H}_\Gamma$ which is defined over $\mathbb{Q}$. We also have the following property.
\begin{prop}
	The commutator of the connected component $\textbf{H}_\Gamma^\circ$ satisfies $[\textbf{H}_\Gamma^\circ, \textbf{H}_\Gamma^\circ] \leq \textbf{F}$. Moreover, if $\textbf{C}$ is a Cartan subgroup of $\textbf{H}_\Gamma$, then there is a decomposition
	$$
	\textbf{H}_\Gamma = \textbf{F} \cdot \textbf{C}.
	$$
\end{prop}
\begin{proof}
	Since $\Gamma$ is Zariski dense in $\textbf{H}_\Gamma$, we have 
	$$
	[\textbf{H}_\Gamma^\circ, \textbf{H}_\Gamma^\circ] = [\overline{\textbf{H}_\Gamma^\circ \cap \Gamma},\overline{ \textbf{H}_\Gamma^\circ \cap \Gamma}] = \overline{[\textbf{H}_\Gamma^\circ \cap \Gamma, \textbf{H}_\Gamma^\circ \cap \Gamma]} \leq \overline{U(H_\Gamma) \cap \Gamma} = \overline{\Fitt(\Gamma)} = \textbf{F}.
	$$
	The decomposition then follows from \textbf{SG5} of Subsection \ref{background_linear_algebraic_groups}.
\end{proof}

\subsection{The automorphism group of the thickening}
We start by introducing thickenings of Zariski dense subgroups in unipotent algebraic groups. Let $\textbf{U}$ be a $\mathbb{Q}$-defined unipotent linear algebraic group with a Zariski dense subgroup $N \leq \textbf{U}(\Q)$. We then have $\textbf{U}(\Q)$ is isomorphic to the rational Mal'tsev completion of $N$, i.e., $\textbf{U}$ is torsion-free, radicable and for each $x \in \textbf{U}$, there exists a $k \in \mathbb{N}$ such that $x^k \in N$. For any $m \in \mathbb{N}$, we define
$$
N^{\frac{1}{m}} = \left< x \in \textbf{U}(\Q) \: | \: x^m \in N \right>.
$$
The following property will be of use when defining the notion of a thickening which follows from \cite[Exercise 7, Chapter 6]{segal}. 
\begin{lemma}\label{thickening_index}
With notations as above, we have $[N^{\frac{1}{m}} : N] < \infty$ for all $m > 0$.
\end{lemma}

Now take $\Gamma$ as before and $\textbf{F}$ be the Zariski closure of $\Fitt(\Gamma)$ in the $\Q$-algebraic hull $\textbf{H}_\Gamma$. Since $\textbf{F}$ is a $\mathbb{Q}$-defined unipotent algebraic group, we have that $(\Fitt(\Gamma))^{\frac{1}{m}}$ is well defined.

\begin{defn}
A subgroup $\tilde{\Gamma}$ of $\textbf{H}_{\Gamma}$ which is of the form $\tilde{\Gamma} =(\Fitt(\Gamma))^{\frac{1}{m}} \cdot \Gamma$ is called a \emph{thickening} of $\Gamma$.
\end{defn}
\noindent 
This next lemma states that any thickening of a finitely generated virtually FAR \emph{WTN}-group $\Gamma$ is also finitely generated, as a straightforward consequence of Lemma \ref{thickening_index}.
\begin{lemma}\label{thickening_finite_index}
Let $\Gamma$ be any finitely generated virtually FAR \emph{WTN}-group with a thickening $\tilde{\Gamma}$. Then $[\tilde{\Gamma} : \Gamma] < \infty$. In particular, $\tilde{\Gamma}$ is finitely generated.
\end{lemma}

From now on, we fix a thickening $\tilde{\Gamma}$ of $\Gamma$. For any $\varphi \in \Aut(\Gamma)$, we write $\Psi \colon \textbf{H}_\Gamma \to \textbf{H}_{\Gamma}$ for the extension of $\varphi$ to $\textbf{H}_\Gamma$. Since $\Psi$ preserves the group of $\mathbb{Q}$-points of $\textbf{H}_\Gamma$ and preserves $\Gamma$, we claim that $\Psi$ must preserve $\tilde{\Gamma}$. Indeed, as $\tilde{\Gamma} = (\Fitt(\Gamma))^{\frac{1}{m}} \cdot \Gamma$ for some $m \in \mathbb{N}$, it is enough to demonstrate that $\Psi((\Fitt(\Gamma))^{\frac{1}{m}}) = (\Fitt(\Gamma))^{\frac{1}{m}}$, but this is immediate from the definitions.
  By restricting $\Psi$ to $\tilde{\Gamma}$, we see that there exists a natural inclusion $\Aut(\Gamma) \hookrightarrow \Aut(\tilde{\Gamma})$.

The following proposition shows that this identification $\Aut(\Gamma) \hookrightarrow \Aut(\tilde{\Gamma})$ gives a finite index subgroup. 
\begin{prop}\label{index_auto_group_thickening}
Let $\Gamma$ be a finitely generated virtually torsion-free FAR \emph{WTN}-group, and let $\tilde{\Gamma}$ be a thickening of $\Gamma$. Then the group 
$$
\Aut(\Gamma) = \{\psi \in \Aut(\tilde{\Gamma}) \: | \: \psi(\Gamma) = \Gamma \}
$$ 
is a subgroup of finite index in $\Aut(\tilde{\Gamma})$.
\end{prop}
\begin{proof}
Proposition \ref{thickening_finite_index} implies that $[\tilde{\Gamma}: \Gamma]< \infty$. Write  $m=[\tilde{\Gamma} : \Gamma]$, then there are only finitely many subgroups of $\tilde{\Gamma}$ of index $m$ as $\tilde{\Gamma}$ is finitely generated. Therefore, $\Aut(\tilde{\Gamma})$ acts on this set of subgroups where $\Aut(\Gamma)$ is the subgroup which stabilizes $\Gamma$. Therefore, $[\Aut(\tilde{\Gamma}) : \Aut(\Gamma)]< \infty$.
\end{proof}

\subsection{Thickenings of $\Gamma$ admit supplements}
We introduce the theory of supplements for the Fitting subgroup of a finitely generated virtually FAR \emph{WTN}-group.

\begin{defn}
Let $\Gamma$ be a virtually solvable group, and let $C \leq \Gamma$ be a virtually nilpotent subgroup. We call $C$ a \emph{virtually nilpotent supplement in $\Gamma$} if and only if $\Gamma = \Fitt(\Gamma) \cdot C$.
\end{defn}
\noindent For a general solvable group, virtually nilpotent supplements do not have to exist, even in the context of polycyclic groups for which an example is given after \cite[Corollary 10.4.4]{lennox_robinson}. However, by replacing $\Gamma$ by a thickening, we show that there exists a virtually nilpotent supplement that is maximal with respect to inclusion of subgroups.

\begin{prop}\label{supplement_thickening}
Let $\Gamma$ be a finitely generated virtually torsion-free FAR \emph{WTN}-group and take $\textbf{C} \leq \textbf{H}_\Gamma$ a $\mathbb{Q}$-defined Cartan subgroup of $\textbf{H}_\Gamma$. There exists a $m \in \mathbb{N}$ such that $\tilde{\Gamma} = (\Fitt(\Gamma))^{\frac{1}{m}} \cdot \Gamma$ satisfies
$$
\tilde{\Gamma} = (\Fitt(\Gamma))^{\frac{1}{m}} \cdot C
$$
with $C = \tilde{\Gamma} \cap \textbf{C}$ a virtually nilpotent supplement for $\tilde{\Gamma}$. Moreover, the group $C$ is maximal with respect to inclusion under these conditions, Zariski dense in $\textbf{C}$, and all maximal virtually nilpotent supplements for $\tilde{\Gamma}$ arise from this construction.
\end{prop}
\begin{proof}
We start by constructing a virtually nilpotent supplement for $\textbf{F}$ in $\Gamma \cdot \textbf{F}$. By definition we have a decomposition $\textbf{H}_\Gamma = \textbf{F} \cdot \textbf{C}$. Therefore, for any $g \in \Gamma \subset \textbf{H}_\Gamma$, we can write $g = n_g c_g$ where $n_g \in \textbf{F}$ and $c_g \in \textbf{C} \cap (\Gamma \cdot \textbf{F})$. Hence, we have $\Gamma \cdot \textbf{F} = \tilde{C} \cdot \textbf{F}$ where $\tilde{C} = \left(\Gamma \cdot \textbf{F}\right) \cap \textbf{C}$. 

Now we use $\tilde{C}$ to find a supplement for some thickening $\tilde{\Gamma}$. The natural map $\tilde{C} \to (\Gamma \cdot  \textbf{F}) / \textbf{F}$ is surjective. Given that the latter group is finitely generated, there exist elements $c_1, \ldots, c_k \in \tilde{C}$ such that the images generate $(\Gamma \cdot \textbf{F})/ \textbf{F}$. Write $C \subset \textbf{H}_\Gamma$ for the group generated by the $c_i$ for $i=1, \ldots, k$. If we write each $c_i = g_i n_i$ with $g_i \in \Gamma$ and $n_i \in \textbf{F}$, there exists $m \in \mathbb{N}$ such that $n_i \in (\Fitt(\Gamma))^{\frac{1}{m}}$ for any $i = 1, \ldots, k$. Write $\tilde{\Gamma} = \Gamma \cdot (\Fitt(\Gamma))^{\frac{1}{m}}$, then $C \subset \tilde{\Gamma}$ and hence also $(\Fitt(\Gamma))^{\frac{1}{m}} \cdot C \leq \tilde{\Gamma}$. Given that the map $C \to (\Gamma \cdot \textbf{F})  / \textbf{F}$ is surjective, we have that each $g \in \Gamma$ is of the form $g = cn$ where $c \in C$ and $n \in \textbf{F} \cap \tilde{\Gamma} = (\Fitt(\Gamma))^{\frac{1}{m}}$. Hence, $\tilde{\Gamma} = (\Fitt(\Gamma))^{\frac{1}{m}} \cdot C$, and thus, $C$ is a virtually nilpotent supplement for $\textbf{F}$ in $\tilde{\Gamma}$.

If $C$ is any virtually nilpotent supplement for $\Fitt(\tilde{\Gamma})$, we will show there there exists a $\mathbb{Q}$-defined Cartan subgroup $\textbf{C}$ of the $\Q$-algebraic hull $\textbf{H}_\Gamma$ of $\Gamma$ such that $C \leq \tilde{\Gamma} \cap \textbf{C}$. Let $\overline{C}$ be the Zariski closure of $C$, then by definition, $\overline{C}$ is defined over $\mathbb{Q}$. Given that $\tilde{\Gamma} = (\Fitt(\Gamma))^{\frac{1}{m}} \cdot C$ is Zariski dense in $\textbf{H}_\Gamma$, we have that $\textbf{H}_\Gamma = \textbf{F} \cdot \overline{C}$. Let $\textbf{S}$ be a maximal $d$-subgroup of $\overline{C}$, then by the previous equality $\textbf{S}$ is also a maximal $d$-subgroup in $\textbf{H}_\Gamma$. Thus, $\overline{C}$ contains a maximal $\mathbb{Q}$-defined torus $\textbf{T}$ of $\textbf{H}_\Gamma$. Given that $\overline{C}$ is virtually nilpotent, we have that $\textbf{T}$ is unique and hence a normal subgroup of $\overline{C}$. Write $\textbf{C} = N_{\textbf{H}_\Gamma}(\textbf{T})$ for the Cartan subgroup corresponding to $\textbf{T}$. We then have $\overline{C} \leq \textbf{C}$. Therefore, $C \leq \tilde{\Gamma} \cap \textbf{C}$. Hence, if $C$ is a virtually nilpotent supplement for $(\Fitt(\Gamma))^{\frac{1}{m}}$, then there exists a $\mathbb{Q}$-defined Cartan subgroup $\textbf{C} \leq \textbf{H}_\Gamma$ such that $C \leq \tilde{\Gamma} \cap \textbf{C}$. Moreover, the group $\textbf{C}$ is uniquely characterized by $C$. As $\tilde{\Gamma} \cap \textbf{C}$ is itself a virtually nilpotent supplement, the maximal ones are exactly equal to $\tilde{\Gamma} \cap \textbf{C}$. 

We are left to demonstrate that every maximal virtually nilpotent supplement $C$ of $(\Fitt(\Gamma))^{\frac{1}{m}}$ is Zariski dense in the Cartan subgroup $\textbf{C}$ described above. By the previous discussion, we know that a maximal $d$-subgroup $\textbf{S}$ of $\overline{C}$ is also a maximal $d$-subgroup of $\textbf{C}$. Thus, it suffices to show that $\textbf{U}(\overline{C}) = \textbf{U}(\textbf{C})$. 
So take any $u \in \textbf{U}(\textbf{C})$, which we can write as $u = fu_1$ where $u_1 \in \textbf{U}(\overline{C})$ and $f \in \textbf{F} \cap \textbf{C}$. Since $C$ is maximal, we have $\Fitt(\Gamma) \cap \textbf{C} \leq C$. As $\Fitt(\Gamma) \cap \textbf{C}$ is Zariski dense in $ \textbf{F} \cap \textbf{C}$, we get $\textbf{F} \cap \textbf{C} \leq \overline{C}$ which implies $f \in \overline{C}$ and hence also $u \in \overline{C}$. We conclude that $\overline{C} = \textbf{C}$.
\end{proof}
For the remainder of this article, we will be mainly using the existence of a maximal virtually nilpotent supplement for a certain thickening.
\begin{cor}
For every finitely generated virtually torsion-free FAR \emph{WTN}-group $\Gamma$, there exists a thickening of $\Gamma$ whose Fitting subgroup admits a maximal virtually nilpotent supplement.
\end{cor}

Let $G$ be a group containing an abelian normal subgroup $A \trianglelefteq G$. We call a subgroup $H \subset G$ a complement for $A$ in $G$ if $H A = G$ and $H \cap A = \{e\}$. In particular, every complement leads to an isomorphism $G \approx A \rtimes H$ where $H$ acts on $A$ via conjugation. 
\begin{prop}
	\label{prop:segalgeneral}
Let $G$ be a nilpotent FAR group acting on a torsion-free abelian FAR group $A$. If $C_A(G)= 0$, then the number of complements for $A$ in $A \rtimes G$ up to conjugation is finite. 
\end{prop}
\begin{proof}
By \cite[10.3.6, page 216]{lennox_robinson} for $R = \Z$, the cohomology group $H^1(G,A)$ is bounded. Since $A$ has finite rank, this immediately implies that $H^1(G,A)$ is finite, and hence the proposition follows by the one-to-one correspondence between complements up to conjugation and $H^1(G,A)$.
\end{proof}

We end this section by demonstrating that there are at most finitely many $(\Fitt(\Gamma))^{\frac{1}{m}}$-conjugacy classes of virtually nilpotent supplements for an appropriate choice of thickening of $\Gamma$, by first doing the nilpotent case. Recall that an action of a group $\Gamma$ on $H$ is called \emph{nilpotent} if there exists a series of $\Gamma$-invariant normal subgroups $\{H_i\}$ of $H$ such that the induced action of $\Gamma$ on $H_i / H_{i-1}$ is trivial for all $i$. When $H = \Gamma$ and the action is conjugation, we see that $\Gamma$ acts nilpotently on itself if and only if $\Gamma$ is nilpotent.

\begin{prop}\label{finitely_many_nilpotent_complements}
Let $\Gamma$ be a finitely generated FAR \emph{WTN}-group, and let $N$ be a torsion-free nilpotent normal subgroup. There are at most finitely many $N$-conjugacy classes of maximal nilpotent supplements for $N$ in $\Gamma$.
\end{prop}
\begin{proof}
Denote the set of maximal nilpotent supplement for $N$ in $\Gamma$ as $\mathcal{C}$, which we can assume to be non-empty and thus $\Gamma/N$ is nilpotent. If $\Gamma$ acts nilpotently on $N$ by conjugation, then $G$ is itself nilpotent and thus $\mathcal{C} = \{\Gamma\}$. Hence we can assume that $\Gamma$ does not act nilpotently on $N$. 

Let $M = \pi^{-1}(\tau(N/\gamma_2(N)))$ where $\pi \colon N \to N/\gamma_2(N)$ is the natural projection.  We then see  by \cite[Proposition 3, page 49]{segal} that $\Gamma$ does not act nilpotently on $N/M$. Let $K/M$ be the maximal $\Gamma$-submodule of $N/M$ on which $\Gamma$ acts nilpotently. Since $K/M$ is maximal, we must have $C_{N/K}(\Gamma) = \{1\}$, and by \cite[Exercise 3, page 49]{segal} we may assume that $N/K$ has no non-trivial torsion. Hence, $h(K) < h(N)$. 

Now let $C \in \mathcal{C}$, then $C_{C}(N/K) =\{1\}$ as $C \subset \Gamma$. Since $C$ is nilpotent, we have that $C$ acts nilpotently on $(K \cdot C \cap N)/K$ and thus also $\Gamma = N \cdot C$ acts nilpotently $(K \cdot C \cap N)/K$. This implies by the assumption on $K$ that $(K \cdot C) \cap N = K$. Hence, $K \cdot C/K$ is a nilpotent complement for $N/K$ in $\Gamma/K$. Proposition \ref{prop:segalgeneral} then implies that the set of all such complements breaks up into finitely may conjugacy classes in $\Gamma/K$. Hence, there exist subgroups $C_1, \ldots, C_\ell \in \mathcal{C}$ with
$$
	\mathcal{C} = \bigcup_{i=1}^\ell \bigcup_{g \in N} \mathcal{C}_i^g
	$$
	where for $i = 1, \ldots, \ell$ the set $\mathcal{C}_i$ is equal to 
$$
\mathcal{C}_i = \{ C \in \mathcal{C} \: | \: KC = KC_i\}.
$$ 
We see that each member of $\mathcal{C}$ is a maximal nilpotent supplement for $K$ in the group $KC_i$. Since $h(K) < h(N)$, the inductive hypothesis implies the set $\mathcal{C}_i$ is a union of finitely many conjugacy classes of subgroups in $KC_i$. Subsequently, $\mathcal{C}$ is the union of finitely many conjugacy classes of subgroups.
\end{proof}

\begin{prop}\label{finitely_many_thickenings}
Let $\Gamma$ be a finitely generated FAR \emph{WTN}-group, and let $\tilde{\Gamma} = \Fitt(\Gamma)^{\frac{1}{m}} \cdot \Gamma$ be a thickening of $\Gamma$ as in Proposition \ref{supplement_thickening}. Then there are at most finitely many $\Fitt(\Gamma)^{\frac{1}{m}}$-conjugacy classes of maximal virtually nilpotent supplements in $\Gamma$.
\end{prop}
\begin{proof}
Let $C$ be a maximal virtually nilpotent supplement in $\tilde{\Gamma}$. Proposition \ref{supplement_thickening}  implies that $C = \tilde{\Gamma} \cap \textbf{C}$ where $\textbf{C}$ is a Cartan subgroup of $\textbf{H}_\Gamma$. Now consider $\tilde{\Gamma}_0 = \tilde{\Gamma} \cap \textbf{H}_\Gamma^\circ$ and $C_0 =C \cap \tilde{\Gamma}_0$. We see that $\tilde{\Gamma}_0$ is a finite index normal subgroup of $\tilde{\Gamma}$ where $\tau(\tilde{\Gamma}_0) = 1$ and that $C_0 \trianglelefteq C$.  Given that $\textbf{C}^\circ$ is a Cartan subgroup in $\textbf{H}_\Gamma^\circ$, Proposition \ref{supplement_thickening} implies that $C_0$ is a maximal nilpotent supplement in $\tilde{\Gamma}_0$.  We also see that $C = \tilde{\Gamma} \cap N_{\Gamma \cdot \textbf{F}}(C_0)$ is uniquely determined by $C_0$. By Proposition \ref{finitely_many_nilpotent_complements} applied to $\Fitt(\tilde{\Gamma})$, there are only finitely many $\Fitt(\Gamma)^{\frac{1}{m}}$-conjugacy classes of maximal nilpotent supplements $C_0 \leq \tilde{\Gamma}_0$. Hence, we have only finitely many $\Fitt(\Gamma)^{\frac{1}{m}}$-conjugacy classes of maximal virtually nilpotent supplements in $\tilde{\Gamma}$.
\end{proof}

\subsection{Unipotent shadows of solvable groups of finite abelian ranks}

Let $\textbf{C} = \overline{C} \leq \textbf{H}_\Gamma$ denote the Zariski closure of $C$. By Proposition \ref{supplement_thickening}, the group $\textbf{C}$ is a $\mathbb{Q}$-defined Cartan subgroup of $\textbf{H}_\Gamma$ and $C = \tilde{\Gamma} \cap \textbf{C}$. Now let $\textbf{S} \leq \textbf{C}$ be a maximal $\mathbb{Q}$-defined $d$-subgroup, which contains the torus $T = \textbf{S}^\circ$ that is central in $\textbf{C}^\circ$. In fact, $\textbf{C} = N_{\textbf{H}_\Gamma}(T)$. We now set $C_0 = C \cap \textbf{C}^\circ$ which is a finite index nilpotent normal subgroup of $C$.

Consider the splittings $\textbf{C} = \textbf{U}(\textbf{C}) \cdot \textbf{S}$ and $\textbf{C}^\circ = \textbf{U}(\textbf{C}) \cdot \textbf{S}^\circ.$ We can then write each $c \in \textbf{C}$ uniquely as $c = u_c  s_c$ where $u_c \in \textbf{U}(\textbf{C})$ and $s_c \in \textbf{S}$. We note that if $c \in \textbf{C}^\circ$, then $s_c \in \textbf{S}^\circ$. Therefore, we define
$$
U_{C, \textbf{S}} = \left< u_c \: | \: c \in C \right> \quad \text{ and } \quad U_{C_0} = \left< u_c \: | \: c \in C_0 \right>,
$$
where the latter does not depend on $\textbf{S}$, which is also clear from the proof below. 

\begin{lemma}\label{finite_generator_piece_unipotent_shadow}
The subgroup $U_{C_0} \leq U_{C, \textbf{S}}$ is of finite index and both are Zariski-dense in $\textbf{U}(\textbf{C})$. Finally, $U_{C, \textbf{S}}$ is normalized by $C$ and $U_{C_0}$ is centralized by $C_0$.
\end{lemma}
\begin{proof}
As $\textbf{C}^\circ$ is nilpotent, the map $\psi \colon C_0 \to \textbf{U}(\textbf{C})$ given by $\psi(c) = u_c$ is a group morphism. Since $\textbf{U}(\textbf{C}) = \textbf{U}(\textbf{C}^\circ) = \textbf{U}(\overline{C_0})$, it follows that $U_{C_0}$ is Zariski dense in $\textbf{U}(\textbf{C})$, and hence also the larger group $U_{C,\textbf{S}}$.

Let $\pi \colon C \to \textbf{S}$ be the group morphism given by $\pi(c) = s_c$ and write $T = \pi(C) \leq \textbf{S}$ and $T_0 = \pi(C_0) \leq \textbf{S}$. A direct computation shows that for all $c,d \in C$, it holds that
\begin{equation}\label{finite_generator_piece_unipotent_shadow_formula}
	u_{cd} = u_c s_c u_d s_c^{-1}.
\end{equation} The group $T$ acts on $\textbf{U}(\textbf{C})$ by conjugation. Since $\textbf{S}^\circ$ is central in $\textbf{C}^\circ$, the action by conjugation factors through the finite group $T / T_0$, also implying the last statement. Equation (\ref{finite_generator_piece_unipotent_shadow_formula}) shows that the induced action of $T / T_0$ on $\textbf{U}(\textbf{C})$ preserves the subgroup $U_{C, \textbf{S}}$. Hence, $U_{C, \textbf{S}}$ is normalized by the group $C$. The finite index condition also follows from Equation (\ref{finite_generator_piece_unipotent_shadow_formula}).
\end{proof}

\begin{defn}\label{unipotent_shadow_defn}
Following the notation of \cite[Definition 7.3]{baues_grunewald}, we define
$$
\theta_{C, \textbf{S}} = \left<(\Fitt(\Gamma))^{\frac{1}{m}}, U_{C, \textbf{S}} \right> \quad \text{ and } \quad \theta_{C_0} = \left<(\Fitt(\Gamma))^{\frac{1}{m}}, U_{C_0} \right>.
$$
We call the groups $\theta_{C, \textbf{S}}$ \emph{unipotent shadows} of $\Gamma$.
\end{defn}

By Lemma \ref{finite_generator_piece_unipotent_shadow} we know that $\theta_{C, \textbf{S}}$ is a Zariski dense subgroup of $\textbf{U}$
containing $\theta_{C_0}$ as a normal subgroup of finite index. However, we sometimes need the following additional property:

\begin{defn}
We call $\theta_{C, \textbf{S}}$ a \emph{good unipotent shadow} if the following conditions hold:
$$
\theta_{C, \textbf{S}} \cap \textbf{F} = (\Fitt(\Gamma))^{\frac{1}{m}} = \Fitt(\tilde{\Gamma}).
$$
\end{defn}

We may find good unipotent shadows if we take another thickening of the group $\Gamma$.

\begin{prop}\label{thickening_unipotent_shadow_strong}
Let $\Gamma$ be a finitely generated FAR \emph{WTN}-group. Then there is a thickening $\tilde{\Gamma} = (\Fitt(\Gamma))^{\frac{1}{m}} \cdot \Gamma$ with a maximal virtually nilpotent supplement $C \leq \tilde{\Gamma}$ and a maximal $\mathbb{Q}$-defined $d$-subgroup $\textbf{S} \leq \overline{C}$ such that the subgroup $\theta_{C, \textbf{S}}$ is a good unipotent shadow.
\end{prop}
\begin{proof}
We starting by choosing a $k \in \mathbb{N}$ such that the thickening $(\Fitt(\Gamma))^{\frac{1}{k}} \cdot \Gamma$ admits a maximal virtually nilpotent supplement $C$ which is guaranteed by Proposition \ref{supplement_thickening}. Let $\theta_{C, \textbf{S}}$ be given as in Definition \ref{unipotent_shadow_defn}, and let $\textbf{S}$ be a maximal $\mathbb{Q}$-defined $d$-subgroup in the Zariski-closure $\textbf{C}$ of $C$.

The subgroup $\theta_{C, \textbf{S}}$  is Zariski dense in $\textbf{U}(\textbf{H}_\Gamma)$. Note that $(\Fitt(\Gamma))^{\frac{1}{k}}$ forms a normal subgroup of $\theta_{C, \textbf{S}}$. Moreover, the group $\theta_{C_0}/(\Fitt(\Gamma))^{\frac{1}{k}}$ and hence $\theta_{C, \textbf{S}}/(\Fitt(\Gamma))^{\frac{1}{k}}$ is finitely generated, as the image of $\Gamma \cap \textbf{H}^\circ_\Gamma$ under the projection map $\pi_u: \textbf{H}^\circ_\Gamma/\textbf{F} \to \textbf{U}(\textbf{H}_\Gamma) / \textbf{F}$. In particular, there exists a $m \in \mathbb{N}$ which is divisible by $k$ such that $\theta_{C, \textbf{S}} \cap \textbf{F} \leq (\Fitt(\Gamma))^{\frac{1}{m}}$, exactly as in the proof of Proposition \ref{supplement_thickening}. Now take $\tilde{\Gamma} = (\Fitt(\Gamma))^{\frac{1}{m}} \cdot \Gamma$, of which we will show that it satisfies the conditions of the proposition. 

Note that $C$ is a virtually nilpotent supplement in $\tilde{\Gamma}$ and hence Proposition \ref{supplement_thickening} implies that $C \subset C_1 = \textbf{C} \cap \tilde{\Gamma}$ with $C_1$ a maximal virtually nilpotent supplement. Every element $c \in C_1$ can be written as $c_1 = hc$ where $c \in C$ and $h \in (\Fitt(\Gamma))^{\frac{1}{m}}$. This implies that $\theta_{C_1,\textbf{S}} \cap \textbf{F} = (\theta_{C,\textbf{S}} \cap (\Fitt(\Gamma))^{\frac{1}{m}})(\Fitt(\Gamma))^{\frac{1}{m}}$. By construction of $m$ we know that this equals $(\Fitt(\Gamma))^{\frac{1}{m}}$, which exactly means that $\theta_{C_1,\textbf{S}}$ is a good unipotent shadow.
\end{proof}

We finish with the following compatibility results that shows the interaction between automorphisms of a thickening and the strong unipotent shadow.

\begin{prop}\label{prop7.5_baues_grunewald}
Let $\Gamma$ be a finitely generated virtually torsion-free FAR \emph{WTN}-group. Let $\tilde{\Gamma} = (\Fitt(\Gamma))^{\frac{1}{m}} \cdot \Gamma$ be a thickening of $\Gamma$ with a virtually nilpotent supplement $C \leq \tilde{\Gamma}$ for $\Fitt(\tilde{\Gamma})$ and corresponding strong unipotent shadow $\theta_{C, \textbf{S}}$. We then have the following:
\begin{enumerate}[(1)]
\item Let $\varphi \in \Aut(\tilde{\Gamma})$ satisfy $\varphi(C) = C$, and suppose that $\Psi$ is the extension of $\varphi$ to an automorphism of $\textbf{H}_\Gamma$. We then have $\Psi(\theta_{C_0}) = \theta_{C_0}$.

\item There exists a finite index subgroup of the group of all automorphisms $\varphi \in \Aut(\tilde{\Gamma})$ with $\varphi(C) = C$ such that the extension $\Psi$ to $\textbf{H}_\Gamma$ satisfies $\Psi(\theta_{C, \textbf{S}}) = \theta_{C, \textbf{S}}$. 

\item The group $\tilde{\Gamma}$ normalizes $\theta_{C, \textbf{S}}$.
\end{enumerate}
\end{prop}
\begin{proof}
For part (1), we note that $\Psi(\overline{C}) = \overline{\Psi(C)}$. Since $\textbf{C} = \overline{C}$, we have $\Psi(\textbf{C}) = \textbf{C}$. Since algebraic automorphisms preserve the connected component of the identity, we have $\Psi(\textbf{C}^\circ) = \textbf{C}^\circ$. We may write $\textbf{C}^\circ = \textbf{U}(\textbf{C}) \cdot \textbf{S}^\circ$. Since $\Psi$ sends unipotents to unipotents, we have $\Psi(\textbf{U}(\textbf{C})) = \textbf{U}(\textbf{C})$. Hence, we must have that $U_{C_0}$ is preserved by $\Psi$. Since $\Psi$ also stabilizes $(\Fitt(\Gamma))^{\frac{1}{m}}$, we are done with (1).

For part (2), we let 
$$
K = \{\varphi \in \Aut(\tilde{\Gamma}) \: | \: \varphi(C) = C\},
$$
and note that $\theta_{C_0}$ is finite index in $\theta_{C, \textbf{S}}$ because $[U_{C, \textbf{S}} : U_{C_0}] < \infty$. Hence there exists a natural number $d$ such that $U_{C, \textbf{S}} \leq (U_{C_0})^{\frac{1}{d}}$, and moreover, Lemma \ref{thickening_index} implies that $[(U_{C_0})^{\frac{1}{d}} : U_{C_0}] < \infty$. Hence, $U_{C_0} \leq U_{C, \textbf{S}} \leq (U_{C_0})^{\frac{1}{d}}$. For every $\varphi \in K$, part (1) implies that $\varphi(U_{C_0}) = U_{C_0} $ and hence also $\varphi((U_{C_0})^{\frac{1}{d}}) = (U_{C_0})^{\frac{1}{d}}$. Thus, $K$ acts on the set of groups $H$ such that $U_{C_0} \leq H \leq (U_{C_0})^{\frac{1}{d}}$, which is a finite set. Hence, the stabilizer of $U_{C, \textbf{S}}$ will be a finite index subgroup of $K$. 

For part (3), we see that $U_{C, \textbf{S}}$ is normalized by $C$ by Lemma \ref{finite_generator_piece_unipotent_shadow}. Since we also have that $(\Fitt(\Gamma))^{\frac{1}{m}}$ is normalized by $C$, it then follows $C$ normalizes $\theta_{C, \textbf{S}} = (\Fitt(\Gamma))^{\frac{1}{m}} \cdot U_{C, \textbf{S}}$. Since $(\Fitt(\Gamma))^{\frac{1}{m}} \leq \theta_{C, \textbf{S}}$, we see that $(\Fitt(\Gamma))^{\frac{1}{m}}$ normalizes $\theta_{C, \textbf{S}}$. Because $\tilde{\Gamma} = (\Fitt(\Gamma))^{\frac{1}{m}} \cdot C$, it then follows that $\tilde{\Gamma}$ normalizes $\theta_{C, \textbf{S}}$.
\end{proof}

\section{$S$-arithmetic subgroups of $\Aut(\Gamma)$}\label{section_s_arithmetic_subgroups}

In this section, we construct certain $S$-arithmetic subgroups of $\Aut(\Gamma)$ by embedding it in the group $\Aut({\textbf{H}_\Gamma})$. We start by recalling how to realize the latter as an algebraic group.

\subsection{The algebraic group $\Aut(\textbf{H})$}
\label{closed_subgroups_Aut}

We first describe properties of the automorphism group of virtually solvable linear algebraic groups that have a strong unipotent radical. This section is a recollection of results in \cite{baues_grunewald}. From now on, let $\textbf{H}$ be a virtually solvable linear algebraic group defined over $\mathbb{Q}$. We define $\Aut(\textbf{H})$ as the group of algebraic automorphisms of $\textbf{H}$ defined over $\Q$. We will always assume that $\textbf{H}$ has a strong unipotent radical $\textbf{U} = \textbf{U}(\textbf{H})$, meaning that the centralizer of $\textbf{U}$ lies in $\textbf{U}$ itself.

Let $\mathfrak{u}$ be the Lie algebra of the unipotent radical $\textbf{U}$. We have that $\Aut(\mathfrak{u}) \leq \GL(\mathfrak{u})$ is a $\mathbb{Q}$-defined linear algebraic group. It then follows that the exponential map $\exp \colon \mathfrak{u} \to \textbf{U}$ is an algebraic isomorphism of varieties, where $\mathfrak{u}$ gets its structure as a $\mathbb{Q}$-defined vector space.  Via this isomorphism that sends $\Psi \in \Aut(\textbf{U})$ to $\exp^{-1} \circ \Psi \circ \exp \in \Aut(\mathfrak{u})$, we see that $\Aut(\textbf{U})$ has a natural $\mathbb{Q}$-defined linear algebraic group structure.

There exists a maximal $\mathbb{Q}$-defined closed d-group $\textbf{S} \leq \textbf{H}$ so that $\textbf{H} = \textbf{U} \cdot \textbf{S}$. Fix a choice of $\textbf{S}$ from now on, and define
$$
\Aut(\textbf{H})_{\textbf{S}} = \{\Psi \in \Aut(\textbf{H}) \: | \: \Psi(\textbf{S}) = \textbf{S}\}.
$$
By \cite[Lemma 3.2]{baues_grunewald} the restriction map $\Psi \to \Psi|_{\textbf{U}}$ identifies $\Aut(\textbf{H})_{\textbf{S}}$ with a $\mathbb{Q}$-closed subgroup of $\Aut(\textbf{U})$. Since $\Aut(\textbf{H})_{\textbf{S}}$ acts on $\textbf{U}$ by algebraic automorphisms, the semidirect product $\textbf{U} \rtimes \Aut(\textbf{H})_{\textbf{S}}$ is a $\mathbb{Q}$-defined linear algebraic group. For any $u \in \textbf{U}$, we write $\Inn_u(h) = uhu^{-1}$ for all $h \in \textbf{H}$. Letting 
\begin{equation*}
	\Theta \colon  \textbf{U} \rtimes \Aut(\textbf{H})_{\textbf{S}} \to \Aut(\textbf{H})
\end{equation*}
be given by 
\begin{equation}\label{theta_map}
	\Theta(u, \Psi) = \Inn_u \circ \Psi,
\end{equation}
\cite[Lemma 3.3 and 3.4]{baues_grunewald} then imply that $\Theta$ is surjective with a $\mathbb{Q}$-defined unipotent subgroup of $\textbf{U} \rtimes \Aut(\textbf{H})_{\textbf{S}}$ as its kernel. We write
$$
\textbf{A}_{\textbf{H}} = (\textbf{U} \rtimes \Aut(\textbf{H})_{\textbf{S}}) / \ker \Theta
$$ for the natural quotient. In particular, \cite[Theorem 3.5]{baues_grunewald} implies there is an induced isomorphism of $\mathbb{Q}$-defined linear algebraic groups $\bar{\Theta} \colon \textbf{A}_{\textbf{H}} \to \Aut(\textbf{H})$. 

Now take $\textbf{F} \leq \textbf{U}$ a $\mathbb{Q}$-closed normal subgroup of $\textbf{H}$ such that $[\textbf{H}^\circ, \textbf{H}^\circ] \leq \textbf{F}$. We then see by \textbf{SG5} that $\textbf{H} \cong \textbf{F} \cdot \textbf{C}$ for every $\mathbb{Q}$-closed Cartan subgroup $\textbf{C}$ of $\textbf{H}$.  We let $N_{\Aut(\textbf{H})}(\textbf{F})$ to denote the subgroup of elements in $\Aut(\textbf{H})$ that preserve $\textbf{F}$. We also define
\begin{align*}
	\textbf{A}_{\textbf{H}| \textbf{F}}& = \left\{ \Psi \in N_{\Aut(\textbf{H})}(\textbf{F}) \mid \Psi|_{\textbf{H} / \textbf{F}} = \text{id}_{\textbf{H} / \textbf{F}} \right\},\\
	\textbf{A}_{\textbf{S}} &= \left\{ \Psi \in N_{\Aut(\textbf{H})}(\textbf{F}) \mid  \Psi|_{\textbf{H}^\circ / \textbf{F}} = \text{id}_{\textbf{H}^\circ / \textbf{F}}, \Psi(\textbf{S}) = \textbf{S} \right\},\\
	\textbf{A}_{\textbf{S}}^1 &= \left\{ \Psi \in \textbf{A}_{\textbf{S}} \mid  \Psi|_{\textbf{S}} = \text{id}_{\textbf{S}}\right\}.
\end{align*}
One can see that all of these groups are $\mathbb{Q}$-closed subgroups of $\Aut(\textbf{H})$.

For the following proposition, we introduce some notation. Let $A$ be any group containing $B$ as a subgroup. For $a \in A$, we write $\Inn_a \in \Aut(A)$ for the inner automorphism of $A$ given by $\Inn_a(g) = aga^{-1}$. Given an element $b \in B$, we set $\Inn_b^A \in \Aut(A)$ for the corresponding inner automorphism of $A$ which we to distinguish it from the induced inner automorphism  of $B$. We write $\Inn_B^{\textbf{A}}$ for the subgroup of $\Aut(A)$ consisting of all elements $\Inn_b^A$ for $b \in B$. With these definitions in mind, we finish with \cite[Proposition 3.13]{baues_grunewald} which is stated as follows.
\begin{prop}\label{linear_alg_structure_A_HF}
	Let $\textbf{H}$ be a $\mathbb{Q}$-defined solvable linear algebraic group with unipotent radical $\textbf{U}$. Let $\textbf{F} \leq \textbf{U}$ be a $\mathbb{Q}$-closed subgroup which contains $[\textbf{H}^\circ, \textbf{H}^\circ]$. Then
	$$
	\textbf{A}_{\textbf{H} | \textbf{F}} = \Inn_{\textbf{F}}^{\textbf{H}} \cdot \textbf{A}_{\textbf{S}}^1.
	$$
\end{prop}

\subsection{$S$-arithmetic subgroups of $\Aut(\textbf{H}_\Gamma)$}

Let $\Gamma$ be a finitely generated virtually torsion-free FAR \emph{WTN}-group with spectrum $S$ and notations as before.
By replacing $\Gamma$ by a thickening if necessary, we may assume that $\Gamma$ admits a maximal virtually nilpotent supplement $C$ to $\Fitt(\Gamma)$. 
 We choose a $\mathbb{Q}$-defined $d$-subgroup $\textbf{S} \leq \textbf{C}$ with an associated unipotent shadow $\theta$ as we constructed in the previous section. In particular, we may assume the given unipotent shadow is good. 

In this section, we demonstrate in the given setup that $A_{\Gamma \: | \: \Fitt(\Gamma)}$ is $S$-arithmetic. Therefore, let $\textbf{U} = \textbf{U}(\textbf{H}_\Gamma)$. The unipotent shadow $\theta \leq \textbf{U}$ gives us $S$-arithmetic subgroups of particular $\mathbb{Q}$-closed subgroups of $\Aut(\textbf{H}_\Gamma)$. Using these groups and notions, we will then demonstrate how $\Aut(\Gamma)$ sits in its Zariski closure in $\Aut(\textbf{H}_\Gamma)$. Towards this end, for a subgroup $B \leq \Aut(\textbf{H}_\Gamma)$, we define
$$
B[\theta] = \left\{ \Psi \in B \: |  \: \Psi(\theta) = \theta \right\},
$$
which is the stabilizer of $\theta$ in the group $B$.
This leads to the first construction of $S$-arithmetic subgroups of the automorphism group. 

\begin{prop}\label{prop8.2_baues_grunewald}
Let the data $(\Gamma, C, \textbf{S})$ be defined as above, and suppose that $\Fitt(\Gamma)$ is a $S$-arithmetic lattice in its Zariski closure in $\textbf{H}_\Gamma$. We have
$$
\Inn_{\Fitt(\Gamma)}^{\textbf{H}_\Gamma} \cdot \textbf{A}_{\textbf{S}}^1[\theta] \leq \textbf{A}_{\textbf{H}_\Gamma | \textbf{F}}
$$
is a $S$-arithmetic subgroup of $\textbf{A}_{\textbf{H}_\Gamma | \textbf{F}}$.
\end{prop}
\begin{proof}
We start by demonstrating that $\textbf{A}_{\textbf{S}}^1[\theta]$ is $S$-arithmetic in $\textbf{A}_{\textbf{S}}^1$. 
Given that $\textbf{A}_{\textbf{S}}^1$ acts on $\textbf{U} = \textbf{U}(\textbf{H}_\Gamma)$ by automorphisms, we get an induced action on the Lie algebra of $\textbf{U}$, denoted $\mathfrak{u}$, via $\exp^{-1} \circ \Psi \circ \exp$ where $\Psi \in \textbf{A}_{\textbf{S}}^1$. Embedding $\textbf{A}_{\textbf{S}}^1[\theta]$ as a subgroup of $\Aut(\Fitt(\Gamma))$, we have that $\textbf{A}_{\textbf{S}}^1[\theta]$ corresponds to
$$
N_{ \textbf{A}_{\textbf{S}}^1}(\log ( \theta)) = \{ \Psi \in \textbf{A}_{\textbf{S}}^1 \: | \: \Psi(\log(\theta)) = \log (\theta) \}.
$$
By our given definitions, $\textbf{A}_{\textbf{S}}^1[\theta]$ normalizes both $\textbf{F}$ and $\theta$. In particular, this group also normalizes $\textbf{F} \cap \theta$. Since $\theta$ is a good unipotent shadow, we have $\textbf{F} \cap \theta = \Fitt(\Gamma)$, which is $S$-arithmetic in $\textbf{F}$. Moreover, $\theta/ \Fitt(\Gamma)$ is finitely generated and Zariski-dense in the abelian group $\textbf{U} / \textbf{F}$, where $\textbf{A}_{\textbf{S}}^1$ acts as identity. It then follows that $N_{ \textbf{A}_{\textbf{S}}^1}(\log ( \theta))$ is $S$-arithmetic in $\textbf{A}_{\textbf{S}}^1$ as it is commensurable to $\textbf{A}^1_{\textbf{S}}\left(\Z[\frac{1}{S}]\right).$

Since $\Fitt(\Gamma)$ is $S$-arithmetic in $\textbf{F}$, we see that
$$
\Fitt(\Gamma) \rtimes \textbf{A}_{\textbf{S}}^1[\theta] \leq \textbf{F} \rtimes \textbf{A}_{\textbf{S}}^1
$$
is a $S$-arithmetic group. Let $\Theta \colon \textbf{F} \rtimes \textbf{A}_{\textbf{S}}^1 \to \textbf{A}_{\textbf{H}_\Gamma \: | \: \textbf{F}}$ be the algebraic morphism induced by Equation (\ref{theta_map}). Proposition \ref{linear_alg_structure_A_HF} then implies it is surjective. Since $\Inn_{\Fitt(\Gamma)}^{\textbf{H}_\Gamma} \cdot \textbf{A}_{\textbf{S}}^1[\theta]$ is the image of a $S$-arithmetic group under an algebraic morphism, it is $S$-arithmetic as desired.
\end{proof}

Let us now compare $\Aut(\Gamma)$ to the above $S$-arithmetic group. As in the introduction, we define
$$
A_{\Gamma  |  \Fitt(\Gamma)} = \left\{ \varphi \in \Aut(\Gamma)  \: | \:  \varphi_{\Gamma / \Fitt(\Gamma)} = \text{id}_{\Gamma / \Fitt(\Gamma)} \right\}.
$$
We then see that $A_{\Gamma  |  \Fitt(\Gamma)}$ is a characteristic subgroup of $\Aut(\Gamma)$. We also set
$$
A_{\Gamma  | \Fitt(\Gamma)}^{C} = \left\{ \varphi \in A_{\Gamma  | \Fitt(\Gamma)} \: | \: \varphi(C) = C\right\}.
$$

We now show that $\Inn_{\Fitt(\Gamma)}^\Gamma \cdot A_{\Gamma  | \Fitt(\Gamma)}^C$ has finite index in $A_{\Gamma  |  \Fitt(\Gamma)}$.
\begin{lemma}\label{inn_and_fix_c_finite_agf}
Let the data $(\Gamma, C, \textbf{S})$ be defined as above, then $\Inn_{\Fitt(\Gamma)}^{\Gamma} \cdot A_{\Gamma  | \Fitt(\Gamma)}^C$ has finite index in $A_{\Gamma  |  \Fitt(\Gamma)}$.
\end{lemma}
\begin{proof}
We see that $A_{\Gamma  |  \Fitt(\Gamma)}$ acts on the set of $\Fitt(\Gamma)$-conjugacy classes of virtually nilpotent supplements of $\Fitt(\Gamma)$. Proposition \ref{finitely_many_thickenings} implies that there are finitely many maximal $\Fitt(\Gamma)$-conjugacy classes of virtually nilpotent supplements $C$ for $\Fitt(\Gamma)$. In particular, the kernel of this action is a finite index subgroup of $A_{\Gamma  | \Fitt(\Gamma)}$, which moreover is contained in $\Inn_{\Fitt(\Gamma)}^\Gamma \cdot A^C_{\Gamma | \Fitt(\Gamma)}$. Hence, $\Inn_{\Fitt(\Gamma)}^\Gamma \cdot A_{\Gamma |  \Fitt(\Gamma)}$ is a finite index subgroup of $A_{\Gamma | \Fitt(\Gamma)}$.
\end{proof}

We now explore the relationship between the groups $A_{\Gamma  |  \Fitt(\Gamma)}^C[\theta]$ and $A_{\Gamma  | \Fitt(\Gamma)}$.

\begin{lemma}\label{lemma8.4_baues_grunewald}
Let the data $(\Gamma, C, \textbf{S})$ be described as above, then $\Inn_{\Fitt(\Gamma)}^\Gamma \cdot A_{\Gamma  |  \Fitt(\Gamma)}^C[\theta]$ has finite index in $A_{\Gamma  |  \Fitt(\Gamma)}$.
\end{lemma}
\begin{proof}
This is immediate from Lemma \ref{inn_and_fix_c_finite_agf} and Proposition \ref{prop7.5_baues_grunewald}.
\end{proof}

We note that there exists a semidirect product decomposition of $\textbf{C}$ written as $\textbf{C} = \textbf{U}(\textbf{C}) \cdot \textbf{S}$ with associated projection morphism $\pi_{\textbf{S}} \colon \textbf{C} \to \textbf{S}$. We define
$$
T = \pi_{\textbf{S}}(C) \quad \text{ and } \quad T_0 = \pi_{\textbf{S}}(C \cap \textbf{C}^\circ).
$$
Recall that $\theta = \theta_{C, \textbf{S}} = \left< \Fitt(\Gamma), U_{C, \textbf{S}} \right>$ where $U_{C, \textbf{S}} = \left\{ u_c \: | \: c \in C\right\}$. Every $c \in C$ can be written uniquely as $c= u_c  \pi_\textbf{S}(c)$ and thus $c \: \pi_{\textbf{S}}(c)^{-1} = u_c \in \theta.$

We now consider the following observations.
\begin{lemma}\label{baums_grunewald8.5}
Let the data $(\Gamma, C, \textbf{S})$, $T$, and $T_0$ be described as above. We then have that:
\begin{enumerate}[(1)]
\item The group $T$ normalizes $\Fitt(\Gamma) \cap C$ and $T_0$ centralizes $\Fitt(\Gamma) \cap C$;
\item If $\Psi$ is the extension of the automorphism $\varphi \in A_{\Gamma | \Fitt(\Gamma)}^C[\theta]$ to an automorphism of $\textbf{H}_\Gamma$, then $\Psi(t) t^{-1} \in \Fitt(\Gamma) \cap C$ for each $t \in T$. Moreover,  if $t \in T_0$, then $\Psi(t) t^{-1} = 1$.
\end{enumerate} 
\end{lemma}
\begin{proof}
Let $t \in T$ and take $c \in C$ such that $\pi_{\textbf{S}}(c) = t$. Write $d= u_c = c \: \pi_{\textbf{S}}(c)^{-1} = c t^{-1}\in U_{C,\textbf{S}} \subset\theta$. Since $\Fitt(\Gamma) \cap C$ is normal in $C$, the element $c$ normalizes $\Fitt(\Gamma) \cap C$. Hence,
$$
t(\Fitt(\Gamma) \cap C) t^{-1} = t c^{-1} ( \Fitt(\Gamma) \cap C)c t^{-1} = d^{-1} (\Fitt(\Gamma) \cap C) d.
$$
Since $d \in \theta$, we must have $d^{-1} (\Fitt(\Gamma) \cap C) d \leq \theta$. We also have that $d^{-1} (\Fitt(\Gamma) \cap C) d \leq \textbf{F}$ where $\textbf{F}$ is the Zariski closure of $\Fitt(\Gamma)$ in $\textbf{H}_\Gamma$. Hence, $d^{-1} (\Fitt(\Gamma) \cap C) d \leq \theta \cap \textbf{F} = \Fitt(\Gamma)$. Given that $d^{-1}(\Fitt(\Gamma) \cap C) d \leq \textbf{C}$, we then have $d^{-1}(\Fitt(\Gamma) \cap C) d \leq D = \Fitt(\Gamma) \cap \textbf{C}$. We see that $D$ is virtually nilpotent and is normalized by $C$. Hence, $\left<C, D \right> \leq \Gamma$ is a virtually nilpotent supplement of $\Fitt(\Gamma)$ in $\Gamma$. Since $C$ was chosen to be maximal, we must have $D = \Fitt(\Gamma) \cap \textbf{C} \leq C$ which gives the first part. For the second part, we have that $\Fitt(\Gamma) \cap C \leq \textbf{C}^\circ$ which is abelian, hence we are done with part (1).

Let $\varphi$ and $\Psi$ be as in part (2). Taking notations as in the first part, we have
\begin{equation}\label{eqn_8.5}
\Psi(d)^{-1} \varphi(c)c^{-1} d = \Psi(t) t^{-1}.
\end{equation}
By the assumption on $\varphi$ we know that $\varphi(c)c^{-1} \in \Fitt(\Gamma)$. Since $\varphi(c) c^{-1} \in C$ as well, we have $\varphi(c)c^{-1} \in \Fitt(\Gamma) \cap C \leq \theta$. Since $\varphi \in A_{\Gamma | \Fitt(\Gamma)}^C[\theta]$, we have by Proposition \ref{prop7.5_baues_grunewald}  that $\Psi(d) \in \theta$ for $d \in \theta$. We note that $\Psi = \text{id}_{\textbf{F}} \text{ mod } \textbf{F}$. Hence, $\Psi(t)t^{-1} \in \textbf{F}$. Therefore, $\Psi(t)t^{-1} \in \theta \cap \textbf{F}$. Since $\theta$ is a good unipotent shadow, we have $\Psi(t)t^{-1} \in \Fitt(\Gamma)$. As clearly also $\Psi(t) t^{-1} \in \textbf{C}$, we finish by noting that $\Fitt(\Gamma) \cap \textbf{C}$ is contained in $C$. For the second part, we note that if $s \in \textbf{S}^\circ$, then $\Psi(s)s^{-1} \in \textbf{F} \cap \textbf{S}^\circ = \{1\}$.
\end{proof}

We need the following lemma which shows that $H^1(T,\Gamma)$ is finite where $T$ is a finite group and $\Gamma$ is a torsion-free, polyrational nilpotent group on which $T$ acts by automorphisms.
\begin{lemma}\label{lemma_6.1_BG}
Let $T$ be a finite group and $\Gamma$ be a torsion-free poly-(cyclic or quasicyclic) nilpotent group on which $T$ acts by automorphisms. Then $H^1(T,\Gamma)$ is finite.
\end{lemma}
\begin{proof}
Let $\Delta = \Gamma \rtimes T$ be the split extension which corresponds to the given action of $T$ on $\Delta$ where we consider $\Gamma$ as a normal subgroup of $\Delta$. We see that a $1$-cocycle $\zeta \colon T \to \Gamma$ allows us to define a finite subgroup $T_\zeta = \{(\zeta(t),t) \: | \: t \in T\}$. We then see that two $1$-cocycles $\zeta_1$ and $\zeta_2$ are cohomologous if and only if $T_{\zeta_1}$ and $T_{\zeta_2}$ are conjugate by an element of $\Gamma.$ By assumption, the group $\Gamma$ admits a series of finite length whose factors are either infinite cyclic or locally finite. In particular, we see that $\Delta$ admits a series of finite length whose factors are either infinite cyclic or locally finite, and thus \cite[10.1.17]{lennox_robinson} shows that the maximal torsion subgroup of $\Delta$ fall into finitely many conjugacy classes. We claim that groups of the form $T_{\zeta}$ are maximal in $\Delta$, and towards this end, suppose that $T^\prime$ is a subgroup of $\Delta$ containing $T_\zeta$. Since $\Gamma$ is torsion-free, we see that if $\pi \colon \Delta \to T$ is the natural projection, then $\ker \pi \cap T^\prime = \{1\}$. Hence, $\pi|_{T^\prime}$ is an isomorphism onto its image. Since $|T_\zeta| \leq |T^\prime| \leq |T_\zeta|$ and $T_\zeta \leq T$, we have $T^\prime \cong T_\zeta$. Hence, $T_\zeta$ is maximal. We then see that there are at most finitely many $\Delta$-conjugacy classes of $T_\zeta$ for any $1$-cocycle $\zeta$. Since $\Gamma$ has finite index in $\Delta,$ there are at most finitely many $\Gamma$-conjugacy classes of groups of the form $T_\zeta$ giving our result.
\end{proof}

We let
$$
A_{\Gamma  | \Fitt(\Gamma)}^C[\theta]^1 = \left\{ \varphi \in A_{\Gamma  |  \Fitt(\Gamma)}^C[\theta] \: | \: \Psi(\textbf{S}) = \textbf{S}, \Psi|_{\textbf{S}} = \text{id}_{\textbf{S}}  \right\}.
$$
In the above definition, $\Psi$ is as always the extension of the automorphism $\varphi \in \Aut(\Gamma)$ to an automorphism of $\textbf{H}_\Gamma$.

\begin{lemma}\label{lemma_8.5_baues_grunewald}
Let the data $(\Gamma, C, \textbf{S})$ be described as above, then $\Inn^\Gamma_{\Fitt(\Gamma) \cap C} \cdot A_{\Gamma  |  \Fitt(\Gamma)}^C[\theta]^1$ is finite index in $A_{\Gamma  |  \Fitt(\Gamma)}^C[\theta]$.
\end{lemma}
\begin{proof}
It is clear that $\Inn_{\Fitt(\Gamma) \cap C}$ is contained in $A_{\Gamma  | \Fitt(\Gamma)}^C[\theta]$. 
The finite group $T / T_0$ acts on $\Fitt(\Gamma) \cap C$ by conjugation by Lemma \ref{baums_grunewald8.5}. The group $\Fitt(\Gamma) \cap C$ is a torsion-free nilpotent group that is poly-(cyclic or quasi-cyclic). We have by Lemma \ref{lemma_6.1_BG}  that $H^1(T/T_0, \Fitt(\Gamma) \cap C)$ is finite.

Let $\varphi \in A_{\Gamma \: |\: \Fitt(\Gamma)}^C[\theta]$ with extension given by $\Psi$. Lemma \ref{baums_grunewald8.5} implies that there is a map $D_\varphi \colon T \to \Fitt(\Gamma) \cap C$ given by setting $D_\varphi(s) = \Psi(s) s^{-1}$. We define a map $D \colon A_{\Gamma  |  \Fitt(\Gamma)}^C[\theta] \to Z^1(T / T_0, \Fitt(\Gamma) \cap C)$ by $D(\varphi) = D_\varphi$, of which we will show below that it is indeed well-defined.
We claim that $D$ satisfies the following $3$ properties:
\begin{enumerate}
\item[(1)] The map $D$ is indeed a well-defined map;
\item[(2)] the map $D$ induces a map
$$
\tilde{D} \colon    A_{\Gamma  |  \Fitt(\Gamma)}^C[\theta] / \Inn_{\Fitt(\Gamma) \cap C}^\Gamma \cdot A_{\Gamma |  \Fitt(\Gamma)}^C[\theta]^1  \to H^1(T / T_0, \Fitt(\Gamma) \cap C);
$$
\item[(3)] the map $\tilde{D}$ is injective.
\end{enumerate}

We first demonstrate (1) of the above list by demonstrating that $D_\varphi$ is a $1$-cocycle for all $\varphi \in  A_{\Gamma \: |\: \Fitt(\Gamma)}^C[\theta]$. Let $s_1, s_2 \in T/ T_0$, and consider $D_\varphi(s_1 s_2)$. We write
\begin{eqnarray*}
D_{\varphi}(s_1 s_2) &=& \Psi(s_1 s_2) s_2^{-1} s_1^{-1}\\
&=& \Psi(s_1) \Psi(s_2) s_2^{-1} s_1^{-1}\\
&=& \Psi(s_1)s_1^{-1} s_1 \Psi(s_2) s_2^{-1} s_1^{-1}\\
&=& D_{\varphi}(s_1) (D_{\varphi}(s_2))^{s_1}
\end{eqnarray*}
as desired.

For property (2), suppose that $\varphi_1 =\varphi_2 \circ \zeta$ where 
$$
\varphi_1, \varphi_2 \in A_{\Gamma  |  \Fitt(\Gamma)}^C[\theta] \quad \text{ and } \quad \zeta \in \Inn_{\Fitt(\Gamma) \cap C}^\Gamma \cdot A_{\Gamma  |  \Fitt(\Gamma)}^C[\theta]^1.
$$ 
Let $\Psi_i$ be the extension of $\varphi_i$ and $\chi$ be the extension of $\zeta$ to $\textbf{H}_\Gamma,$ respectively. By definition, $\chi = \Inn(g) \circ \rho$ where $g \in \Fitt(\Gamma) \cap C$ and $\rho$ is the extension to $\textbf{H}_\Gamma$ of an element of $A_{\Gamma  |  \Fitt(\Gamma)}^C[\theta]^1$. Note that for $h = \Psi_2(g)$ and $s \in T$, we have
\begin{eqnarray*}
D_{\varphi_1}(s) &=& \Psi_1(s)s^{-1}\\
&=& \Psi_2(g \rho(s) g^{-1}) s^{-1}\\
&=& h \Psi_2(s) s^{-1} s h^{-1} s^{-1}\\
¨
&=&h \Psi_2(s) s^{-1} s h^{-1} s^{-1}\\
&=& h D_{\varphi_2}(s) (h^{-1})^{s}.
\end{eqnarray*}
Hence, $D_{\varphi_1}$ and $D_{\varphi_2}$ are cohomologous. Therefore, $D$ induces a map
$$
\tilde{D} \colon A_{\Gamma  |  \Fitt(\Gamma) \cap C}[\theta] / \Inn_{\Fitt(\Gamma) \cap C}^\Gamma \cdot A_{\Gamma \: | \: \Fitt(\Gamma)}^C[\theta]^1 \to H^1(T / T_0, \Fitt(\Gamma) \cap C).
$$ 

To finish, we now need to show that $\tilde{D}$ is injective. To proceed, we suppose that $D_{\varphi_1}$ and $D_{\varphi_2}$ are cohomologous. That implies there exists an element $v \in \Fitt(G) \cap C$ such that $D_{\varphi_1}(s) = v^{-1} D_{\varphi_2}(s) v^s.$ We write
\begin{eqnarray*}
D_{\varphi_1}(s) &=& v^{-1} D_{\varphi_2}(s) v^s\\
&=& v^{-1} \Psi_2(s)s^{-1}svs^{-1}\\
&=& v^{-1} \Psi_2(s) vs^{-1}\\
&=& D_{\Inn(v^{-1}) \circ \varphi_2}.
\end{eqnarray*}
Hence, $D_{\varphi_1} = D_{\Inn(v^{-1}) \circ \varphi_2}$. Therefore, 
$$
\varphi_1 \equiv \varphi_2 \text{ mod } \Inn_{\Fitt(\Gamma) \cap C}^\Gamma \cdot A_{\Gamma  | \Fitt(\Gamma)}^C[\theta]^1,
$$
and subsequently, $\tilde{D}$ is injective.

Since $\tilde{D}$ is an injective map into a finite set, the subgroup $\Inn_{\Fitt(\Gamma) \cap C}^G \cdot A_{\Gamma  | \Fitt(\Gamma)}^C[\theta]^1$ is finite index in $A_{\Gamma  | \Fitt(\Gamma)}^C$.
\end{proof}

By putting together Lemma \ref{inn_and_fix_c_finite_agf}, Lemma \ref{lemma8.4_baues_grunewald}, and Lemma \ref{lemma_8.5_baues_grunewald}, we obtain the following lemma.
\begin{lemma}\label{lemma8.7_baues_grunewald}
Let the data $(\Gamma, C, \textbf{S})$ be described as as above, then $\Inn_{\Fitt(\Gamma)}^\Gamma \cdot A_{\Gamma \: | \: \Fitt(\Gamma)}^C[\theta]^1$ is a finite index subgroup of $A_{\Gamma  | \Fitt(\Gamma)}$.
\end{lemma}

The link between the group $\Inn_{\Fitt(\Gamma)}^\Gamma \cdot A_{\Gamma  |  \Fitt(\Gamma)}^C[\theta]^1$ and $\Inn_{\Fitt(\Gamma)}^{\textbf{H}_\Gamma} \cdot \textbf{A}_{\textbf{S}}^1[\theta]$ is given by the following proposition.
\begin{prop}\label{lemma8.8_baues_grunewald}
Let the data $(\Gamma, C, \textbf{S})$ be described as as above, then we have $A_{\Gamma | \Fitt(\Gamma)}^C[\theta]^1 = \textbf{A}_{\textbf{S}}^1[\theta]$.
\end{prop}
\begin{proof}
By definition of $\textbf{A}_{\textbf{S}}^1$ 
we have $A_{\Gamma  |  \Fitt(\Gamma)}^C[\theta]^1 \leq \textbf{A}_{\textbf{S}}^1[\theta]$. For the other direction, let $\Psi \in \textbf{A}_{\textbf{S}}^1[\theta]$. Since $\theta$ is a good unipotent shadow, we have $\theta \cap \textbf{F} = \Fitt(\Gamma)$. We note that 
$$
\Psi(\Fitt(\Gamma)) = \Psi(\textbf{F} \cap \theta) = \Psi(\textbf{F}) \cap \Psi(\theta) = \textbf{F} \cap \theta = \Fitt(\Gamma).
$$
Given that $\Psi_{\textbf{H}^\circ/ \textbf{F}} = \text{id}_{\textbf{H}^\circ / \textbf{F}}$, we have that $\Psi_{\theta / \Fitt(\Gamma)} = \text{id}_{\theta / \Fitt(\Gamma)}$. We let $c \in C$ where $c = us$ for $u \in \theta$ and $s \in \textbf{S}$. We then see that $\Psi(u) = xu$ where $x \in \Fitt(\Gamma)$. We then write
$$
\Psi(c) = \Psi(us) = \Psi(u) \Psi(s) = \Psi(u) s = xus = xc \in \Gamma.
$$
Since $\Psi(\textbf{C}) = \textbf{C}$, we have $\Psi(c) \in \Gamma \cap \textbf{C}= C$. Thus, $\Psi$ stabilizes $C$. Therefore,
$$
\Psi(\Gamma) = \Psi(\Fitt(\Gamma) C) = \Psi(\Fitt(\Gamma)) \cdot \Psi(C) = \Fitt(\Gamma) \cdot C = \Gamma.
$$
Hence, $\Psi$ is the extension of a map in $A_{\Gamma  |  \Fitt(\Gamma)}^C[\theta]^1$.
\end{proof}

As a natural corollary, we have the following.
\begin{cor}\label{s_arithmetic_A_H|G}
Let the data $(\Gamma, C, \textbf{S})$ be described as as above, and suppose that $\Fitt(\Gamma)$ is a $S$-arithmetic lattice in its Zariski closure in $\textbf{H}_\Gamma$, then the group $A_{\Gamma  | \Fitt(\Gamma)}$ is a $S$-arithmetic subgroup of $\textbf{A}_{\textbf{H}_\Gamma |  \textbf{F}}$.
\end{cor}
\begin{proof}
Lemma \ref{lemma8.7_baues_grunewald} implies that $\Inn_{\Fitt(\Gamma)}^\Gamma \cdot A_{\Gamma  |  \Fitt(\Gamma)}^C[\theta]^1$ is a finite index subgroup of $A_{\Gamma \: | \: \Fitt(\Gamma)}$. Proposition \ref{lemma8.8_baues_grunewald} implies that 
$$
\Inn_{\Fitt(\Gamma)}^\Gamma \cdot A_{\Gamma  |  \Fitt(\Gamma)}^C[\theta]^1 = \Inn_{\Fitt(\Gamma)}^{\Gamma} \cdot \textbf{A}_{\textbf{S}}^1[\theta].
$$
Since every automorphism of $\Gamma$ extends uniquely to an automorphism of $\textbf{H}_\Gamma$, we have that $\Inn_{\Fitt(\Gamma)}^{\Gamma} = \Inn_{\Fitt(\Gamma)}^{\textbf{H}_\Gamma}$. Hence,
$$
\Inn_{\Fitt(\Gamma)}^\Gamma \cdot A_{\Gamma  |  \Fitt(\Gamma)}^C[\theta]^1 = \Inn_{\Fitt(\Gamma)}^{\textbf{H}_\Gamma} \cdot \textbf{A}_{\textbf{S}}^1[\theta].
$$
Proposition \ref{prop8.2_baues_grunewald} implies that $\Inn_{\Fitt(\Gamma)}^\Gamma \cdot A_{\Gamma  |   \Fitt(\Gamma)}^C[\theta]^1$ is a $S$-arithmetic subgroup of $\textbf{A}_{\textbf{H}_\Gamma  |  \textbf{F}}$. Since $\Inn_{\Fitt(\Gamma)}^\Gamma \cdot A_{\Gamma  | \Fitt(\Gamma)}^C[\theta]^1$ is a finite index subgroup of $A_{\Gamma  | \Fitt(\Gamma)}$, we have that $A_{\Gamma  |  \Fitt(\Gamma)}$ is a $S$-arithmetic subgroup of $\textbf{A}_{\textbf{H}_\Gamma  | \textbf{F}}$ as desired.
\end{proof}

%

\section{The proof of Theorem \ref{theorem1.4_baues_grunewald}}\label{sec_theorem1.4_proof}

\subsection{For \emph{WTN}-groups}
This section is devoted to the proof of Theorem \ref{theorem1.4_baues_grunewald} for \emph{WTN}-groups. Towards this end, let $\Gamma$ be a finitely generated FAR \emph{WTN}-group. We embed $\Gamma$ into the group of $\mathbb{Q}$-points of its $\Q$-algebraic hull $\textbf{H}_\Gamma$ with the induced embedding $\Aut(\Gamma) \leq \Aut(\textbf{H}_\Gamma)$. We let $\textbf{U}(\textbf{H}_\Gamma)$ be the unipotent radical of $\textbf{H}_\Gamma$. We fix a thickening $\tilde{\Gamma} = (\Fitt(\Gamma))^{\frac{1}{m}} \cdot \Gamma$ of $\Gamma$ in $\textbf{H}_\Gamma$. Moreover, we choose this thickening to satisfy the data $(\tilde{\Gamma}, C, \textbf{S})$ as defined in Section \ref{sec:thick}. Indeed, the group $C \leq \tilde{\Gamma}$ is a maximal virtually nilpotent supplement of $\Fitt(\Gamma)$. We fix $\textbf{C}$ as the Zariski closure of $C$, fix $\textbf{S} \leq \textbf{C}$ a maximal $\mathbb{Q}$-defined $d$-subgroup, and fix $\theta = \theta_{\tilde{\Gamma}}$ as a good unipotent shadow for $\tilde{\Gamma}$. We know that such a thickening exists by Proposition \ref{thickening_unipotent_shadow_strong}.

\begin{prop}\label{finite_index_aut(g)}
Let $\Gamma$ be a finitely generated virtually FAR \emph{WTN}-group. We then have $\Inn_\Gamma \cdot A_{\Gamma  | \Fitt(\Gamma)}$ has finite index in $\Aut(\Gamma)$.
\end{prop}
\begin{proof}
We claim that
$$
A_{\Gamma  |  \Fitt(\Gamma)} = \Aut(\Gamma) \cap \textbf{A}_{\textbf{H}_\Gamma  | \textbf{F}} = \Aut(\Gamma) \cap \textbf{A}_{\textbf{H}_\Gamma  |  \textbf{U}(\textbf{H}_\Gamma)}.
$$
We note that if $\varphi \in A_{\Gamma  | \Fitt(\Gamma)}$, then $\varphi|_{\Gamma / \Fitt(\Gamma)} = \text{id}_{\Gamma / \Fitt(\Gamma)}$. As $\Gamma/\Fitt(\Gamma)$ is Zariski dense in $\textbf{H}_\Gamma / \textbf{F}$, it is clear that $\Psi$ induces the identity on $\textbf{H}_\Gamma / \textbf{F}$. Hence, we have $\varphi \in \Aut(\Gamma) \cap \textbf{A}_{\textbf{H}_\Gamma | \textbf{F}}$, leading to the inclusion $A_{\Gamma |  \Fitt(\Gamma)} \leq \Aut(\Gamma) \cap \textbf{A}_{\textbf{H}_\Gamma  | \textbf{F}}$. 

For the opposite inclusion, Proposition \ref{unipotent_hull_fitt_alg_hull} implies that $\textbf{U}(\textbf{H}_\Gamma) \cap \Gamma = \textbf{F} \cap \Gamma = \Fitt(\Gamma)$ where $\textbf{F}$ is the Zariski closure of $\Fitt(\Gamma)$ in $\textbf{H}_\Gamma$. Thus, $\Psi_{\textbf{H}_\Gamma / \textbf{F}}$ when restricted to $\Gamma / \Fitt(\Gamma)$ is equal to $\varphi|_{\Gamma / \Fitt(\Gamma)}$
and thus $\varphi \in A_{\Gamma |  \Fitt(\Gamma)}$ and $A_{\Gamma |  \Fitt(\Gamma)} = \Aut(\Gamma) \cap \textbf{A}_{\textbf{H}_\Gamma | \textbf{F}}$. The second equality follows similarly.

Let $\pi_{\textbf{S}} \colon \textbf{H}_\Gamma \to \textbf{S}$ be the natural projection, and let $K= \pi_{\tilde{\textbf{S}}}(\Gamma)$ and $H = \pi_{\textbf{S}}(\Gamma_0)$ where $\Gamma_0 =\Gamma \cap \textbf{H}_\Gamma^\circ$. Observe that $\textbf{S}$ is a $\mathbb{Q}$-defined $d$-group and $(\textbf{S})^\circ$ is a $\mathbb{Q}$-defined torus. Let $\varphi \in \Aut(\Gamma)$, and let $\Psi \in \Aut(\textbf{H}_\Gamma)$ be its extension to $\textbf{H}_\Gamma$. We then see the algebraic isomorphism induces an algebraic isomorphism $\Psi_{\textbf{S}}$ of the quotient $\textbf{S}$, and since $\Psi$ is the extension of $\varphi$, we must have that $\Psi$ preserves $K$ and $H$. We denote the restriction of $\Psi_{\textbf{S}}$ to $K$ as $\varphi_K$, and observe the map $\lambda \colon \Aut(\Gamma) \to \Aut(K)$ given by $\lambda(\varphi) = \varphi_K$ is a morphism where $\varphi_K(H) = H$. When $\varphi \in A_{\Gamma  |  \Fitt(\Gamma)}$, we have that $\Psi_{\textbf{S}} = \text{id}_{\textbf{S}}$ which implies that $\varphi_K = \text{id}_K$. Hence, $\ker \lambda = A_{\Gamma  |  \Fitt(\Gamma)}$.  

By the rigidity of tori, it holds that 
$$
|N_{\Aut(\Gamma)}((\textbf{S})^\circ) : Z_{\Aut(G)}(\textbf{S})^\circ)| < \infty.
$$
Hence, a finite index subgroup of $\Aut(\Gamma) / A_{\Gamma  |  \Fitt(\Gamma)}$ acts by identity on $H$. Letting 
$$
\Aut(K, H) = \left\{ \varphi \in \Aut(K) \: | \: \varphi(H) = H, \varphi|_H = \text{id}_H \right\}
$$
and observing that $H$ is a finitely generated torsion-free abelian subgroup of $K$ of finite index, \cite[Lemma 6.2.]{baues_grunewald} implies that $\Inn_H^K$ is finite index in $\Aut(K,H)$. We then see that $\lambda(\Inn_{\Gamma_0}^\Gamma) = \Inn_H^K$ by definition. Therefore, we have that $\lambda(\Inn_{\Gamma_0}^\Gamma)$ has finite index in $\lambda(\Aut(\Gamma))$. At this point, we have that $\Inn_{\Gamma_0}^\Gamma \cdot A_{\Gamma |  \Fitt(\Gamma)}$ is the pull back of a finite index subgroup which means that it is finite index in $\Aut(\Gamma)$. Since $\Inn_{\Gamma_0}^\Gamma \cdot A_{\Gamma  | \Fitt(\Gamma)} \leq \Inn_{\Gamma} \cdot A_{\Gamma  | \Fitt(\Gamma)}$, we have that $\Inn_{\Gamma} \cdot A_{\Gamma  | \Fitt(\Gamma)}$ is finite index in $\Aut(\Gamma)$ as desired.
\end{proof}

For the convenience of the reader, we restate Theorem \ref{theorem1.4_baues_grunewald}.
\begin{thm}\label{finite_index_aut(gamma)}
Let $\Gamma$ be a finitely generated solvable group of finite abelian ranks with spectrum $S$ and trivial maximal normal torsion subgroup such that $\Fitt(\Gamma)$ is unipotently $S$-arithmetic. The group $A_{\Gamma  |  \Fitt(\Gamma)}$ is a normal subgroup of $\Aut(\Gamma)$ that is $S$-arithmetic in its Zariski closure in $\Aut(\textbf{H}_\Gamma)$. There exists a nilpotent group $B \leq \Gamma$ such that $A_{\Gamma |  \Fitt(\Gamma)} \cdot \Inn_B^\Gamma$ has finite index in $\Aut(\Gamma)$. 
\end{thm}
\begin{proof}
We see that $\Aut(\Gamma)$ is contained in $\Aut(\tilde{\Gamma})$ as a subgroup of finite index as a consequence of Proposition \ref{index_auto_group_thickening}. We note that $\Fitt(\tilde{\Gamma}) \cap \Gamma = \Fitt(\Gamma)$ which implies that $A_{\Gamma  |  \Fitt(\Gamma)} = A_{\tilde{\Gamma} | \Fitt(\tilde{\Gamma})} \cap \Aut(\Gamma)$ is finite index in $A_{\tilde{\Gamma}  | \Fitt(\tilde{\Gamma})}$. Corollary \ref{s_arithmetic_A_H|G} the implies that $A_{\Gamma  |  \Fitt(\Gamma)}$ is a $S$-arithmetic subgroup of $\textbf{A}_{\textbf{H}_\Gamma | \textbf{F}}$.

Proposition \ref{finite_index_aut(g)} implies that $\Inn_\Gamma \cdot  A_{\Gamma  |  \Fitt(\Gamma)}$ is finite index in $\Aut(\Gamma)$. Moreover $\Tilde{\Gamma} = \Fitt(\tilde{\Gamma}) \cdot C$ contains $\Gamma$ as a finite index subgroup. Hence, $\Fitt(\Gamma) \cdot (\Gamma \cap C)$ is finite index in $\Gamma$. Since $\Inn_{\Fitt(\Gamma)}$ is a subgroup of $A_{\Gamma  | \Fitt(\Gamma)}$, we have 
$$
\Inn_{\Gamma}\cdot A_{\Gamma  | \Fitt(\Gamma)} = \Inn_{\Gamma \cap C}^\Gamma \cdot A_{\Gamma  |  \Fitt(\Gamma)}.
$$
A finite index nilpotent subgroup $B \leq C \cap \Gamma$ leads to the result.
\end{proof}

By the same argument, $\Aut(\Gamma)$ contains a finite index subgroup which is a subgroup of $\Aut\left(\textbf{H}_\Gamma\right) (\Z[\frac{1}{S}])$ and thus lies within an $S$-arithmetic group.

\subsection{For virtually torsion-free groups}
In this subsection, we transfer our arithmeticity results from finitely generated virtually  FAR \emph{WTN}-groups $\Gamma$ to the more general case where we do no longer assume that $\Gamma$ is \emph{WTN} but is virtually torsion-free, i.e.~with $\tau(\Gamma)$ finite. 
We note that $\tau(\Gamma)$ is characteristic and that the quotient group $\hat{\Gamma} = \Gamma / \tau(\Gamma)$ is WTN. If $S$ is the spectrum for $\Gamma$, we then see that $S$ is the spectrum for $\hat{\Gamma}$. From now on, we assume that $\Fitt(\Gamma)$ is unipotently $S$-arithmetic. 

First we fix some notation. For any normal subgroup $N \triangleleft G$ with quotient $\bar{G} = G/N$, we denote the inclusion and quotient maps as $i_N: N \to G$ and $\pi_{\bar{G}}: G \to \bar{G}$. If $N$ is moreover characteristic and $\varphi \in \Aut(G)$ an automorphism, then we write $\varphi_N$ and $\varphi_{\bar{G}}$ for the induced automorphisms of $N$ and $\bar{G}$ respectively. As $\tau(\Gamma)$ is characteristic, we have an induced morphism $\kappa: \Aut(\Gamma) \to \Aut(\hat{\Gamma})$ defined as $\kappa(\varphi) = \varphi_{\hat{\Gamma}}$, which will be the main tool for this section.

As the group $\Gamma$ is $\mathbb{Z}[\frac{1}{S}]$-linear, it has a finite index normal subgroup $\Gamma_0$ that is torsion-free, and in particular, satisfies $\Gamma_0 \cap \tau(\Gamma) = \{1\}$. We can assume that $\Gamma_0$ is characteristic, by replacing it by $\Gamma^{n}$ where $n$ is the index of $\Gamma_0$ in $\Gamma$. Denote 
$$
Q = \Gamma / \Gamma_0, \quad \hat{\Gamma}_0 = \Gamma_0 / \tau(\Gamma) \quad \text{ and } \quad \tilde{Q} = \hat{\Gamma}/\hat{\Gamma}_0.
$$
By the discussion above, $\hat{\Gamma}_0$ is isomorphic to $\Gamma_0$ via the map $\pi_{\hat{\Gamma}}$ and $\tilde{Q}$ is a quotient of $Q$ by the finite normal subgroup $\pi_Q(\tau(\Gamma))$. The quotient morphisms $\pi_{\hat{\Gamma}} \colon \Gamma \to \hat{\Gamma}$ and $\pi_Q \colon \Gamma \to Q$ induces an injective morphism
$$
i \colon \Gamma \to \hat{\Gamma} \times Q .
$$
On the other hand, we also have a map
\begin{align*} j \colon \Aut(\Gamma) &\to \Aut(\hat{\Gamma}) \times \Aut(Q) \\ \varphi &\mapsto (\varphi_{\hat{\Gamma}}, \varphi_Q) = (\kappa(\varphi),\varphi_Q),
\end{align*}
where an easy computation shows that it is injective. We define finite index subgroups of $\Aut(\Gamma)$ and $\Aut(\hat{\Gamma})$ as follows:
\begin{eqnarray*}
A_0 &=&\{\varphi \in \Aut(\Gamma) \: | \: \varphi_{Q} = \text{id}_{Q}\} \leq \Aut(\Gamma),\\
\tilde{A}_0 &=& \{\varphi \in \Aut(\hat{\Gamma}) \: | \: \varphi_{\tilde{Q}} = \text{id}_{\tilde{Q}}\} \leq \Aut(\hat{\Gamma}).
\end{eqnarray*}

Note that $A_0$ is the kernel of the projection on the second component $\Aut(\hat{\Gamma}) \times \Aut(Q) \to \Aut(Q)$ composed with $j$.

\begin{lemma}\label{11.3_bause_grunewald}
Using the above notation, the following statements hold.
\begin{enumerate}[(1)]
\item The group $i(\Gamma)$ is of finite index in $\hat{\Gamma} \times Q$.
\item The group $\pi_{\hat{\Gamma}} \left( \Fitt(\Gamma)\right)$ has finite index in $\Fitt(\hat{\Gamma})$.
\item It holds that the image $\kappa\left(A_0\right) = \tilde{A}_0$, and in particular, $\kappa(\Aut(\Gamma))$ is of finite index in $\Aut(\hat{\Gamma})$.
\item The subgroup $A_{\Gamma | \Fitt(\Gamma)} \leq \Aut(\Gamma)$ is mapped by $\kappa$ onto a finite index subgroup of $A_{\hat{\Gamma} | \Fitt(\hat{\Gamma})}$.
\end{enumerate}
\end{lemma}
\begin{proof}
The first statement is clear, as the group $Q$ is finite and $i(\Gamma_0) = \hat{\Gamma}_0 \times \{1\}$ is already a finite index subgroup of $\hat{\Gamma}$. 

For the second statement, note that $\pi_{\hat{\Gamma}}(\Fitt(\Gamma))$ is a normal nilpotent subgroup in $\hat{\Gamma}$, implying that $\pi_{\hat{\Gamma}}(\Fitt(\Gamma)) \leq \Fitt(\hat{\Gamma})$. Hence, it suffices to show that this inclusion is finite index. For this, take $N_0 = \Fitt(\hat{\Gamma}) \cap \hat{\Gamma}_0$, which is a nilpotent normal subgroup of $\hat{\Gamma}$ that lies in $\hat{\Gamma}_0$. Since $\hat{\Gamma}_0$ is finite index in $\hat{\Gamma}$, the subgroup $N_0$ has finite index in $\Fitt(\hat{\Gamma})$, and thus, $\pi_{\hat{\Gamma}}^{-1}(N_0)$ is also finite index in $\pi_{\hat{\Gamma}}^{-1}(\Fitt(\hat{\Gamma}))$. Note that $\pi_{\hat{\Gamma}}$ is injective on the group $\Gamma_0$, and thus, $\pi_{\hat{\Gamma}}^{-1}(N_0)$ is a nilpotent normal subgroup of $\Gamma_0$. As $\Gamma_0$ is characteristic, it follows that $\pi_{\hat{\Gamma}}^{-1}(N_0) \subset \Fitt(\Gamma)$ and the latter is hence also of finite index in $\pi_{\hat{\Gamma}}^{-1}(\Fitt(\hat{\Gamma}))$. We thus get the second statement.

For the third statement, we need to show that $\kappa(A_0) = \tilde{A}_0$, which easily implies the second statement. It follows immediately that $\kappa(A_0) \leq  \tilde{A}_0$ as $\tilde{Q}$ is a quotient of $Q$. Hence, it suffices to show that for each $\varphi \in \tilde{A}_0$ there exists $\psi \in A_0$ such that $\kappa(\psi) = \varphi$. For any $\gamma \in \Gamma$, we write $\tilde{\gamma} = \pi_{\tilde{\Gamma}}(\gamma)$. Given that $\pi_{\tilde{\Gamma}}$ forms an isomorphism between $\Gamma_0$ and $\tilde{\Gamma}_0$, the automorphism $\varphi$ defines $\varphi_0 \in \Aut(\Gamma_0)$ uniquely with the property that
$$
\pi_{\tilde{\Gamma}}(\varphi_0(\gamma)) = \varphi(\tilde{\gamma})
$$ 
for $\gamma \in \Gamma_0$. Now let $\Gamma =  \bigsqcup_{i=1}^k \displaystyle \gamma_i \Gamma_0$ be a coset decomposition. There exists a unique $\delta_i \in \Gamma_0$ such that $\varphi_0(\tilde{\gamma}_i) = \tilde{\gamma}_i \tilde{\delta}_i$. Now define the map $\psi: \Gamma \to \Gamma$ as
$$
\psi(\gamma_i s) = \gamma_i \delta_i \varphi_0(s)
$$
with $s \in \Gamma_0$. A routine verification shows that $\psi$ is an automorphism of $\Gamma$ which is a lift of $\varphi$.

For the final statement, from (2) we find a surjective map $\Gamma / \Fitt(\Gamma) \to \hat{\Gamma}/ \Fitt(\hat{\Gamma})$ which has finite kernel. Given that this morphism is surjective, the morphism $\Aut(\Gamma) \to \Aut(\hat{\Gamma})$ sends $A_{\Gamma | \Fitt(\Gamma)}$ to $A_{\hat{\Gamma} | \Fitt(\hat{\Gamma})}$. Let $A_1$ be the preimage of $A_{\hat{\Gamma} | \Fitt(\hat{\Gamma}}$ in $\Aut(\Gamma)$, then one can easily show that a finite index subgroup of $A_1$ acts as the identity on $\hat{\Gamma} / \Fitt(\hat{\Gamma})$. Hence, $A_{\Gamma | \Fitt(\Gamma)}$ has finite index in $A_1$ which implies (4).
\end{proof}

We now demonstrate for a finitely generated virtually torsion-free FAR-group $\Gamma$ that $\Aut(\Gamma)$ and $\Aut(\Gamma / \tau(\Gamma))$ are commensurable.
\begin{prop}\label{aut_gamma_arithmetic}
Let $\Gamma$ be a finitely generated virtually torsion-free FAR-group. The groups $\Aut(\Gamma)$ and $\Aut(\hat{\Gamma})$ are commensurable, and in particular, the group $\Aut(\hat{\Gamma})$ is an $S$-arithmetic group if and only if $\Aut(\Gamma)$ is $S$-arithmetic.
\end{prop}
\begin{proof}
The groups $\Aut(\Gamma)$ and $\Aut(\hat{\Gamma})$ are commensurable since they contain finite index subgroups $A_0$ and $\tilde{A}_0$ that are isomorphic via $\kappa$. As $S$-arithmeticity is preserved by commensurability, the second part of the proposition follows.
\end{proof}

\begin{prop}\label{arithmetic_A_Gamma_Fitt}
If $\Fitt(\hat{\Gamma})$ is unipotently $S$-arithmetic, then the subgroup $A_{\Gamma | \Fitt(\Gamma)}$ of $\Aut(\Gamma)$ is $S$-arithmetic.
\end{prop}
\begin{proof}
From (4) of Lemma \ref{11.3_bause_grunewald} we know that $\kappa(A_{\Gamma | \Fitt(\Gamma)})$ is a finite index subgroup of $A_{\hat{\Gamma} | \Fitt(\hat{\Gamma})} \times \Aut(Q)$. Since Corollary \ref{s_arithmetic_A_H|G} implies that $A_{\hat{\Gamma} | \Fitt(\hat{\Gamma})}$ is $S$-arithmetic, it  follows that $A_{\Gamma | \Fitt(\Gamma)}$ is $S$-arithmetic.
\end{proof}

We now prove a generalization of Theorem \ref{theorem1.4_baues_grunewald} to finitely generated virtually torsion-free solvable groups of finite abelian ranks which are not necessarily \emph{WTN}.
\begin{proof}
As mentioned above, the map $\kappa$ sends $A_{\Gamma | \Fitt(\Gamma)}$ to a finite index subgroup of $A_{\hat{\Gamma} | \Fitt(\hat{\Gamma})}$. From Theorem \ref{finite_index_aut(gamma)} we know that there exists a $N \leq \Gamma \cap \Gamma_0$ nilpotent subgroup such that $A_{\hat{\Gamma} | \Fitt(\hat{\Gamma})} \cdot \pi_{\hat{\Gamma}}(\Inn_N)$ has finite index in $\Aut(\hat{\Gamma})$. We then have $A_{\Gamma | \Fitt(\Gamma)} \cdot \Inn_N$ is finite index in $\Aut(\Gamma)$. We finish by appealing to Proposition \ref{arithmetic_A_Gamma_Fitt}.
\end{proof}

As mentioned previously, Baues and Grunewald \cite{baues_grunewald} construct an example of a polycyclic group $\Gamma$ such that $\Aut(\Gamma)$ is not arithmetic. Therefore, there is no $S$-arithmeticity statement that can be made for $\Aut(\Gamma)$ in general, but we can always guarantee it for a certain finite index subgroup. To do this we define
$$
\hat{F}(\Gamma)  = \kappa^{-1}(\Fitt(\hat{\Gamma})).
$$
We note that $\hat{F}(\Gamma)$ is normal in $\Gamma$ and that $\Gamma / \hat{F}(\Gamma)$ is a virtually abelian group. Hence, there exists a finite index normal subgroup $\Gamma_0$ of $\Gamma$ such that $\hat{F}(\Gamma) \leq \Gamma_0$ and where $\Gamma_0 / \hat{F}(\Gamma)$ is an abelian finite index subgroup of $\Gamma / \tilde{F}(\Gamma)$. We then have the group $\Gamma / \Gamma_0$ acts by conjugation on the abelian group $\Gamma_0 / \tilde{F}(\Gamma)$. Hence, we have following theorem.

\begin{thm}\label{S_arithmetic_finite_index_aut}
	Let $\Gamma$ be a finitely generated virtually torsion-free solvable group of finite abelian ranks with spectrum $S$. Suppose there is a normal subgroup $\Gamma_0 \leq \Gamma$ of finite index such that $\Gamma_0 / \hat{F}(\Gamma)$ is abelian and where $\Gamma / \Gamma_0$ acts trivially on $\Gamma_0 / \hat{F}(\Gamma)$. If $\Fitt(\Gamma)$ is unipotently $S$-arithmetic, then $\Aut(\Gamma)$ is a $S$-arithmetic group.
\end{thm}

\begin{proof}
By Proposition \ref{aut_gamma_arithmetic}, we may assume that $\tau(\Gamma) = 1$. By assumption, there exists a normal subgroup $\Gamma_0 \leq \Gamma$ of finite index such that $\Gamma_0 / \Fitt(\Gamma)$ is abelian and where $\Gamma / \Gamma_0$ acts trivially by conjugation on $\Gamma_0/ \Fitt(\Gamma)$. Proposition \ref{finite_index_aut(g)} implies $\Inn_\Gamma \cdot A_{\Gamma \: | \: \Fitt(\Gamma)}$ is finite index in $\Aut(\Gamma)$ with $A_{\Gamma \: | \: \Fitt(\Gamma)}$ $S$-arithmetic. By the assumptions, there exists a finite index subgroup of $\Inn_\Gamma$ that lies in $ A_{\Gamma \: | \: \Fitt(\Gamma)}$, leading to the result.
\end{proof}

We have the following immediate corollary.
\begin{cor}
	Let $\Gamma$ be a finitely generated virtually torsion-free solvable group of finite abelian ranks with spectrum $S$ such that  $\Fitt(\Gamma)$ is unipotently $S$-arithmetic. There exists a finite index subgroup $\Gamma_0 \leq \Gamma$ such that $\Aut(\Gamma_0)$ is $S$-arithmetic. If $\Gamma / \Fitt(\Gamma)$ is nilpotent, then $\Aut(\Gamma)$ is $S$-arithmetic. 
\end{cor}

The above theorem and its corollary are a generalization of the corollaries to \cite[Theorem 1.9]{baues_grunewald} to finitely generated virtually torsion-free solvable groups of finite abelian ranks.

\subsection{Outer automorphisms}

To end this section, we present some insights in the natural map $\pi_\Gamma: \Out(\Gamma) \to \Out_a(\textbf{H}_\Gamma)$, induced by the quotient map $\Aut(\textbf{H}_\Gamma) \to \Out_a(\textbf{H}_\Gamma)$. On the one hand, we investigate the structure of the kernel of $\pi_\Gamma$, on the other hand, we show that the image $\pi_\Gamma(\Out(\Gamma)) \subset \Out_a(\textbf{H}_\Gamma)$ is $S$-arithmetic, where $S$ is the spectrum of $\Gamma$. Both of these statements are found in Theorem \ref{out(G)_arithmetic} which we restate for the reader's convenience.

\begin{thm}
Let $\Gamma$ be a finitely generated virtually FAR WTN-group with spectrum $S$, and suppose that $\Fitt(\Gamma)$ is unipotently $S$-arithmetic. Then
\begin{enumerate}[(1)]
\item the image $\pi_\Gamma(\Out(\Gamma))$ is a $S$-arithmetic group;
\item the kernel $\ker(\pi_\Gamma)$ has finite Hirsch length, is virtually abelian, and is centralized by a finite index subgroup of $\Out(\Gamma)$;
\item if $G$ is moreover virtually nilpotent, then the kernel of $\pi_\Gamma$ is finite.
\end{enumerate}
\end{thm}
\begin{proof}
In this proof, we denote by
	\begin{align*}
	H_F = \Inn_{\textbf{H}_\Gamma} \cap \Aut(\Gamma) \supset E_F = \Inn_{\textbf{F}}^{\textbf{H}_\Gamma} \cap A_{\Gamma  |  \Fitt(\Gamma)}.
	\end{align*}
We first demonstrate the following facts about $H_F$ and $E_F$.
\begin{enumerate}[(1)]
		\item The subgroup $\Inn_{\Fitt(\Gamma)}^\Gamma$ has finite index in $E_F$.
		\item The commutator group $[A_{\Gamma  |  \Fitt(\Gamma)}, H_F]$ is a subgroup of $E_F$.
	\end{enumerate}
To prove this, note that $\Inn_{\Fitt(\Gamma)}^{\textbf{H}_\Gamma}$ is $S$-arithmetic in $\Inn_{\textbf{F}}^{\textbf{H}_\Gamma}$. As $A_{\Gamma  |  \Fitt(\Gamma)}$ is $S$-arithmetic in $\textbf{A}_{\textbf{H}_\Gamma \: | \: \textbf{F}}$, we must have that $E_F$ is $S$-arithmetic in $\Inn_{\textbf{F}}^{\textbf{H}_\Gamma}$. Since $\Inn_{\Fitt(\Gamma)}^\Gamma \leq E_F$ and $\Inn_{\Fitt(\Gamma)}^\Gamma$ is $S$-arithmetic in $\Inn_{\textbf{F}}^{\textbf{H}_\Gamma}$, we have that $\Inn_{\Fitt(\Gamma)}^\Gamma$ is finite index in $E_F.$ Therefore, we have our first statement.

For the second, let $\varphi \in A_{\Gamma | \Fitt(\Gamma)}$ with $\Psi$ the extension to the $\Q$-algebraic hull $\textbf{H}_\Gamma$ and $h \in \textbf{H}_\Gamma$ with $\Inn_h \in H_F$. Since $\Psi \in \textbf{A}_{\textbf{H}_\Gamma | \textbf{F}}$, we may write
	$$
	\Psi \circ \Inn_h \circ \Psi^{-1} = \Inn_{\Psi(h)} = \Inn_h \circ \Inn_k 
	$$
	where $k \in \textbf{F}$. Hence,
	\begin{eqnarray*}
		\Inn_{h^{-1}} \circ \Psi \circ \Inn_h \circ \Psi^{-1} &=&\Inn_{h^{-1}}  \circ \Inn_h \circ \Inn_k\\
		&=& \Inn_k \in \Inn_{\textbf{F}}^{\textbf{H}_\Gamma} \cap A_{\Gamma | \Fitt(\Gamma)} =E_F. 
	\end{eqnarray*}
	
	Let $\bar{H}_F, \bar{E}_F$ and $\overline{A_{\Gamma  | \Fitt(\Gamma)}}$ be the images of $H_F, E_F$ and $A_{\Gamma  | \Fitt(\Gamma)}$ in $\Out(\Gamma)$, where the latter is a finite index subgroup by Theorem \ref{theorem1.4_baues_grunewald}. Hence, since Theorem \ref{theorem1.4_baues_grunewald} implies $A_{\Gamma \: | \: \Fitt(\Gamma)}$ is $S$-arithmetic in $\textbf{A}_{\textbf{H}_\Gamma \: | \: \textbf{F}}$,  \textbf{AR1} implies that $\pi_\Gamma(\overline{A_{\Gamma \: | \: \Fitt(\Gamma)}})$ is a $S$-arithmetic lattice in $\pi_\Gamma(\textbf{A}_{\textbf{H}_\Gamma \: | \: \textbf{F}})$. Therefore, $\pi_\Gamma(\Out(\Gamma))$ is $S$-arithmetic, leading to the first statement of the theorem.
	
	 By definition, $\ker(\pi_\Gamma) = \bar{H}_F$.  
	Fact (1) implies $\bar{E}_F$ is finite, whereas fact (2) implies $[\overline{A_{\Gamma  |  \Fitt(\Gamma)}}, \bar{H}_F] \leq \bar{E}_F$ and thus the former is finite.  In particular, a finite index subgroup of $\overline{A_{\Gamma  |  \Fitt(\Gamma)}}$ and thus of $\Out(\Gamma)$ centralizes $\bar{H}_F = \ker(\pi_\Gamma)$. 
As the group $\bar{H}_F$ has a finite index subgroup that is $S$-arithmetic in a solvable group with $S$ finite, it indeed has finite Hirsch length. When $\Gamma$ is virtually nilpotent, then it is clear that $E_F$ has finite index in $H_F$ and hence $\bar{H}_F$ is finite.
\end{proof}

\section{Further applications of the $\Q$-algebraic hull}\label{section_futher_applications}
%

%

In this section, we apply the $\mathbb{Q}$-algebraic hull to generalize the results from \cite{dere1}. The first subsection uses Theorem \ref{algebraic_hull_thm} to compute the automorphism group of $\text{BS}(1,n)$ recovering a classic calculation. We also demonstrate that necessity of the hypothesis of Theorem \ref{theorem1.4_baues_grunewald}. The second  subsection deals with the existence of strongly scale-invariant monomorphisms as introduced in \cite{pete_nekrashevych}, the final subsection with the existence of injective morphisms with a well-defined Reidemeister zeta function \cite{felshtyn}. Notations and background are given in the separate sections. 

\subsection{Computing automorphism groups}

 For a non-zero integer $n$, we denote the Baumslag-Solitar group $\text{BS}(1,n)$ as
$$
\text{BS}(1,n) = \left< x,t \: | \: t^{-1}xt = x^n\right>.
$$
\begin{thm}\label{calculation_baumslag_solitar}
Suppose that $n$ is a nonzero integer such that $|n| > 1$ where $S$ is the set of prime factors of $n$, and consider $\Gamma = BS(1,n)$. Then $A_{\Gamma \: | \: \Fitt(\Gamma)} = \Aut(\Gamma)$ is an $S$-arithmetic groups. If $|S| = 1$, then $\Out(\Gamma)$ is finite, and hence, $S$-arithmetic. If $|S|>1$, then $\Out(\Gamma)$ is a $S$-arithmetic group in a non-split torus.
\end{thm}
\noindent For $|n| > 1$, it is clear that $\text{BS}(1,n)$ is a finitely generated solvable group of finite abelian ranks that is not virtually polycyclic. Moreover, the spectrum of $\text{BS}(1,n)$ is given by $\{p_1, \ldots, p_k\}$ where $S= \{p_1, \ldots, p_k\}$ is the set of prime divisors of $n$. It holds that $\Fitt(\text{BS}(1,n)) = \mathbb{Z}[\frac{1}{p_1}, \ldots, \frac{1}{p_k}]$ which is a $S$-arithmetic subgroup of $\mathbb{Q}$.  

Let us first compute the automorphism group $\Aut(\Gamma)$ using the methods of this paper. Note that this gives an elementary version of finding the presentation of \cite[Proposition A]{collins} for $\Aut(\Gamma)$. From Example \ref{ex:BS} we know that $H_\Gamma \cong \C^2 \rtimes \C^\ast$ with the action by multiplication on the first component. The algebraic automorphisms of this group can be computed as in Section \ref{closed_subgroups_Aut}, where $\Aut(H_\Gamma)_S \cong (\C^\ast)^2$ via the diagonal automorphisms. The unipotent radical $\C^2$ has a trivial action from the second component. In particular, $\Aut(H_\Gamma) \cong \C \rtimes (\C^\ast)^2$, where $(a,b,c) \in  \C \rtimes (\C^\ast)^2$ corresponds to the automorphism $\varphi(x,y,z) = (bx+a(1-z),cy,z)$. 

The condition $\varphi(\Gamma) = \Gamma$ then corresponds to $a \in \Z[\frac{1}{S}], b \in \Z[\frac{1}{S}]^\ast, c=1$, and thus, $$\Aut(\Gamma) \cong \mathbb{Z}\left[\tfrac{1}{p_1}, \ldots, \tfrac{1}{p_k}\right] \rtimes \left< \left. \prod_{i=1}^k \pm p_i^{m_i} \: \right| \: m_i \in \mathbb{Z} \right>.$$ In particular, $A_{\Gamma \mid \Fitt(\Gamma)} = \Aut(\Gamma)$. A presentation is then given by $$\Aut(\Gamma) = \left< C, T, Q_1, \ldots Q_k \: \left| \:  \begin{array}{c} T^2 = 1, TCT = C^{-1},  TQ_i = Q_iT, 1 \leq i \leq k \\ Q_i^{-1} C Q_i = C^{p_i}, Q_i Q_j = Q_j Q_i, 1 \leq i < j \leq k  \end{array}  \right. \right>.
$$
with the automorphisms defined by
$$\begin{array}{c c}
	Q_i(t) = t, & Q_i(x)=x^{p_i}\\
	C(t) =tx & C(x) = x\\
	T(t) = t, & T(x) = x^{-1}
\end{array}$$.

\begin{proof} 
From the discussion above, $$\Aut(\Gamma) \cong \mathbb{Z}\left[\tfrac{1}{p_1}, \ldots, \tfrac{1}{p_k}\right] \rtimes \left< \left. \prod_{i=1}^k \pm p_i^{m_i} \: \right| \: m_i \in \mathbb{Z} \right>,
$$which is equal to the $\Z[\frac{1}{S}]$-points in $\mathbb{Q} \rtimes \mathbb{Q}^*$. Suppose that $n = \delta \prod_{i=1}^k p_i^{\ell_i}$ is the prime decomposition of $n$ where $\delta = \pm 1$. We claim that $$
\Inn(\Gamma) = \left<C^{-n+1} , T^{\epsilon} \prod_{i=1}^k Q_i^{\ell_i} \right>
$$
where $\epsilon =0$ if $\delta = 1$ and $\epsilon = 1$ if $\delta = -1$. We see that 
$$
x^{-1} t x =tt^{-1} x^{-1}tx  = tx^{-n}x  = tx^{-n+1} = C^{-n+1}(t)
$$
Since $C^{-n+1}(x) = x$, we see that $C^{n-1}$ is the inner automorphism associated to the element $x$. Given that $T^\epsilon \prod_{i=1}^k Q^{\ell_i}_i$ is the inner automorphism associated to $t$, our claim then follows. Therefore, we have 
$$
\Out(\Gamma) = \left< C, T, Q_1, \ldots Q_k \: \left| \:  \begin{array}{c} T^2 = 1, TCT = C^{-1},  TQ_i = Q_iT, 1 \leq i \leq k \\ Q_i^{-1} C Q_i = C^{p_i}, Q_i Q_j = Q_j Q_i, 1 \leq i < j \leq k \\  C^{n-1} =1 , T^{\epsilon} \prod_{i=1}^k Q_i^{\ell_i} = 1 \end{array}  \right. \right>.
$$
We note that the subgroup $\left<C\right>$ is normal in $\Out(\Gamma)$. As $\Out(\Gamma) / \left< C \right>$ is clearly virtually a free abelian group of rank $k-1$, we find that also $\Out(\Gamma)$ is virtually a free abelian group of rank $k-1$. 

When $n$ has a single prime factor, the group $\Out(\Gamma)$ is finite and thus $S$-arithmetic. However, when $|n|$ has $2$ or more prime factors $k$, $\Out(\Gamma)$ is a finite extension of $\mathbb{Z}^{k-1}$. So to demonstrate that this group is $S$-arithmetic, it suffices to find a non-split torus $\textbf{A}$ such that $\textbf{A}(\mathbb{Z}[\frac{1}{S}]) \cong \mathbb{Z} \times \{\pm 1\}$ and taking $k-1$ direct products of these.

Take $d = \prod_{p \in S}p$, and consider
$$
\textbf{A} = \left\{ \begin{bmatrix}
    a & db\\
    b & a
\end{bmatrix} \: : \: a^2 - db^2 = 1 \right\}.
$$
To study $\textbf{A}(\mathbb{Z}[\frac{1}{S}]) \cong \mathbb{Z} \times \{\pm 1\}$, take $a,b \in \mathbb{Z}[\frac{1}{S}]$ with $a^2 - db^2 = 1$. Let $v_p(a)$ be the $p$-adic valuation for a prime $p \in S$, which we normalize by setting $v_p(p) = 1$. If $b = 0$, then $a = \pm 1$ and if $a = 0$, then $b^2 = \frac{-1}{d}$ which is a contradiction. Thus, we may assume that $a,b \neq 0$.  It is clear that $v_p(d) = 1$, $v_p(a^2) = 2v_p(a)$ and $v_p(db^2) = 1 + 2v_p(b).$ We see that the 
$$
0 = v_p(1) = v_p(a^2 - db^2) = \min\{2v_p(a), 1+ 2v_p(b)\}.
$$
In particular, we see that $v_p(a)$ and $v_p(b^2)$ are non-negative. As this holds for each prime $p \in S$, we have $a,b \in \mathbb{Z}$ and thus, $\textbf{A}(\mathbb{Z}[\frac{1}{S}]) \cong \textbf{A}(\mathbb{Z}).$ It is a standard fact that the latter is virtually $\Z$.

\end{proof}

Next, we give examples of finitely generated, torsion-free FAR-groups $\Gamma$ without $S$-arithmetic Fitting subgroup such that the conclusion of Theorem \ref{theorem1.4_baues_grunewald} does not hold. Now if $p$ and $q$ distinct odd primes, we then see that the spectrum of $\Gamma = \text{BS}(1,p) \times \text{BS}(1,q)$ is given by $S= \{p, q\}$ where $\Fitt(\text{BS}(1,p) \times \text{BS}(1,q)) = \mathbb{Z}[\frac{1}{p}] \times \Z[\frac{1}{q}]$. We will now demonstrate that $\text{BS}(1, p) \times \text{BS}(1, q)$ does not satisfy the conclusions of Theorem \ref{theorem1.4_baues_grunewald}. To proceed, we prove a number of some preliminary results. 
\begin{lemma}\label{theorem1.4_baues_grunewald_nec_lemma1}
	Let $p$ and $q$ be distinct primes. Then 
	$$
	\Aut\left( \mathbb{Z} \left[\tfrac{1}{p} \right] \times \mathbb{Z} \left[ \tfrac{1}{q} \right]  \right) \cong \Aut\left( \mathbb{Z} \left[\tfrac{1}{p} \right] \right) \times \Aut \left(\mathbb{Z} \left[ \tfrac{1}{q} \right]  \right).
	$$
\end{lemma}
\begin{proof}
	Letting $A = \mathbb{Z}[\frac{1}{p}]$ and $B = \mathbb{Z}[ \frac{1}{q}],$ we note that $\Aut(\mathbb{Z}[\frac{1}{p}] \times \mathbb{Z}[ \frac{1}{q}])$ is given by invertible $2 \times 2$-matrices
	$$
	\begin{bmatrix}
		f & g \\
		h & k
	\end{bmatrix} \text{ where } f \in \text{Hom}_{\mathbb{Z}}(A,A), k \in \text{Hom}_{\mathbb{Z}}(B,B), g \in \text{Hom}_{\mathbb{Z}}(B,A), \text{ and } h \in \text{Hom}_{\mathbb{Z}}(A,B).
	$$
Let $f \colon A \to B$ be a group homomorphism. Since $A$ and $B$ are both additive subgroups of $\mathbb{Q}$, this is equivalent to multiplication by a rational number $r$ such that $r \cdot A \subseteq B$. As $f(1) \in B$, we can write $r = \frac{b}{q^m}$ with $b, m \in \Z$. However, as
	$$
	f\left(\frac{a}{p^k}\right) =  \frac{a b }{p^k q^m } \in \mathbb{Z}\left[ \tfrac{1}{q}\right]
	$$
	for all $a, k,\in \Z$, we must have that $b =0$. Hence, $\text{Hom}_{\mathbb{Z}}(A,B) = \{0\}$, and similarly, $\text{Hom}_{\mathbb{Z}}(B,A) = \{0\}$. Therefore,
 	$$
	\Aut\left(\mathbb{Z}\left[\tfrac{1}{p}\right]\right) \times \Aut\left(\mathbb{Z}\left[ \tfrac{1}{q}\right]\right).
	$$
\end{proof}
Using the previous lemma, we now demonstrate that $\Aut(\text{BS}(1,p) \times \text{BS}(1,q)) \cong \Aut(\text{BS}(1,p)) \times \Aut(\text{BS}(1,q))$.
\begin{prop}\label{prop_direct_product_auto}
	Let $p, q$ be distinct primes. Then
	$$
	\Aut(\text{BS}(1,p) \times \text{BS}(1,q)) \cong \Aut(\text{BS}(1,p)) \times \Aut(\text{BS}(1,q))
	$$
\end{prop}
\begin{proof}
	It is easy to see that 
	$$
	\Aut(\text{BS}(1,p)) \times \Aut(\text{BS}(1,q))\leq \Aut(\text{BS}(1,p) \times \text{BS}(1,q)).
	$$
	Thus, to show our proposition, we need to demonstrate that any element of $\Aut(\text{BS}(1,p) \times \text{BS}(1,q))$ stabilizes the factors in $\text{BS}(1,p) \times \text{BS}(1,q)$. Note that both $N_1 = \Fitt(\text{BS}(1,p))$ and $N_2 = \Fitt(\text{BS}(1,q))$ are characteristic in $\text{BS}(1,p) \times \text{BS}(1,q)$ by the previous lemma. As the center of $\text{BS}(1,p) \times \text{BS}(1,q)/N_1$ is equal to $\text{BS}(1,p)/N_1$ and similarly for $N_2$, the proposition follows easily.
\end{proof}

The next theorem justifies the assumption of $S$-arithmeticity of $\Fitt(\Gamma)$ in the statement of Theorem \ref{theorem1.4_baues_grunewald} where $\Gamma$ is a finitely generated virtualy torsion-free FAR with spectrum $\{p,q\}$.
\begin{thm}
	\label{thm_nonarith}
Let $p$ and $ q$ be distinct primes and denote $S = \{p, q\}$. Let $\Gamma =\text{BS}(1,p) \times \text{BS}(1, q))$. Then $\Gamma$ satisfies the following:
\begin{enumerate}[(1)]
\item The spectrum of $\Gamma$ is $S$;
\item The group $\Fitt(\Gamma)$ is not unipotently $S$-arithmetic;
\item The group $A_{\Gamma \: | \: \Fitt(\Gamma)}$ is not $S$-arithmetic.
\end{enumerate}
\end{thm}
\begin{proof}
Since $\Aut(\text{BS}(1, n)) \cong A_{\text{BS}(1,n) \: | \: \Fitt(\text{BS}(1,n))}$, Proposition \ref{prop_direct_product_auto} implies that 
$$
A_{\Gamma \: | \: \Fitt(\Gamma)} \cong A_{\text{BS}(1,p) \: | \: \Fitt(\text{BS}(1, p))} \times A_{\text{BS}(1,q) \: | \: \Fitt(\text{BS}(1, q))}
$$
where $\Gamma = \text{BS}(1,p) \times \text{BS}(1,q)$. It is easy to see that $\Fitt(\text{BS}(1,p) \times \text{BS}(1,q))$ is not $\{p, q\}$-arithmetic in $\mathbb{Q}$-defined unipotent algebraic group. We claim that $A_{\Gamma \: | \:  \Fitt(\Gamma)}$ is not $S$-arithmetic, and for a contradiction, suppose otherwise. Thus, there exists a $\mathbb{Q}$-defined  linear algebraic group $\textbf{G}$ so that $A_{\Gamma \: | \: \Fitt(\Gamma)}$ is commensurable with $\textbf{G}(\mathbb{Z}[\frac{1}{S}])$. Note that the commutator subgroup is a finite index subgroup of $\Z[1/p] \oplus \Z[1,q]$, which would have to be unipotently $S$-arithmetic and that is impossible. 
\end{proof}

The above examples can be generalized to any direct product of the form $\text{BS}(1,n_1) \times \text{BS}(1, n_2)$ where $n_1$ and $n_2$ are two natural numbers greater than $1$ where $n_1$ and $n_2$ have distinct prime factors.

\subsection{Strongly scale-invariant solvable groups}\label{strongly_scale_invariant}
We start with the following definition.
\begin{defn}
	Let $\Gamma$ be any group with an endomorphism $\varphi \colon \Gamma \to \Gamma$. We say  monomorphism $\varphi \colon \Gamma \to \Gamma$ is called \emph{strongly scale-invariant} if $\varphi(\Gamma)$ has finite index in $\Gamma$ and $\bigcap_{n>0} \varphi^n(\Gamma)$ is finite. We say a group $G$ is \emph{strongly scale-invariant} if it admits a strongly scale-invariant monomorphism.
\end{defn}
Motivated by Gromov's work \cite{gromov} on expanding maps, Nekrashevych and Pete \cite{pete_nekrashevych} conjectured that the only finitely generated groups that admit a strongly scale-invariant endomorphisms are virtually nilpotent, which they demonstrated for the class of non-elementary torsion-free hyperbolic groups. The first named author then by using the $\mathbb{Q}$-algebraic hull demonstrated that all virtually polycyclic groups that admit a strongly scale-invariant endomorphism are virtually nilpotent. We extend these results to the class of finitely generated virtually FAR-groups. 

Given a group $G$ with a monomorphism $\varphi \colon G \to G$, we say that $H < G$ is $\varphi$-invariant if $\varphi(H) < H$. If $H$ is $\varphi$-invariant for all monomorphisms $\varphi \colon G \to G$, we say that $H$ is an \emph{injectively characteristic subgroup} of $G$. We observe that most injectively characteristic subgroups we will be considering will be in fact fully characteristic, meaning that it is invariant under all endomorphisms of the group $G$. It is easy to see that all injectively characteristic subgroups are normal since they must be invariant under conjugation.

If $\Gamma$ is a strongly scale-invariant group equipped with a strongly scale-invariant monomorphism $\varphi \colon \Gamma \to \Gamma$, then the finite residual $\Fr(\Gamma)$ satisfies
$$
\Fr(\Gamma) \leq \bigcap_{n > 0} \varphi^n(\Gamma).
$$
Thus, if a group $\Gamma$ has an infinite finite residual, then $\Gamma$ cannot be strongly scale-invariant. Therefore, all groups in this section are assumed to have a finite residual that is finite.

We now proceed to give basic properties of strongly scale-invariant monomorphisms, as given in \cite{dere1}. The first lemma says that a monomorphism $\varphi \colon \Gamma \to \Gamma$ is strongly scale-invariant only if $\varphi^n \colon \Gamma \to \Gamma$ is strongly scale-invariant for any $n \in \mathbb{N}$.
\begin{lemma}
	\label{sscipower}
	Let $\varphi \colon \Gamma \to \Gamma$ be a monomorphism, and let $n>0$ be a natural number. We then have that $\varphi$ is strongly scale-invariant if and only if $\varphi^n \colon \Gamma \to \Gamma$ is strongly scale-invariant.
\end{lemma}

The following lemma shows that a monomorphism $\varphi \colon \Gamma \to \Gamma$ being strongly scale-invariant holds for the restriction of $\varphi$ to any $\varphi$-invariant subgroup.
\begin{lemma}\label{stable_subgroup_scale_invariant}
	Let $\varphi \colon \Gamma \to \Gamma$ be a strongly scale-invariant monomorphism, and let $\Gamma^\prime$ be a $\varphi$-invariant subgroup. Then $\restr{\varphi}{\Gamma^\prime} \colon \Gamma^\prime \to \Gamma^\prime$ is strongly scale-invariant.
\end{lemma}

Finally, in the special case of nilpotent groups of finite rank, the following holds.

\begin{lemma}
	\label{sscinilpotent}
	Let $N$ be a Zariski-dense subgroup of a unipotent group $\textbf{U}$ and $\varphi: N \to N$ an endomorphism that is strongly scale-invariant, with unique extension $\phi$ to $\textbf{U}(\Q)$. For every $x \in \textbf{U}(\Q)$, there exists a unique $y$ in $\textbf{U}(\Q)$ with $x = y \phi(y)^{-1}$.
\end{lemma}

In this section, we apply the construction of the $\Q$-algebraic hull to characterize what finitely generated virtually FAR-groups are strongly scale-invariant.

\begin{thm}\label{torsion_free_by_free_abelian_ssin}
	Let $\Gamma$ be a finitely generated virtually FAR-group. If $\Gamma$ is strongly scale-invariant, then $\Gamma$ is virtually nilpotent. 
\end{thm}
\begin{proof}
As mentioned before, we know that $\Fr(\Gamma)$ is finite for a strongly scale-invariant group $\Gamma$, and thus, by Lemma \ref{lem:equivalencesRF}, we know that $\tau(\Gamma)$ is a finite group. By \cite[Proposition 2.7.]{cornulier}, we know that every injective morphism $\varphi \colon \Gamma \to \Gamma$ preserves $\tau(\Gamma)$, meaning that $\varphi(\tau(\Gamma)) \subset \tau(\Gamma)$. As $\tau(\Gamma)$ is finite, this means that $\varphi(\tau(\Gamma)) = \tau(\Gamma)$ and thus that $\varphi$ induces an injective morphism on the quotient $\Gamma / \tau(\Gamma)$. As it suffices to show that the latter group is virtually nilpotent, we may assume that $\tau(\Gamma) = 1$.
	
	In this case, $\Gamma$ has a $\Q$-algebraic hull $\textbf{H}_\Gamma$. By using Lemma \ref{stable_subgroup_scale_invariant} and replacing $\Gamma$ by a finite index subgroup if necessary, we can assume $\textbf{H}_\Gamma$ is connected. Take $\Psi$ the extension of $\varphi$ to $\textbf{H}_\Gamma$. The group $\textbf{H}_\Gamma/ \textbf{U}(\textbf{H}_\Gamma)$ is a torus, and so in particular as $\Aut(\textbf{H}_\Gamma)$ is algebraic, $\Psi$ has finite order on this quotient. By replacing $\varphi$ by a suitable power, which is possible by Lemma \ref{sscipower}, we can hence assume that $\Psi$ induces the identity on this quotient. As $\Gamma / \Fitt(\Gamma)$ is isomorphic to a subgroup of  $\textbf{H}_\Gamma/ \textbf{U}(\textbf{H}_\Gamma)$ by Proposition \ref{unipotent_hull_fitt_alg_hull}, we in particular have that $\varphi$ induces the identity on the quotient  $\Gamma / \Fitt(\Gamma)$.
	
	Now, assume that $\Gamma$ is not virtually nilpotent, or thus in particular that the group $\Gamma / \Fitt(\Gamma)$ contains an element of infinite order $x \Fitt(\Gamma)$. There exists $ y \in \Fitt(\Gamma)$ such that $\varphi(x) = x y$. As $\Fitt(\Gamma) = \Gamma \cap \textbf{U}(\textbf{H}_\Gamma)$, it is $\varphi$-invariant and thus, $\restr{\varphi}{\Fitt(\Gamma)}$ is also strongly scale-invariant by Lemma \ref{sscinilpotent}. Hence, Lemma \ref{sscinilpotent} implies that there exists a $z \in \textbf{F}(\Q) \subset \textbf{U}(\textbf{H}_\Gamma)$ such that $z \Psi(z)^{-1} = y$. A short computation shows that $$\Psi(xz) = \Psi(x) \Psi(z) = x y y^{-1} z = xz.$$ So by taking a suitable power of $xz$, we find an element $\gamma \in \Gamma$ of infinite order such that $\varphi(\gamma) = \gamma$, which is a contradiction. We conclude that $\Gamma$ is virtually nilpotent.
\end{proof}

\subsection{Reidemeister zeta functions}
\label{sec:Reid}

Given an endomorphism $\varphi: \Gamma \to \Gamma$ on a group $\Gamma$, we call elements $x, y \in \Gamma$ $\varphi$-twisted conjugate if there exists $z \in \Gamma$ such that $$ x = z y \varphi(z)^{-1}.$$ This is an equivalence relation generalizing the classical conjugacy relation, which is the relation corresponding to $\varphi = \text{id}_\Gamma$. The number of equivalence classes (possibly infinite) is denoted as $R(\varphi) \in \mathbb{N} \cup \left\{\infty \right\}$ and is called the Reidemeister number of $\varphi$. The Reidemeister number has important relations to topological fixed point theory, see \cite{felshtyn}, and hence, there is a lot of literature discussing methods to compute it and to determine whether it is infinite or not.  

Recall the following lemma about the Reidemeister number, see \cite[Lemma 1.1.]{dacibergwong}.

\begin{lemma}\label{daciberg_wong_lemma1.1}
Let $\varphi \colon \Gamma \to \Gamma$ be an endomorphism of a group $\Gamma$ with $\varphi$-invariant normal subgroup $N$ and quotient group $G = \Gamma / N$. Consider the following commutative diagram
\[\begin{tikzcd}
	1 & N & \Gamma & G & 1 \\
	1 & N & \Gamma & G & 1
	\arrow[from=1-1, to=1-2]
	\arrow[from=1-2, to=1-3]
	\arrow["{\varphi_N}"', from=1-2, to=2-2]
	\arrow[from=1-3, to=1-4]
	\arrow["\varphi"', from=1-3, to=2-3]
	\arrow[from=1-4, to=1-5]
	\arrow[from=2-1, to=2-2]
	\arrow[from=2-2, to=2-3]
	\arrow[from=2-3, to=2-4]
	\arrow["{\bar{\varphi}}", from=1-4, to=2-4]
	\arrow[from=2-4, to=2-5]
\end{tikzcd}\]
where $\bar{\varphi}$ and $\varphi_N$ are the induced and restricted endomorphism, respectively. The following statements about the Reidemeister number hold.
\begin{enumerate}[(1)]
\item If $R(\varphi) < \infty$, then $R(\bar{\varphi}) < \infty$ as well.
\item If $G$ is a finite group, then $R(\varphi) < \infty$ implies $R(\varphi_N) < \infty$.
\end{enumerate}
\end{lemma}

As $\varphi^n: \Gamma \to \Gamma$ is again an endomorphism of $\Gamma$, we can also consider the Reidemeister numbers $R(\varphi^n)$. Moreover, if all these numbers are finite, then we call $\varphi$ tame, in which case we can construct the Reidemeister zeta function $$R_\varphi(z) = \exp\left(\sum_{n=1}^\infty \frac{R(\varphi^n) z^n}{n}\right).$$ In order to study whether $R_\varphi(z)$ is rational, the paper \cite{dere1} shows that it often reduces the the class of groups that admit a tame monomorphism. For example, it is known that in the class of virtually polycyclic groups only the virtually nilpotent ones can admit a tame monomorphism. We generalize this result to the class of finitely generated virtually torsion-free FAR-groups.

\begin{thm}
Let $\Gamma$ be a residually finite and finitely generated virtually FAR-group with a tame monomorphism $\varphi: \Gamma \to \Gamma$. Then the group $\Gamma$ is virtually nilpotent. In particular, the Reidemeister zeta function $R_\varphi(z)$ of $\varphi$ is rational if the group $\Gamma$ is in addition torsion free.
	\end{thm}
	
\begin{proof}
By assumption, we know that $\Fr(\Gamma)$ is trivial or thus $\tau(\Gamma)$ is finite. Just as before, we know that $\varphi(\tau(\Gamma)) = \tau(\Gamma)$ and $\varphi$ induces a tame automorphism on the quotient $\Gamma / \tau(\Gamma)$ by Lemma \ref{daciberg_wong_lemma1.1}. Hence we can assume that $\tau(\Gamma)$ is trivial. So in particular, we know that $\Gamma$ has a $\Q$-algebraic hull. Similarly we can restrict to a normal subgroup of finite index such that $\Gamma / \Fitt(\Gamma)$ is torsion-free abelian.

We proceed by contradiction, and suppose that the group $\Gamma$ is not virtually nilpotent, and in particular $\Gamma / \Fitt(\Gamma)$ is non-trivial. We note that $\Fitt(\Gamma)$ is injectively characteristic since $\Fitt(\Gamma) = \textbf{U}(\textbf{H}_\Gamma) \cap \Gamma$. Now take a power $\varphi^n$ such that $\varphi^n$ induces the identity on $\Gamma / \Fitt(\Gamma)$ as before. So the twisted conjugacy classes of $\bar{\varphi}^n$ are just conjugacy classes and thus $R(\bar{\varphi}^n) = \infty$. Again by \cite[Lemma 1.1.]{dacibergwong} we conclude that $R(\varphi^n) = \infty$ and thus that $\varphi$ is not tame.  

As rationality is known for torsion-free virtually nilpotent groups, we conclude that our theorem holds.
\end{proof}

Similarly as in \cite{dere2}, one can also show that if an endomorphism $\varphi: \Gamma \to \Gamma$ is tame, it must be eventually virtually nilpotent. 
\section*{Acknowledgements}

We are grateful to Karl Lorensen for pointing out an error in a previous version, but also giving a reference to fix it and shorten the paper.

\bibliographystyle{plain}
\bibliography{ref}
\end{document}